\documentclass{amsart}
\usepackage[all]{xy}

\newcommand{\w}{\widetilde}

\newtheorem{theorem}{Theorem}
\newtheorem{lemma}{Lemma}
\newtheorem{proposition}{Proposition}
\newtheorem{corollary}{Corollary}

\newtheorem{definition}{Definition}

\theoremstyle{remark}
\newtheorem{remark}{Remark}[section]
\numberwithin{equation}{section}

\newcommand{\C}{{\mathbb C}}       
\newcommand{\R}{{\mathbb R}}       
\newcommand{\Z}{{\mathbb Z}}       
\newcommand{\D}{{\mathbb D}} 
\newcommand{\N}{{\mathbb{N}}}

\newcommand{\internalcomment}[1]{}

\begin{document}

\title[A bound for Lagrange polynomials of Leja points and applications]
{A uniform bound for the Lagrange polynomials of Leja points for the unit disk and applications in multivariate Lagrange interpolation}





\begin{abstract}

We study uniform estimates for the family of fundamental Lagrange
polynomials associated with any Leja sequence for the complex unit disk.
The main result claims that all these polynomials are
uniformly bounded on the disk, i.e.
independently on the range $N$ of the associated
$N$-Leja section. We also give an improved estimate for special values of $N$.
As a first application, we get 
an analogous estimate 
for any compact subset whose boundary is an Alper-smooth Jordan curve.

We then deal with some applications in multivariate Lagrange interpolation.
We first give some estimates of the Lebesgue constant for
the intertwining sequence of Leja sequences, a result that also 
requires an explicit formula for the associated fundamental Lagrange polynomials
with uniform estimates. We finish by dealing with a method to get
explicit bidimensional Leja sequences and then give a positive example
for the bidisk.

\end{abstract}

\maketitle

\tableofcontents

\section{Introduction}

\subsection{Definition of Leja points for the unit disk}

In this paper we deal with the estimates of the Lagrange polynomials
of Leja points for the unit disk. We denote the complex unit disk by
\begin{eqnarray*}
\D & = &
\left\{z\in\C,\,|z|<1\right\}
\,,
\end{eqnarray*}
and $\overline{\D}$ is the closed one.

Next, for $N\geq1$, for $\eta_1,\ldots,\eta_N$ different complex numbers and $z\in\C$,
we consider the fundamental Lagrange interpolation polynomial ({\em FLIP}) associated with
$\eta_k$,
\begin{eqnarray}\label{lagsec}
l_k^{(N)}(z) & = &
\prod_{j=1,j\neq k}^N
\frac{z-\eta_j}{\eta_k-\eta_j}
\,,\;\;\;
k=1,\ldots,N
\,.
\end{eqnarray}

\bigskip

The problem of finding {\em good} sets $\left\{\eta_k\right\}_{k\geq1}$ for
Lagrange interpolation
(i.e. for which we can have some control of the associated FLIPs)
is a domain of big interest. 
One of them is called a Leja sequence and will be considered
in the whole paper.

\begin{definition}\label{defleja}

A {\em Leja sequence $\mathcal{L}$} for a compact set $K\subset\C$ is a sequence\\
$\left(\eta_1,\eta_2,\ldots,\eta_k,\ldots\right)$ 
that satisfies the following property: 
\begin{eqnarray}
\eta_1
& \in &
\partial K
\,,
\end{eqnarray}
and for all
$k\geq2$,
\begin{eqnarray}\label{defleja1}
\sup_{z\in K}
\left[
\prod_{j=1}^{k-1}
\left|
z-\eta_j
\right|
\right]
& = &
\prod_{j=1}^{k-1}
\left|
\eta_k-\eta_j
\right|
\,.
\end{eqnarray}

For all $N\geq1$, the
{\em $N$-Leja section $\mathcal{L}_N$} of a Leja sequence
$\mathcal{L}$ is the finite sequence given by the first
$N$ points of $\mathcal{L}$.

\end{definition}

These sequences took their name from F. Leja (see~\cite{leja1}) but they were first
considered by A. Edrei (see \cite{edrei1}, p. 78).
Of course, these sequences are not necessarily unique.
Even if we fix the first $k$ points $\eta_j$, the choice for 
$\eta_{k+1}$ can be multiple in general.
On the other hand, it follows by the maximum principle that
all the $\eta_j$'s lie on the boundary $\partial K$.

In the following, we will essentially deal with the special case
$K=\overline{\D}$ for which all the Leja sequences are explicit.
One can find in~\cite{bialascalvi} their complete description with proof: if we fix
$\eta_1=1$, we have for all $k\geq2$,
\begin{eqnarray}\label{explicitleja}
\eta_k
\;=\;
\exp
\left(
i\pi\sum_{l=0}^sj_l2^{-l}
\right)
& \mbox{ where } &
k-1
\;=\;
\sum_{l=0}^sj_l2^l
\,,\;
j_l\in\{0,1\}
\,.
\end{eqnarray}
In particular, the first $2^s$ points form a complete set of roots of unity of
degree $2^s$.

Finally, these Leja sequences can be seen
as an approximation to
one-dimensional Fekete sets (see~\cite{fekete1}):
a $N$-{\em Fekete set} for the compact subset
$K$ is a set of $N$ elements
$\zeta_1,\ldots,\zeta_N\in K$ which
maximize (in modulus) the Vandermonde determinant, i.e.
\begin{eqnarray}\label{defvdmfekete}
\left|
VDM\left(\zeta_1,\ldots,\zeta_N\right)
\right|
& = &
\sup_{z_1,\ldots,z_N\in K}
\;
\left|
VDM\left(z_1,\ldots,z_N\right)
\right|
\\\nonumber
& = &
\sup_{z_1,\ldots,z_N\in K}
\;
\prod_{1\leq i<j\leq N}\left|z_j-z_i\right|
\,.
\end{eqnarray}
One of the essential
differences is that determining Fekete sets is an
$N$-dimensional (with respect to $\C$) optimization problem while determining
Leja sequences is just a $1$-dimensional one. In addition, it follows from 
Definition~\ref{defleja} that the construction of 
Leja sequences is inductive (unlike any $N$-Fekete set
which requires a new research of an $N$-tuple
$\left(\zeta_1,\ldots,\zeta_N\right)$ for
every $N\geq1$).

\subsection{Main result and a first application}

The main result that we will prove is the following.

\begin{theorem}\label{unifestim}

Let $\mathcal{L}=\left(\eta_1,\eta_2,\ldots\right)$
be a Leja sequence for the unit disk. Then
the FLIPs $l_k^{(N)}$ are uniformly
bounded with respect to $N\geq1$ and
$k=1,\ldots,N$, i.e.
\begin{eqnarray*}\label{estimate}
\sup_{N\geq k\geq1}
\left(
\sup_{z\in\overline{\D}}
\left|
l_k^{(N)}(z)\right|
\right)
& \leq &
\pi\exp(3\pi)
\,.
\end{eqnarray*}

\end{theorem}

\bigskip

Some remarks are in order.
First, the explicit bound may not be optimal: indeed, this can  be seen
along the whole proof; on the other hand, an improvement for the bound will be given 
in a special case for $N=2^p-1$ in Section~\ref{secspecialN} (Theorem~\ref{specialN}).
We prove that, except for an asymptotically negligible number of values for $k=1,\ldots,N$,
the FLIPs are asymptotically bounded by $1$. 

Next, an important interpretation of this result is that any $N$-Leja section for the disk
has essentially the same property as any $N$-Fekete set. Indeed, it is known
that the FLIPs associated with
any $N$-Fekete set are always
bounded by $1$: this can be shown by noticing that
every FLIP can be written as follows for all $z\in\C$,
\begin{eqnarray}\label{defflip1}
l_{k}^{(N)}(z)
& = &
\frac{
VDM\left(\zeta_1,\ldots,\zeta_{k-1},\,z\,,\zeta_{k+1},\ldots,\zeta_N\right)
}{
VDM\left(\zeta_1,\ldots,\zeta_N\right)
}
\end{eqnarray}
(where we have replaced $\zeta_k$ with $z$).
It follows by~(\ref{defvdmfekete}) (and from the fact
that
$\sup_{z\in\overline{\D}}
\left|
l_{k}^{(N)}(z)
\right|
\geq
\left|l_{k}^{(N)}\left(\zeta_k\right)\right|=1$) 
that
\begin{eqnarray}\label{unifboundfekete}
\sup_{z\in\overline{\D}}
\left|
l_{k}^{(N)}(z)
\right|
\;=\;
1
& \mbox{for all} &
k=1,\ldots,N
\,.
\end{eqnarray}
Thus, the Fekete sets are essentially the best ones for Lagrange interpolation
and uniform stability of the associated FLIPs.
Nevertheless, constructing them is generally a hard task.
Therefore, a natural question is if there exist {\em simpler} sets
with the same property. Theorem~\ref{unifestim}
gives an affirmative answer with the Leja sequences for the unit
disk (with a slightly bigger bound but still universal).

Finally, as another application we can immediately deduce an estimate analogous 
to~\cite[formula~(1.15)]{chkifa1} for the Lebesgue constant 
$\Lambda_{N}$, $N\geq1$,  of the $N$-section from any Leja sequence
for the unit disk. We remind that the Lebesgue constant is defined for $N\geq1$ by
\begin{eqnarray}\label{lebesgue}
\Lambda_{N}
& = &
\sup_{z\in\overline{\D}}
\left(
\sum_{k=1}^N
\left|
l_k^{(N)}(z)
\right|
\right)
\,.
\end{eqnarray}

\begin{corollary}\label{appl2lebesgue}

For all $N\geq1$, 
\begin{eqnarray}\label{estimlebesgue}
\Lambda_{N}
& \leq &
\pi\exp\left(3\pi\right)
\times N
\;=\;
O(N)
\,.
\end{eqnarray}

\end{corollary}

The first estimate of $\Lambda_N$ was obtained in~\cite[Corollary~7]{calviphung1}
where it was proved that $\Lambda_N=O(N\ln N)$. Afterwards, this estimate was improved 
in~\cite[formula~(1.15)]{chkifa1} where it was proved that 
$\Lambda_N\leq2N$.
We cannot {\em a priori} hope any improvement of the bound for $O(N)$ 
by this way since we crudely estimate
\begin{eqnarray}\label{estimviolent}
\Lambda_{N}
\;=\;
\sup_{z\in\overline{\D}}
\left(
\sum_{k=1}^N
\left|
l_k^{(N)}(z)
\right|
\right)
& \leq &
\sum_{k=1}^N
\,
\sup_{z\in\overline{\D}}
\left|
l_k^{(N)}(z)
\right|
\,.
\end{eqnarray}

\subsection{On the compact sets with Alper-smooth Jordan boundary}

Another application of Theorem~\ref{unifestim} is an estimate of
the FLIPs for compact sets whose boundary is an Alper-smooth Jordan curve.
It is a special class of compact sets: for example, twice continuously 
differentiable Jordan curves are Alper-smooth. 
We denote by $\Phi$ the conformal mapping from
$\overline{\C}\setminus\D$ onto $\overline{\C}\setminus K$.

If $\mathcal{L}$ is a Leja sequence for the unit disk, the image 
$\Phi(\mathcal{L})=\left(\Phi\left(\eta_j\right)\right)_{j\geq1}$
(that is well-defined since the $\eta_j$'s belong to the unit circle)
will not necessarily be a Leja sequence for $K$. Nevertheless, we can give good estimates
as specified by the following result.

\begin{theorem}\label{unifestimcompact}

Let $\mathcal{L}=\left(\eta_j\right)_{j\geq1}$ be a Leja sequence for the unit disk with
$\left|\eta_1\right|=1$, and let $\Phi(\mathcal{L})$
be its image by 
the conformal mapping $\Phi$. Then for all $N\geq1$, we have
\begin{eqnarray*}
\max_{1\leq p\leq N}
\left[
\sup_{z\in K}
\left|
\prod_{j=1,j\neq p}^N
\dfrac{z-\Phi(\eta_j)}{\Phi(\eta_p)-\Phi(\eta_j)}
\right|
\right]
& = &
O\left(N^{2A/\ln(2)}\right)
\,,
\end{eqnarray*}
where $A$ is a positive constant depending only on $K$.

\end{theorem}

As an immediate consequence, we get an estimate of the Lebesgue constant for the compact set $K$.

\begin{corollary}\label{appl2lebesguecompact}

Let us consider $\Phi\left(\mathcal{L}\right)$ where $\mathcal{L}$ is a Leja sequence for the unit disk
and $\Phi$ is the conformal mapping for $K$. We have the following estimate:
\begin{eqnarray}\label{estimlebesguecompact}
\Lambda_{N}(K)
& = &
O\left(N^{1+2A/\ln(2)}\right)
\,,
\end{eqnarray}
where $A$ is a positive constant depending only on $K$.

\end{corollary}

\begin{remark}\label{cteA}

As noticed by J. Ortega-Cerd\`a, one can make explicit the constant $A$
(see~\cite[Lemmas~2, 3 and their proofs]{bialascalvi}):
\begin{eqnarray}
A
& = &
\sup_{|w|=1}
\int_0^{2\pi}
\left|
\dfrac{\Phi'\left(e^{it}\right)}{\Phi\left(e^{it}\right)-\Phi(w)}
-
\dfrac{1}{e^{it}-w}
\right|
dt
\,.
\end{eqnarray}

\end{remark}

\medskip

\subsection{Applications in multivariate Lagrange interpolation}\label{applmultivarintro}

Now we deal with the essential applications of Theorem~\ref{unifestim} in multivariate
Lagrange interpolation. In this subsection, we will consider the bidimensional case
(although the following results may be generalized to $\C^d$).
We first remind the lexicographic order defined on $\N^2$:
we say that $\left(k_1,l_1\right)\leq\left(k_2,l_2\right)$ iff
$k_1+l_1<k_2+l_2$, or 
$k_1+l_1=k_2+l_2$ and $k_1\geq k_2$. This definition gives a numeration of
$\N^2$ as follows: $(0,0),\,(1,0),\,(0,1),\,(2,0),(1,1),\ldots$
This allows us to define an order on the standard basis monomials by
\begin{eqnarray}\label{deforderpol}
\left\{e_j(z,w)\;=\;z^{k(j)}w^{l(j)}\,,\;j\geq1\right\}
\,.
\end{eqnarray}
Similarly, we can define an order for the product of two complex sequences
$\left(\eta_k\right)_{k\geq0}$ and $\left(\theta_l\right)_{l\geq0}$, by
\begin{eqnarray}\label{defintertwin}
\left\{H_j\;=\;\left(\eta_{k(j)},\theta_{l(j)}\right)\,,\;j\geq1\right\}
\,.
\end{eqnarray}

Next, we denote for all $n\geq0$, the space
$\mathcal{P}_n\subset\C[z,w]$ of complex polynomials of degree at most $n$, and
\begin{eqnarray}\label{Nn}
N_n
\;=\;
\dim\mathcal{P}_n
\;=\;
\dfrac{(n+1)(n+2)}{2}
\,.
\end{eqnarray}
Notice that the set
$\left\{H_1,H_2,\ldots,H_{N_n}\right\}=\left\{\left(\eta_k,\theta_l\right),\,k+l\leq n\right\}$ 
is the so-called intertwining array of
the arrays
$\left\{\eta_0,\eta_1\ldots,\eta_n\right\}$
and
$\left\{\theta_0,\theta_1,\ldots,\theta_n\right\}$
(e.g., see~\cite{calvi2005} for a general approach).
Here we need to extend this definition and consider
the arrays $\left\{H_1,H_2,\ldots,H_{N}\right\}$ for any value of
$N\geq1$ (and not only for special values as $N_n$).

For all $N\geq1$, let us consider the unique integers $n\geq0$ and $m$ with $0\leq m\leq n$, such that
\begin{eqnarray}\label{defnm}
N_{n-1}
\;<\;
N
\;\leq\;
N_n
& \mbox{ and } &
m
\;=\;
N-N_{n-1}-1
\,.
\end{eqnarray}
We can then consider the space
\begin{eqnarray}\label{defPnm}
\mathcal{P}_{n,m}
& = &
\mbox{span}\left\{e_j\,,\;1\leq j\leq N\right\}
\,.
\end{eqnarray}
One has $\mathcal{P}_{n-1}\subset\mathcal{P}_{n,m}\subset\mathcal{P}_n$
(with the convention that
$\mathcal{P}_{-1}=\{0\}$), and $\mathcal{P}_{n,m}$ is spanned by $\mathcal{P}_{n-1}$
and the monomials $z^n,\,z^{n-1}w,\ldots,z^{n-m}w^m$.
In particular, $\mathcal{P}_{n,n}=\mathcal{P}_n$.
On the other hand,
by~(\ref{defnm}), $\dim\mathcal{P}_{n,m}=N_{n-1}+m+1=N$.
Similarly, we consider the set 
\begin{eqnarray}\label{defOmegaN}
\;\;\;\;\;\;\;\;\;\;\;\;\;\;\;\;\;
\Omega_N
\;:=\;
\Omega_{n,m}
& = &
\left\{H_j\,,\;1\leq j\leq N\right\}
\\\nonumber
& = &
\left\{
\left(\eta_0,\theta_0\right),\ldots,
\left(\eta_0,\theta_{n-1}\right),
\left(\eta_n,\theta_0\right),
\ldots,\left(\eta_{n-m},\theta_m\right)
\right\}
\,.
\end{eqnarray}

A first property is that for all $N\geq1$, the set
$\Omega_N$ is unisolvent for the space $\mathcal{P}_{n,m}$
(Section~\ref{applmultivar}, Corollary~\ref{unisolvent}): for every function $f$ defined on
$\Omega_N$, there exists a unique polynomial $P\in\mathcal{P}_{n,m}$
such that $P(z,w)=f(z,w)$
for every $(z,w)\in\Omega_N$. 
This extends in the special case of $\Omega_N\subset\C^2$ the classical idea of 
Biermann in~\cite{bierman1903}
(see also~\cite{calvi2005} for a general version with
{\em block unisolvent arrays}).
In addition, $P$ is the
Lagrange interpolation polynomial $L^{(N)}[f]$ of $f$ with respect to $\Omega_N$,
and defined below.
In particular, the generalized Vandermonde determinant of $\Omega_N$ is nonzero, i.e.
\begin{eqnarray*}
VDM\left(H_1,H_2,\ldots,H_N\right)
\;=\;
\det\left[e_i\left(H_j\right)\right]_{1\leq i,j\leq N}
\;\neq\;
0
\,.
\end{eqnarray*}
This allows us to introduce the (generalized) fundamental Lagrange polynomials (FLIPs)
by extending the expression in~(\ref{defflip1}) as
\begin{eqnarray}\label{defflipd}
\;\;\;\;\;
& &
l_{H_p}^{(N)}(z,w)
\;=\;
\dfrac{VDM\left(H_1,H_2,\ldots,H_{p-1},(z,w),H_{p+1},\ldots,H_N\right)
}{
VDM\left(H_1,H_2,\ldots,H_N\right)}
\,,\;
1\leq p\leq N.
\end{eqnarray}
We can then define the multivarate Lagrange polynomial of any function
$f$ defined on $\Omega_N$, by
\begin{eqnarray}\label{deflagpold}
L^{(N)}[f](z,w)
& = &
\sum_{p=1}^N
f\left(H_p\right)
\,
l_{H_p}^{(N)}(z,w)
\,,\;
\forall\,(z,w)\in\C^2
\,.
\end{eqnarray}
As it has been pointed out (e.g., see~\cite[p. 54]{bloomboscalvilev2012}), there is no cancellation 
in~(\ref{defflipd})
so that the formulas cannot be simplified as in~(\ref{lagsec}). Nevertheless, Proposition~\ref{explicitlagrangepol}
in Section~\ref{applmultivar} gives an explicit formula of $l_{H_p}^{(N)}$, $\forall\,p=1,\ldots,N$.

Finally, we similarly define the Lebesgue constant of the sequence
$\left(H_j\right)_{j\geq1}$ (with respect to a compact set
$K\supset\left(H_j\right)_{j\geq1}$):
\begin{eqnarray}
\Lambda_N\left(K,\left(H_j\right)_{j\geq1}\right)
& = &
\sup_{(z,w)\in K}
\left[
\sum_{p=1}^N
\left|
l_{H_p}^{(N)}(z,w)
\right|
\right]
\,.
\end{eqnarray}

\medskip

We then have the following result that gives an estimate of the Lebesgue constant
for the bidisk and the product of compacts whose boundary is an Alper-smooth
Jordan curve.

\begin{theorem}\label{lebcteintertwin}

Let $\left(\eta_j\right)_{j\geq0}$ and $\left(\theta_i\right)_{i\geq0}$
be Leja sequences for the unit disk (with $\left|\eta_0\right|=\left|\theta_0\right|=1$),
and let us consider the intertwining sequence $\left(H_j\right)_{j\geq1}$ defined
as in~(\ref{defintertwin}). One has the following estimate:
\begin{eqnarray}\label{lebcteintertwinbidisc}
\Lambda_N\left(\overline{\D}^2,\left(H_j\right)_{j\geq1}\right)
& = &
O\left(N^{3/2}\right)
\,.
\end{eqnarray}

If $K_1$ (resp., $K_2$) is a compact set whose boundary is an Alper-smooth Jordan curve,
$\Phi_1$ (resp., $\Phi_2$) its associated conformal mapping, let us consider
the sequence $\left(\Phi\left(H_j\right)\right)_ {j\geq1}$ defined by the intertwining of the sequences
$\left(\Phi_1\left(\eta_j\right)\right)_{j\geq0}$
and
$\left(\Phi_2\left(\theta_i\right)\right)_{i\geq0}$.
Then
\begin{eqnarray}\label{lebcteintertwincompact}
\Lambda_N\left(K_1\times K_2,\left(\Phi\left(H_j\right)\right)_ {j\geq1}\right)
& = &
O\left(N^{2A/\ln(2)+3/2}\right)
\,,
\end{eqnarray}
where $A$ is a positive constant depending only on $K_1\times K_2$.

\end{theorem}

The essential part for the proof of this theorem is the
following result that, as Theorem~\ref{unifestim}, gives a uniform estimate
for the bidimensional
fundamental Lagrange polynomials (FLIPs) associated with some intertwining 
sequences. It is in addition an
application of Proposition~\ref{explicitlagrangepol} whose proof is handled in
Section~\ref{applmultivar}, and that gives an explicit formula of
the FLIPs constructed from the intertwining sequence of any sequences
of pairwise distinct elements.

\begin{proposition}\label{unifestimintertwining}

Let $\left(\eta_j\right)_{j\geq0}$ and $\left(\theta_i\right)_{i\geq0}$
be Leja sequences for the unit disk (with $\left|\eta_0\right|=\left|\theta_0\right|=1$),
and let us consider the intertwining sequence $\left(H_j\right)_{j\geq1}$ defined
as in~(\ref{defintertwin}). 
For all $N\geq1$ and all $p,\,q\geq0$ such that
$\left(\eta_p,\theta_q\right)\in\Omega_{N}$, one has the following estimate:
\begin{eqnarray*}
\sup_{(z,w)\in\overline{\D}^2}
\left|
l_{(\eta_p,\theta_q)}^{(N)}(z,w)
\right|
\;\leq\;
2\left(n-p-q+1\right)
\pi^2\exp(6\pi)
\;=\;
O\left(N^{1/2}\right)
\,.
\end{eqnarray*}

If $K_1$ (resp., $K_2$) is a compact set whose boundary is an Alper-smooth Jordan curve,
$\Phi_1$ (resp., $\Phi_2$) its associated conformal mapping, and
$\left(\Phi\left(H_j\right)\right)_ {j\geq1}$ the intertwining sequence of
$\left(\Phi_1\left(\eta_j\right)\right)_{j\geq0}$
and
$\left(\Phi_2\left(\theta_i\right)\right)_{i\geq0}$, one has for 
all $N\geq1$ and all $p,\,q\geq0$ such that
$\left(\eta_p,\theta_q\right)\in\Omega_{N}$,
\begin{eqnarray*}
\sup_{(z,w)\in K_1\times K_2}
\left|
l_{(\Phi_1(\eta_p),\Phi_2(\theta_q))}^{(N)}(z,w)
\right|
& \leq &
M(n+2)^{4A/\ln(2)}
\left(n-p-q+1\right)
\\
& = &
O\left(N^{2A/\ln(2)+1/2}\right)
\,,
\end{eqnarray*}
where $M$ and $A$ are positive constants depending only on $K_1\times K_2$.

\end{proposition}

Although there are results which give polynomial estimates for $\C^d$
in~\cite{calviphung1}, Theorem~\ref{lebcteintertwin} yields
here precise exponents for $N$. In particular, it gives examples of
sequences that satisfy
$\lim_{n\rightarrow+\infty}\left[\Lambda_{N_n}\left(K,\left(H_j\right)_{j\geq1}\right)\right]^{1/n}=1$,
where $N_n=(n+1)(n+2)/2$ and $K$ is any of the compact sets mentioned in
Theorem~\ref{lebcteintertwin}.

In addition, there is an 
improvement of some results from~\cite{calviphung1} for $d=2$: indeed, on the one hand it is proved
in~\cite[p.~621]{calviphung1} that
\begin{eqnarray*}
\Lambda_{N_n}\left(\overline{\D}^2,\left(H_j\right)_{j\geq1}\right)
\;=\;
O\left(
n^{(2^2+7\times2-6)/2}(\ln n)^2
\right)
\;=\;
O\left(n^{6}(\ln n)^2\right)
\,;
\end{eqnarray*}
on the other hand, 
an application of~(\ref{lebcteintertwinbidisc}) with $N=N_n=O\left(n^2\right)$ (by~(\ref{Nn})), gives that
\begin{eqnarray*}
\Lambda_{N_n}\left(\overline{\D}^2,\left(H_j\right)_{j\geq1}\right)
\;=\;
O\left(N_n^{3/2}\right)
\;=\;
O\left(n^3\right)
\,.
\end{eqnarray*}

Next, the compact set $K=\overline{\D}^2$
(resp., $K=K_1\times K_2$, where $K_1$ and $K_2$ are compact sets 
whose boundary is an Alper-smooth Jordan curve) satisfies the following properties:
first, it is nonpluripolar (for having nonempty interior); next, it is polynomially convex (as
the product of $1$-connected one-dimensional compact sets).
Lastly, it is $L$-regular (see~\cite[Chapter~5]{klimek} for the definition): indeed, $\C\setminus K_1$ 
(resp., $\C\setminus K_2$) being a domain whose boundary is
a Jordan curve, it admits a continuous Green function (by Caratheodory's theorem),
hence $K_1$ (resp., $K_2$) is $L$-regular (as well as $\overline{\D}$).
Finally, we can deduce by~\cite[Corollary~2]{siciak1962} that
$\overline{\D}^2$ and $K_1\times K_2$ are $L$-regular as products of
$L$-regular compact sets.

These properties and an application of Theorem~4.1 from~\cite{bloomboschrislev1992}, lead to 
the following consequences:

\begin{itemize}

\item

$\lim_{n\rightarrow+\infty}\left|VDM\left(H_1,H_2,\ldots,H_{N_n}\right)\right|^{1/l_n}=D(K)$,
where\\ $l_n=\sum_ {j=1}^nj\left(N_j-N_{j-1}\right)$ and 
$D(K)$ is the transfinite diameter of $K$ (see~\cite{zaharjuta1975} for its general definition 
and existence);

\item

by~\cite{bermanboucksomwitt2011}, the empirical measures satisfy
$\lim_{n\rightarrow+\infty}\dfrac{1}{N_n}\sum_{j=1}^{N_n}\delta_{H_j}=c\mu_K$ \\weak-*,
where $\delta_{H_j}$ is the unit Dirac measure on $H_j$, $c$ is some universal constant and
$\mu_K$ is the equilibrium measure of $K$ (for the definition and existence of $\mu_K$,
see~\cite[Chapter~1]{safftotik} for $K\subset\C$, and~\cite[Chapter~5]{klimek} for $K\subset\C^d$);

\item

the Lagrange polynomial $L^{(N_n)}[f]$ converges to $f$ as $n\rightarrow+\infty$, 
uniformly on $K$ and for each $f$ holomorphic on a
neighborhood of $K$.

\end{itemize}

Another consequence of the above theorem is an application 
to an approximation property of Lagrange polynomial interpolation.
Theorem~7.2
from~\cite{renwang} provides a generalized version of Jackson's Theorem
(see~\cite{jackson}).
Given $m\in\N$ and $\gamma$ with $0<\gamma\leq1$, we consider the space
$Lip_{\gamma}^m\left(\overline{\D}^2\right)$ of the functions 
$f\in\mathcal {O}\left(\D^2\right)\bigcap C\left(\overline{\D}^2\right)$
which satisfy for all $|\alpha|=\alpha_1+\alpha_2\leq m$,
\begin{eqnarray*}
\left(
\dfrac{\partial^{\alpha}f}{\partial z^{\alpha}}
\right)
\left(e^{ih}\zeta\right)
-
\left(
\dfrac{\partial^{\alpha}f}{\partial z^{\alpha}}
\right)
(\zeta)
\;=\;
O\left(|h|^{\gamma}\right)
\,,
\end{eqnarray*}
for all
$\zeta\in\mathbb{T}^2$ (the unit torus in $\C^2$) and $h\in\R$.
We then have the following result.

\begin{corollary}\label{appljackson}

Let be $m\geq3$ and $0<\gamma\leq1$. 
Let us consider the intertwining sequence $\left(H_j\right)_{j\geq1}$
of two any Leja sequences for the unit disk as in Theorem~\ref{lebcteintertwin}.
Then for all $f\in Lip_{\gamma}^m\left(\overline{\D}^2\right)$,
\begin{eqnarray*}
\sup_{(z,w)\in\overline{\D}^2}
\left|f(z,w)-L^{(N)}[f](z,w)\right|
& = &
O\left(1/N^{(m+\gamma-3)/2}\right)
\,,
\end{eqnarray*}
where $L^{(N)}[f]$ is the multivariate Lagrange polynomial of $f$ defined in~(\ref{deflagpold}).

\end{corollary}

\medskip

As in the one-dimensional case~(\ref{defleja1}), we will consider the generalized definition of 
a Leja sequence in $\C^2$.
By analogy, given a compact set $K$, we can choose $H_1\in K$ (or $H_1\in\partial K$)
and, if we have constructed
$H_1,\ldots,H_{N}$, we choose for $H_{N+1}$ any point (not necessarily unique) that satisfies
\begin{eqnarray*}
\left|
VDM\left(H_1,H_2,\ldots,H_N,H_{N+1}\right)
\right|
& = &
\sup_{(z,w)\in K}
\left|
VDM\left(H_1,H_2,\ldots,H_N,(z,w)\right)
\right|
\,.
\end{eqnarray*}

As for the Fekete points, it is hard to construct examples of explicit 
bidimensional Leja sequences (even in one dimension, except for special compact sets).
Here we present a way to construct explicit bidimensional Leja sequences for the product of compact sets
(provided we already know their partial one-dimensional Leja sequences).

\begin{theorem}\label{explicitlejasequ}

Let $K_1\subset\C$ (resp., $K_2\subset\C$) be any compact set and
$\left(\eta_j\right)_{j\geq0}\subset K_1$ 
(resp., $\left(\theta_i\right)_{i\geq0}\subset K_2$) any Leja sequence.
Then the
intertwining sequence of $\left(\eta_j\right)_{j\geq0}$ and
$\left(\theta_i\right)_{i\geq0}$ is a Leja sequence for the compact set
$K_1\times K_2$.

\end{theorem}

As an immediate application, we can present some explicit bidimensional Leja sequences for the
unit bidisk.

\begin{corollary}\label{explicitbileja}

Let $\left(\eta_j\right)_{j\geq0}$ and $\left(\theta_i\right)_{i\geq0}$ be
defined as in~(\ref{explicitleja}). Then their intertwining sequence is a
Leja sequence for the unit bidisk $\overline{\D}^2$.

\end{corollary}

In particular, this result gives a partial answer for question~(6) in~\cite{bloomboscalvilev2012}.

Theorem~\ref{explicitlejasequ} is an application of Lemma~\ref{explicitVDM}
and another consequence is the formula of Schiffer and Siciak
for the two-variable Vandermonde determinant
(see~\cite{schiffersiciak1962}, or formula~(4.8.2) from~\cite{bloomboschrislev1992}).

\begin{corollary}\label{redemschifsic}

For all $n\geq0$,
\begin{eqnarray*}
VDM
\left(
\left(\eta_0,\theta_0\right),\ldots,
\left(\eta_n,\theta_0\right),\ldots,\left(\eta_0,\theta_n\right)
\right)
\,=\,
\prod_{j=1}^n
\left[
VDM\left(\eta_0,\ldots,\eta_j\right)
\times
VDM\left(\theta_0,\ldots,\theta_j\right)
\right]
.
\end{eqnarray*}

\end{corollary}

\bigskip

\section*{Acknowledgments}

I would like to thank J. Ortega-Cerd\`a for having introduced me this nice problem,
for all the rewarding ideas about it and for his support for finalizing this work.
I would also like to thank J.-P.~Calvi and M.A. Chkifa for the interesting discussions about it.

\bigskip

\section{A couple of reminders and preliminar results}

\subsection{Reminders}

First, we recall some propertis and applications of the roots of the unity:
for any $m\geq1$, let
\begin{eqnarray}\label{rootunity}
\Omega_m &= &
\left\{\exp\left(\frac{2iu\pi}{m}\right),\,u=0,\ldots,m-1\right\}
\end{eqnarray}
be the set of the $m$-th roots of the unity. Then one has the following result.

\begin{lemma}\label{estim1}

For all $m\geq1$, $z_k\in\Omega_m$ and $z\in\overline{\D}$ with $z\neq z_k$, one has
\begin{eqnarray}\label{estim11}
\prod_{z_j\in\Omega_m,z_j\neq z_k}
(z-z_j)
& = &
\frac{z^m-1}{z-z_k}
\;=\;
\sum_{j=0}^{m-1}
z_k^{m-j-1}z^j
\,.
\end{eqnarray}
It follows that
\begin{eqnarray}\label{estim12}
\prod_{z_j\in\Omega_m,z_j\neq z_k}
\left|z_k-z_j\right|
& = &
m
\end{eqnarray}
and for all $z\neq z_k$,
\begin{eqnarray}\label{estim125}
& &
\left|
l_k^{(m)}
(z)
\right|
\;=\;
\prod_{z_j\in\Omega_m,z_j\neq z_k}
\frac{
\left|
z-z_j
\right|
}{
\left|
z_k-z_j
\right|
}
\;=\;
\frac{1}{m}
\frac{\left|z^m-1\right|}{\left|z-z_k\right|}
\;=\;
\frac{1}{m}
\left|
\sum_{j=0}^{m-1}
z_k^{m-j-1}z^j
\right|
\,.
\end{eqnarray}
In addition,
\begin{eqnarray}\label{estim13}
\sup_{z\in\overline{\D}}
\left|
l_k^{(m)}
(z)
\right|
& = &
\left|
l_k^{(m)}
(z_k)
\right|
\;=\;
1\,.
\end{eqnarray}

\end{lemma}

\begin{proof}

First, since $\Omega_m$ is the set of the $m$-th roots of the unity, one has
\begin{eqnarray*}
\prod_{z_j\in\Omega_m,z_j\neq z_k}
(z-z_j)
& = &
\frac{\prod_{z_j\in\Omega_m}
(z-z_j)}{z-z_k}
\;=\;
\frac{z^m-1}{z-z_k}
\,.
\end{eqnarray*}
On the other hand (as $z_k\in\Omega_m$),
\begin{eqnarray*}
\frac{z^m-1}{z-z_k}
& = &
\frac{z^m-z_k^m}{z-z_k}
\;=\;
\sum_{j=0}^{m-1}
z_k^{m-j-1}z^j
\,,
\end{eqnarray*}
and this proves~(\ref{estim11}).

Next, (\ref{estim12}) follows since
\begin{eqnarray*}
\prod_{z_j\in\Omega_m,z_j\neq z_k}
\left(z_k-z_j\right)
& = &
\lim_{z\rightarrow z_k,z\neq z_k}
\prod_{z_j\in\Omega_m,z_j\neq z_k}
\left(z-z_j\right)
\;=\;
\lim_{z\rightarrow z_k,z\neq z_k}
\frac{z^m-1}{z-z_k}
\\
& = &
\frac{d}{dz}
\left(z^m\right)_{|z=z_k}
\;=\;
mz_k^{m-1}
\,.
\end{eqnarray*}

Finally, (\ref{estim11}) and~(\ref{estim12}) yield~(\ref{estim125})
for all $z\neq z_k$ since
\begin{eqnarray*}
\left|
l_k^{(m)}
(z)
\right|
& = &
\frac{
\prod_{z_j\in\Omega_m,z_j\neq z_k}
\left|
z-z_j
\right|
}{
\prod_{z_j\in\Omega_m,z_j\neq z_k}
\left|
z_k-z_j
\right|
}
\;=\;
\frac{1}{m}
\prod_{z_j\in\Omega_m,z_j\neq z_k}
\left|
z-z_j
\right|
\,.
\end{eqnarray*}
In particular, for all $z\in\overline{\D}$,
\begin{eqnarray*}
\left|
l_k^{(m)}
(z)
\right|
& = &
\frac{1}{m}
\left|
\sum_{j=0}^{m-1}
z_k^{m-j-1}z^j
\right|
\;\leq\;
\frac{1}{m}
\sum_{j=0}^{m-1}
\left|z_k\right|^{m-j-1}|z|^j
\;\leq\;
1
\,.
\end{eqnarray*}
On the other hand,
\begin{eqnarray*}
\left|
l_k^{(m)}
\left(z_k\right)
\right|
& = &
\prod_{z_j\in\Omega_m,z_j\neq z_k}
\left|\frac{z_k-z_j}{z_k-z_j}\right|
\;=\;1
\end{eqnarray*}
and the last assertion of the lemma follows.

\end{proof}

In all the following, we will assume that any considered
Leja sequence $\mathcal{L}=\left(z_1,z_2,\ldots\right)$
starts at $1$, i.e.
\begin{eqnarray}
z_1
& = &
1
\,.
\end{eqnarray}

\bigskip

Next, for any integer $N\geq1$, $\mathcal{L}_N$ will mean the
$N$-section of $\mathcal{L}$, i.e. the first $N$ indexed points
\begin{eqnarray}
\mathcal{L}_N
& := &
\left(
z_1,z_2,\ldots,z_{N-1},z_N
\right)
\,.
\end{eqnarray}
We will also consider its binary decomposition
\begin{eqnarray}\label{bindec}
N & = &
2^{p_1}+2^{p_2}+\cdots+2^{p_n}
\,,
\end{eqnarray}
where 
\begin{eqnarray}\label{defnpj}
n\geq1
& \mbox{ and }&
p_1\;>\;p_2\;>\;\cdots\;>\;p_n\;\geq\;0
\end{eqnarray}
($n$ is the number of
"ones" in this decomposition).

\subsection{Preliminar results}

In this subsection, we will give some preliminar results that will be useful in order to prove 
Theorem~\ref{unifestim}.
We begin with the following lemma that is a rewriting of the FLIPs
by using the binary decomposition of $N$.

\begin{lemma}\label{lkpregeneral}

Let remind the binary decomposition~(\ref{bindec}),
\begin{eqnarray}\label{binaryN}
N & = &
2^{p_1}+2^{p_2}+\cdots+2^{p_n}
\,,
\end{eqnarray}
where $n\geq2$ and
$p_1>p_2>\cdots>p_n\geq0$. Then for all
$k=1,\ldots,2^{p_1}$ and $z\in\C$ with $z\neq z_k$,
\begin{eqnarray}
\left|
l_k^{(N)}(z)
\right|
& = &
\frac{1}{2^{p_1}}
\frac{\left|z^{2^{p_1}}-1\right|}{\left|z-z_k\right|}
\times
\prod_{q=2}^n
\frac{
\left|z^{2^{p_q}}-\alpha_q^{2^{p_q}}\right|
}{
\left|z_k^{2^{p_q}}-\alpha_q^{2^{p_q}}\right|
}
\,,
\end{eqnarray}
where for all $q=2,\ldots,n$,
\begin{eqnarray}\label{defalphaq}
\alpha_q
& = &
\rho_1\rho_2\cdots\rho_{q-1}
\end{eqnarray}
and for $j=1,2,\ldots,n$,
$\rho_j$ is a $2^{p_j}$-th root of $-1$.

One also has for all $k=1,\ldots,2^{p_1}$ and $z\in\C$,
\begin{eqnarray}
\left|
l_k^{(N)}(z)
\right|
& = &
\frac{1}{2^{p_1}}
\left|
\sum_{j=0}^{2^{p_1}-1}
z_k^{2^{p_1}-j-1}z^j
\right|
\times
\prod_{q=2}^n
\frac{
\left|z^{2^{p_q}}-\alpha_q^{2^{p_q}}\right|
}{
\left|z_k^{2^{p_q}}-\alpha_q^{2^{p_q}}\right|
}
\,.
\end{eqnarray}

\end{lemma}

\begin{proof}

First, we know 
(see Theorem~5 from~\cite{bialascalvi} or
Theorem~1 from~\cite{calviphung1}) that
\begin{eqnarray}\label{splitleja}
\mathcal{L}_N & = &
\left(
\Omega_{2^{p_1}},\rho_1\w{\mathcal{L}}_{N-2^{p_1}}
\right)
\,,
\end{eqnarray}
where $\Omega_{2^{p_1}}$ is the set of the
$2^{p_1}$-th roots of the unity,
$\rho_1$ is a $2^{p_1}$-th root of $-1$ and
$\w{\mathcal{L}}_{N-2^{p_1}}$ is (maybe another) 
$\left(N-2^{p_1}\right)$-Leja section for the disk (that also starts at $1$).
Now $N-2^{p_1}=2^{p_2}+\cdots+2^{p_n}$ then we can follow the process with
\begin{eqnarray*}
\w{\mathcal{L}}_{N-2^{p_1}} & = &
\left(
\Omega_{2^{p_2}},\rho_2\w{\w{\mathcal{L}}}_{N-2^{p_1}-2^{p_2}}
\right)
\end{eqnarray*}
and so on.
We also know by Theorem~5 from~\cite{bialascalvi} that
any $2^{p_n}$-Leja section (that starts at $1$)
will consist of the $2^{p_n}$-th roots of the unity,
$\Omega_{2^{p_n}}$ (where the equality is meant as sets).
This finally gives
\begin{eqnarray*}
\mathcal{L}_N
& = &
\left(
\Omega_{2^{p_1}},\rho_1\Omega_{2^{p_2}},
\rho_1\rho_2\Omega_{2^{p_3}},\ldots,
\rho_1\cdots\rho_{n-1}\Omega_{2^{p_n}}
\right)
\\
& = &
\left(
\Omega_{2^{p_1}},\alpha_2\Omega_{2^{p_2}},
\alpha_3\Omega_{2^{p_3}},\ldots,
\alpha_n\Omega_{2^{p_n}}
\right)
\,.
\end{eqnarray*}
In particular, $z_k\in\Omega_{2^{p_1}}$ for all
$1\leq k\leq2^{p_1}$, then 
\begin{eqnarray*}
l_k^{(N)}(z)
& = &
\left(
\prod_{z_j\in\Omega_{2^{p_1}},z_j\neq z_k}
\frac{z-z_j}{z_k-z_j}
\right)
\times
\prod_{q=2}^n
\left(
\prod_{z_j\in\alpha_q\Omega_{2^{p_q}}}
\frac{z-z_j}{z_k-z_j}
\right)
\,.
\end{eqnarray*}

On the one hand, one has by~(\ref{estim125}) from Lemma~\ref{estim1} that 
\begin{eqnarray*}
\left|
\prod_{z_j\in\Omega_{2^{p_1}},z_j\neq z_k}
\frac{z-z_j}{z_k-z_j}
\right|
& = &
\frac{1}{2^{p_1}}
\left|
\frac{z^{2^{p_1}}-1}{z-z_k}
\right|
\;=\;
\frac{1}{2^{p_1}}
\left|
\sum_{j=0}^{2^{p_1}-1}
z_k^{2^{p_1}-j-1}z^j
\right|
\,.
\end{eqnarray*}
On the other hand, for all $q=2,\ldots,n$,
\begin{eqnarray*}
\left|
\prod_{z_j\in\alpha_q\Omega_{2^{p_q}}}
\frac{z-z_j}{z_k-z_j}
\right|
& = &
\left|
\prod_{z_j\in\Omega_{2^{p_q}}}
\frac{z-\alpha_qz_j}{z_k-\alpha_qz_j}
\right|
\;=\;
\left|
\prod_{z_j\in\Omega_{2^{p_q}}}
\frac{z/\alpha_q-z_j}{z_k/\alpha_q-z_j}
\right|
\\
& = &
\left|
\frac{
(z/\alpha_q)^{2^{p_q}}-1
}{
(z_k/\alpha_q)^{2^{p_q}}-1
}
\right|
\;=\;
\frac{
\left|z^{2^{p_q}}-\alpha_q^{2^{p_q}}\right|
}{
\left|z_k^{2^{p_q}}-\alpha_q^{2^{p_q}}\right|
}
\end{eqnarray*}
(since $\left|\alpha_q\right|=1$
for all $q=2,\ldots,n$),
and the lemma is proved.

\end{proof}

As a consequence, we get the following result that will be useful
in the next section.

\begin{lemma}\label{lkgeneral}

Under the same hypothesis for $N$ as in
Lemma~\ref{lkpregeneral}, one has for all
$k=1,\ldots,2^{p_1}$ and $z\in\C$ with $z\neq z_k$,
\begin{eqnarray}\label{descrlkgen}
\left|
l_k^{(N)}(z)
\right|
& = &
\frac{1}{2^{p_1}}
\frac{\left|z^{2^{p_1}}-1\right|}{\left|z-z_k\right|}
\times
\prod_{q=2}^n
\frac{
\left|z^{2^{p_q}}+\omega_0^{2^{p_q}}\right|
}{
\left|z_k^{2^{p_q}}+\omega_0^{2^{p_q}}\right|
}
\,,
\end{eqnarray}
where
$\omega_0$ is a $2^{p_1}$-th root of $-1$. In fact,
\begin{eqnarray}\label{defomega0}
\omega_0
& = &
\alpha_{n+1}
\;=\;
\rho_1\rho_2\cdots\rho_{n-1}\rho_n
\,.
\end{eqnarray}

Similarly, one also has for all $k=1,\ldots,2^{p_1}$ and $z\in\C$
(with the same $\omega_0$),
\begin{eqnarray}\label{descrlkgenz0}
\left|
l_k^{(N)}(z)
\right|
& = &
\frac{1}{2^{p_1}}
\left|
\sum_{j=0}^{2^{p_1}-1}
z_k^{2^{p_1}-j-1}z^j
\right|
\times
\prod_{q=2}^n
\frac{
\left|z^{2^{p_q}}+\omega_0^{2^{p_q}}\right|
}{
\left|z_k^{2^{p_q}}+\omega_0^{2^{p_q}}\right|
}
\,.
\end{eqnarray}

\end{lemma}

\begin{proof}

It suffices to show that, for all $q=1,\ldots,n$, one has
$\omega_0^{2^{p_q}}=-\alpha_q^{2^{p_q}}$. This is immediate
from~(\ref{defalphaq}) since
\begin{eqnarray*}
\omega_0^{2^{p_q}}
& = &
\left(\rho_1\cdots\rho_{q-1}\right)^{2^{p_q}}
\times
\rho_q^{2^{p_q}}
\times
\left(\rho_{q+1}\cdots\rho_n\right)^{2^{p_q}}
\\
& = &
\alpha_q^{2^{p_q}}
\times
\exp\left(i\pi\right)
\times
\exp
\left[i\pi
\left(
2^{p_q-p_{q+1}}+\cdots+2^{p_q-p_{n}}
\right)
\right]
\end{eqnarray*}
and
$0<p_q-p_{q+1}<\cdots<p_q-p_n$.
Similarly,
\begin{eqnarray*}
\omega_0^{2^{p_1}}
& = &
\rho_1^{2^{p_1}}
\times
\left(\rho_{2}\cdots\rho_n\right)^{2^{p_1}}
\;=\;
-1
\,.
\end{eqnarray*}

\end{proof}

\begin{remark}\label{omega0}

$\omega_0$ is any of $2^{p_n}$ possible choices for the
$(N+1)$-th Leja point
$z_{N+1}$. Indeed,
\begin{eqnarray*}
\prod_{j=1}^N
\left|
z-z_j
\right|
& = &
\prod_{q=1}^n
\left|
z^{2^{p_q}}+\omega_0^{2^{p_q}}
\right|
\;\leq\;
2^n
\end{eqnarray*}
for all $|z|\leq1$. In addition, this bound is reached if and only if
for all $q=1,\ldots,n$, 
$\left|
z^{2^{p_q}}+\omega_0^{2^{p_q}}
\right|
=
\left|
\left(z/\omega_0\right)^{2^{p_q}}+1
\right|=2$,
then
if and only if
\begin{eqnarray*}
\left(
\frac{z}{\omega_0}
\right)^{2^{p_n}}
\;=\;
1
\,,
& \mbox{ i.e. } &
z_{N+1}
\;\in\;
\omega_0
\,
\Omega_{2^{p_n}}
\,.
\end{eqnarray*}

On the other hand, notice that the data of
$\Omega_{2^{p_1}}$ and $\omega_0$ (any fixed
$2^{p_1}$-th root of $-1$), give 
a $N$-Leja section (that starts at $z_1=1$).
Indeed, set
\begin{eqnarray*}
\rho_q
\;=\;
\exp
\left(
\frac{i\pi}{2^{p_q}}
\right)
\;\mbox{ for all }
q=2,\ldots,n
\,,
& \mbox{ and } &
\rho_1
\;:=\;
\frac{\omega_0}{\rho_2\cdots\rho_n}
\,.
\end{eqnarray*}
Then $\rho_q$ is a $2^{p_q}$-th root of $-1$
for all $q=2,\ldots,n$, and 
$\rho_1$ is a $2^{p_1}$-th root of $-1$ since
$\omega_0$ is and $\rho_q\in\Omega_{2^{p_1}}$
for all $q=2,\ldots,n$
(recall from the hypothesis of Lemma~\ref{lkpregeneral} that
$n\geq2$ and $p_1>p_2>\cdots>p_n\geq0$). Then the $N$-sequence defined
by
\begin{eqnarray}\label{omega0toLN}
\mathcal{L}_N
& := &
\left(
\Omega_{2^{p_1}}
\,,\,
\rho_1\Omega_{2^{p_2}}
\,,\,
\rho_1\rho_2\Omega_{2^{p_3}}
\,,\;
\cdots
\;,\,
\rho_1\cdots\rho_{n-1}\Omega_{2^{p_n}}
\right)
\,,
\end{eqnarray}
is the $N$-section of a Leja sequence
(that starts at $1$), and for all
$k=1,\ldots,2^{p_1}$, the function defined by~(\ref{descrlkgen})
or~(\ref{descrlkgenz0}) is the FLIP (at least in modulus)
associated with the $k$-th point from 
$\Omega_{2^{p_1}}\subset\mathcal{L}_N$.

\end{remark}

\bigskip

Now we give the proof of the two following preliminary (and classical) lemmas.

\begin{lemma}\label{estimsinus}

For all $x\in\left[0\,,\,\dfrac{\pi}{2}\right]$, one has the following estimates:
\begin{eqnarray}\label{estimsinus1}
\frac{2}{\pi}\,x
\;\leq\;
\sin(x)
\;\leq\;
x
\,.
\end{eqnarray}
One also has for all $x\in\R$,
\begin{eqnarray}\label{estimsinus2}
|\sin(x)|
& \leq &
|x|
\,.
\end{eqnarray}

\end{lemma}

\begin{proof}

The first estimate of~(\ref{estimsinus1}) follows from the concavity of the function
$\sin(x)$ on $\left[0\,,\,\dfrac{\pi}{2}\right]$
(by writing 
$x=\left(1-\dfrac{2}{\pi}x\right)\times0\,+\,\left(\dfrac{2}{\pi}x\right)\times\dfrac{\pi}{2}$). 
The second one can be deduced
by considering the variations on $\left[0\,,\,\dfrac{\pi}{2}\right]$ of the function
$x\mapsto x-\sin(x)$.

In particular,
the estimate~(\ref{estimsinus1}) yields~(\ref{estimsinus2}) 
for all $x\in\left[0,\dfrac{\pi}{2}\right]$.
If $x\geq\dfrac{\pi}{2}$, then
\begin{eqnarray*}
|\sin(x)|
\;\leq\;
1
\;\leq\;
\frac{\pi}{2}
\;\leq\;
x
\end{eqnarray*}
and this proves~(\ref{estimsinus2}) for all $x\geq0$.
Finally, if $x\leq0$, then applying the above estimate to $-x\geq0$ leads to
\begin{eqnarray*}
|\sin(x)|
\;=\;
|-\sin(x)|
\;=\;
|\sin(-x)|
\;\leq\;
|-x|
\;=\;
|x|
\,.
\end{eqnarray*}

\end{proof}

The following one can be proved by
considering the variations of the function
$x\in\R\mapsto\exp(x)-x-1$.

\begin{lemma}\label{estimexp}

For all $x\in\R$, one has
\begin{eqnarray*}
\exp(x)
& \geq &
1+x
\,.
\end{eqnarray*}

\end{lemma}

\bigskip

Next, we give and prove the following trigonometric formula.

\begin{lemma}\label{prodcos}

For all $m\geq0$ and $\alpha\notin\,\pi\mathbb{Z}$, one has
\begin{eqnarray}\label{prodcosform}
\prod_{j=1}^m
\cos\left(\frac{\alpha}{2^j}\right)
& = &
\frac{\sin(\alpha)}{2^m\sin\left(\alpha/2^m\right)}
\,.
\end{eqnarray}

\end{lemma}

\begin{proof}

First, we claim that, for all $m\geq0$,
$\alpha/2^m\notin\pi\mathbb{Z}$.

Indeed, this assertion can be proved by induction on $m\geq0$. For $m=0$,
$\alpha/2^m=\alpha\notin\pi\mathbb{Z}$.
Now if for $m\geq1$ being given, one has that
$\alpha/2^m\in\pi\mathbb{Z}$,
then one also has that
$\alpha/2^{m-1}=2\times\alpha/2^m\in\,\pi\mathbb{Z}$. This is impossible by induction hypothesis.

In particular, this implies that $\sin(\alpha/2^m)\neq0$ for all $m\geq0$, then the expressions
which appear in~(\ref{prodcosform}) make sense. The formula~(\ref{prodcosform})
will be proved by induction on
$m\geq0$. For $m=0$, one has that
\begin{eqnarray*}
\prod_{j=1}^0
\cos\left(\frac{\alpha}{2^j}\right)
\;=\;
\prod_{j\in\emptyset}
\cos\left(\frac{\alpha}{2^j}\right)
\;=\;
1
\;=\;
\frac{\sin(\alpha)}{2^0\sin\left(\alpha/2^0\right)}
\,.
\end{eqnarray*}

If it is true for $m\geq0$, then let us consider $m+1$.
First, $\cos\left(\alpha/2^{m+1}\right)\neq0$, otherwise
$\alpha/2^{m+1}\in\,\dfrac{\pi}{2}+\pi\mathbb{Z}$
then
$\alpha/2^m=2\times\alpha/2^{m+1}\in\pi\mathbb{Z}$.
This is impossible by the above claim.
Next, we get by induction hypothesis
\begin{eqnarray*}
\prod_{j=1}^{m+1}
\cos\left(\frac{\alpha}{2^j}\right)
& = &
\prod_{j=1}^{m}
\cos\left(\frac{\alpha}{2^j}\right)
\times
\cos\left(\frac{\alpha}{2^{m+1}}\right)
\;=\;
\frac{\sin(\alpha)}{2^m\sin\left(\alpha/2^m\right)}
\times
\cos\left(\frac{\alpha}{2^{m+1}}\right)
\\
& = &
\frac{\sin(\alpha)}
{2^m\times
2\sin\left(\alpha/2^{m+1}\right)\cos\left(\alpha/2^{m+1}\right)}
\times
\cos\left(\frac{\alpha}{2^{m+1}}\right)
\\
& = &
\frac{\sin(\alpha)}{2^{m+1}\sin\left(\alpha/2^{m+1}\right)}
\,,
\end{eqnarray*}
and this proves the induction.

\end{proof}

\bigskip

In the next section, we will not need the binary decomposition of $l$,
else the alternating binary one as specified by the following result.

\begin{lemma}\label{binaryalt}

For all integer $l\geq1$, one can write
\begin{eqnarray}\label{binaryaltform}
& &
l 
\;=\;
\sum_{i=1}^{2L}
(-1)^{i-1}2^{s_i}
\,,
\mbox{ where }
L\;\geq\;1
\mbox{ and }
s_1>s_2>\cdots>s_{2L}\geq0
\,.
\end{eqnarray}

\end{lemma}

\begin{proof}

First, $l$ being a nonzero integer can be written with a binary decomposition, i.e.
\begin{eqnarray*}
l
& = &
\sum_{j=1}^K2^{r_i}
\,,
\end{eqnarray*}
 where $K\geq1$ and
$r_1>r_2>\cdots>r_K\geq0$.

Next, the formula~(\ref{binaryaltform}) will be proved by induction on $K\geq1$. 
In addition, we will prove that
\begin{eqnarray}\label{s1=r1+1}
s_1
& = &
r_1+1
\,.
\end{eqnarray}
If $K=1$,
then
\begin{eqnarray*}
l
& = &
2^{r_1}
\;=\;
2^{r_1+1}-2^{r_1}
\,,
\end{eqnarray*}
and the assertion is true by setting $L=2$, $s_1=r_1+1$ and $s_2=r_1$
(and we have
$s_1>s_2\geq0$).

Now let be $K\geq2$. If
$r_i=r_1-i+1$ for all
$i=1,\ldots,K$ (this happens when $l$ is a
chain of consecutive "ones" in its binary decomposition), then
\begin{eqnarray*}
l
& = &
\sum_{i=1}^K
2^{r_1-i+1}
\;=\;
2^{r_1-K+1}
\sum_{i=1}^K2^{K-i}
\;=\;
2^{r_1-K+1}
\sum_{i=0}^{K-1}2^i
\\
& = &
2^{r_1-K+1}
\frac{2^K-1}{2-1}
\;=\;
2^{r_1+1}-2^{r_1-K+1}
\;=\;
2^{r_1+1}-2^{r_K}
\,,
\end{eqnarray*}
and the assertion is true by setting
$L=2$, $s_1=r_1+1$ and $s_2=r_K$ (and we have
$s_1>s_2\geq0$).

Otherwise, there is at least one interior "zero" in the binary decomposition of $l$.
Let $v\geq1$ be such that,
\begin{eqnarray}\label{defv}
\begin{cases}
r_i\;=\;r_1-i+1\,,\;\mbox{for all }i=1,\ldots,v\,;
\\
r_{v+1}\;\leq\;r_v-2
\end{cases}
\end{eqnarray}
(counting the consecutive "ones" from $r_1$ in the first chain from the binary decomposition
of $l$, $r_v$ is the rank of the last "one").
Then
\begin{eqnarray*}
\sum_{i=1}^v
2^{r_i}
& = &
\sum_{i=1}^v
2^{r_1-i+1}
\:=\;
2^{r_1-v+1}
\sum_{i=1}^v
2^{v-i}
\;=\;
2^{r_1-v+1}
\sum_{i=0}^{v-1}
2^i
\\
& = &
2^{r_1-v+1}
\frac{2^v-1}{2-1}
\;=\;
2^{r_1+1}-2^{r_1-v+1}
\,.
\end{eqnarray*}
On the other hand, let us consider the (nonzero) integer
\begin{eqnarray*}
\sum_{i=1}^{K'}2^{r'_i}
& := &
\sum_{i=v+1}^{K}2^{r_i}
\,,
\end{eqnarray*}
where $K':=K-v$, $r'_i:=r_{v+i}$ for all $i=1,\ldots,K'$. One still has that
$K'\geq1$ (since $v+1\leq K$) and
$r'_1>r'_2>\cdots>r'_{K'}=r_K\geq0$.
Since $K'=K-v\leq K-1$, the induction hypothesis can be applied to
\begin{eqnarray*}
\sum_{i=1}^{K'}2^{r'_i}
& = &
\sum_{i=1}^{2L'}
(-1)^{i-1}2^{s'_i}
\,,
\end{eqnarray*}
with
$L'\geq1$ and
$s'_1>s'_2>\cdots>s'_{2L'}\geq0$.
It follows that
\begin{eqnarray*}
l
& = &
\sum_{i=1}^v
2^{r_i}
+
\sum_{i=v+1}^{K}2^{r_i}
\;=\;
2^{r_1+1}-2^{r_1-v+1}
+
\sum_{i=1}^{K'}2^{r'_i}
\\
& = &
2^{r_1+1}-2^{r_1-v+1}
+
\sum_{i=1}^{2L'}
(-1)^{i-1}2^{s'_i}
\;=\;
\sum_{i=1}^{2L}
(-1)^{i-1}2^{s_i}
\,,
\end{eqnarray*}
where we have set
$L:=L'+1$, $s_1:=r_1+1$, $s_2:=r_1-v+1$
and
$s_i:=s'_{i-2}$ for all $i=3,\ldots,2L=2L'+2$. In particular, one still has that
$s_1>s_2$ and by~(\ref{defv}),
$s_2=r_1-v+1=r_v>r_{v+1}+1=r'_1+1=s'_1=s_3$
(since by the induction hypothesis applied to~(\ref{s1=r1+1}),
one has $s'_1=r'_1+1$). The induction hypothesis applied to
$s'_1>s'_2>\cdots>s'_{2L'}=s'_{2L-2}\geq0$, yields
$s_3>s_4>\cdots>s_{2L}\geq0$, and the proof is finished.

\end{proof}

\begin{remark}

In the above formula, $L$ is the number of chains of consecutive "ones" in the binary decomposition
of $l$.

\end{remark}

Since we will deal with alternating sums in the next section,
the following result will be useful.

\begin{lemma}\label{estimrestalt}

Let us consider
\begin{eqnarray*}
\sum_{i=1}^M
(-1)^ia_i
& &
\mbox{ (resp., }
\sum_{i=1}^M
(-1)^{i-1}a_i
\mbox{ )}
\,,
\end{eqnarray*}
where
$a_i>0$ is decreasing. Then
for all $J=1,\ldots,M$,
\begin{eqnarray}
\left|
\sum_{i=J}^M
(-1)^ia_i
\right|
\;=\;
\sum_{i=J}^M
(-1)^{i-J}a_i
& \leq &
a_{J}
\,.
\end{eqnarray}

\end{lemma}

\begin{proof}

The proof is by descending induction on
$J=M,\ldots,1$. The assertion is obvious for
$J=M$. If $1\leq J\leq M-1$, then the induction hypothesis yields
\begin{eqnarray*}
\sum_{i=J}^M
(-1)^{i-J}a_i
& = &
a_J-\sum_{i=J+1}^M(-1)^{i-J-1}a_i
\;\geq\;
a_J-
\left|\sum_{i=J+1}^M(-1)^{i-J-1}a_i\right|
\\
& \geq &
a_J-a_{J+1}
\;\geq\;
0
\,,
\end{eqnarray*}
since $a_i$ is decreasing.
In particular, 
$\sum_{i=J}^M
(-1)^{i-J}a_i\geq0$ then
\begin{eqnarray*}
\left|
\sum_{i=J}^M
(-1)^{i}a_i
\right|
\;=\;
\left|
\sum_{i=J}^M
(-1)^{i-J}a_i
\right|
\;=\;
\sum_{i=J}^M
(-1)^{i-J}a_i
\,.
\end{eqnarray*}
On the other hand,
\begin{eqnarray*}
\sum_{i=J}^M
(-1)^{i-J}a_i
& = &
a_J-
\sum_{i=J+1}^M(-1)^{i-J-1}a_i
\;\leq\;
a_J
\,,
\end{eqnarray*}
since by the induction hypothesis,
$\sum_{i=J+1}^M(-1)^{i-J-1}a_i
=\left|\sum_{i=J+1}^M(-1)^{i-J-1}a_i\right|\geq0$.
The proof is finished.

\end{proof}

\bigskip

\section{Proof of Theorem~\ref{unifestim} in a special but essential case}\label{specialessential}

In this part, we will deal with the FLIP $l_1^{(N)}$
(the one that is associated with $z_1=1$), i.e. 
by recalling~(\ref{descrlkgen}) and~(\ref{descrlkgenz0}) from Lemma~\ref{lkgeneral},
\begin{eqnarray}\label{descrl1genznon0}
\left|
l_1^{(N)}(z)
\right|
& = &
\frac{1}{2^{p_1}}
\frac{\left|z^{2^{p_1}}-1\right|}{\left|z-1\right|}
\times
\prod_{q=2}^n
\frac{
\left|z^{2^{p_q}}+\omega_0^{2^{p_q}}\right|
}{
\left|1+\omega_0^{2^{p_q}}\right|
}
\;\mbox{ (for all $z\neq1$)}
\,,
\\\label{descrl1genz0}
& = &
\frac{1}{2^{p_1}}
\left|
\sum_{j=0}^{2^{p_1}-1}
z^j
\right|
\times
\prod_{q=2}^n
\frac{
\left|z^{2^{p_q}}+\omega_0^{2^{p_q}}\right|
}{
\left|1+\omega_0^{2^{p_q}}\right|
}
\;\mbox{ (for all $z\in\C$)}
\,,
\end{eqnarray}
with the following hypothesis:

\begin{itemize}

\item

$N\geq1$ and $N$ is not a pure power of $2$, i.e. there is
$p_1\geq1$ such that
$2^{p_1}<N<2^{p_1+1}$. We can then remind the binary decomposition of $N$ given
by~(\ref{bindec}) with~(\ref{defnpj}):
\begin{eqnarray}\label{restrN}
N
\;=\;
2^{p_1}+\cdots+2^{p_n}
& \mbox{with} &
p_1>p_2>\cdots>p_n\geq0
\mbox{ and }
n\geq2\,;
\end{eqnarray}

\item

as it was defined in Lemma~\ref{lkgeneral},
$\omega_0$ is a $2^{p_1}$-th root of $-1$, then it can be written as
$\exp\left((2l+1)i\pi/2^{p_1}\right)$ with
$0\leq l\leq2^{p_1}-1$. In this section, we will not
consider the special case of
$\omega_0=\exp\left(i\pi/2^{p_1}\right)$, i.e.
\begin{eqnarray}\label{restromega0}
\omega_0
\;=\;
\exp\left(\dfrac{2l+1}{2^{p_1}}i\pi\right)
& \mbox{with} &
1\;\leq\;l\;\leq\;2^{p_1}-1\,;
\end{eqnarray}

\item

we will deal with 
$|z|=1$. It can be written as $\exp\left(i\theta\right)$
where $\theta\in]-\pi,\,\pi]$, but
the following writing will be more useful in this section:
\begin{eqnarray}\label{restrz}
z
\;=\;
\exp\left(i\pi\theta\right)
& \mbox{with} &
\theta\in\,]-1,1]
\,.
\end{eqnarray}

\end{itemize}

The goal of this section is to prove Theorem~\ref{unifestim}
for the FLIP $l_1^{(N)}$ under these above
restrictions, i.e. that for all $|z|=1$ and all
$l=1,\ldots,2^{p_1}-1$,
\begin{eqnarray}\label{estimatespecial}
\left|
l_1^{(N)}
\left(
z
\right)
\right|
& \leq &
\pi\exp\left(3\pi\right)
\,.
\end{eqnarray}

We begin with noticing that since $|\theta|\leq1$, then one has that:

\begin{enumerate}

\item

or $|\theta|\leq1/2^{p_1}$, then this case will be handled
in Subsection~\ref{step1};

\item

or else $1/2^{p_1}<|\theta|\leq1/2^{p_n}$. Since by~(\ref{restrN}) one has
$0<1/2^{p_1}<\cdots<1/2^{p_n}\leq1$,
it follows that
\begin{eqnarray}\label{defpqtheta}
1/2^{p_{q_{\theta}-1}}
\;<\;
|\theta|
\;\leq\;
1/2^{p_{q_{\theta}}}
& \mbox{ with } &
q_{\theta}\;=\;2,\ldots,n
\,.
\end{eqnarray}
We will deal with this case in Subsection~\ref{step2};

\item

otherwise $1/2^{p_n}<|\theta|\leq1$, then
this case will be handled in Subsection~\ref{step3}.

\end{enumerate}

\bigskip

\subsection{Some preliminary results}

Before dealing with the different cases, we need a couple of preliminary results
which will be useful in this section. We begin with the first one.

\begin{lemma}\label{estimquotcos}

For all $q=2,\ldots,n$, for all $l=1,\ldots,2^{p_1}-1$, and all $|\theta|\leq1$, one has
\begin{eqnarray*}
\frac{
\left|
\cos
\left(
\dfrac{2l+1}{2^{p_1-p_q+1}}\pi
-
\dfrac{2^{p_q}\pi\theta}{2}
\right)
\right|
}{
\left|
\cos
\left(
\dfrac{2l+1}{2^{p_1-p_q+1}}\pi
\right)
\right|
}
& \leq &
1+
\dfrac{2^{p_q}\pi|\theta|}{2}
\left|
\tan
\left(
\dfrac{2l+1}{2^{p_1-p_q+1}}\pi
\right)
\right|
\,.
\end{eqnarray*}

\end{lemma}

\begin{proof}

Indeed, one has that
\begin{eqnarray*}
\frac{
\left|
\cos
\left(
\dfrac{2l+1}{2^{p_1-p_q+1}}\pi
-
\dfrac{2^{p_q}\pi\theta}{2}
\right)
\right|
}{
\left|
\cos
\left(
\dfrac{2l+1}{2^{p_1-p_q+1}}\pi
\right)
\right|
}
& \leq &
\;\;\;\;\;\;\;\;\;\;\;\;\;\;\;\;\;\;\;\;\;\;\;\;\;\;\;\;\;\;\;\;\;\;\;\;\;\;\;\;\;\;\;\;
\;\;\;\;\;\;\;\;\;\;\;\;\;\;\;\;\;\;\;\;\;\;\;\;\;\;\;\;\;\;\;\;\;\;\;\;\;\;\;\;
\end{eqnarray*}
\begin{eqnarray*}
\;\;\;\;\;\;\;\;\;\;
& \leq &
\frac{
\left|
\cos
\left(
\dfrac{2l+1}{2^{p_1-p_q+1}}\pi
\right)
\right|
\left|
\cos
\left(
\dfrac{2^{p_q}\pi\theta}{2}
\right)
\right|
+
\left|
\sin
\left(
\dfrac{2l+1}{2^{p_1-p_q+1}}\pi
\right)
\right|
\left|
\sin
\left(
\dfrac{2^{p_q}\pi\theta}{2}
\right)
\right|
}{
\left|
\cos
\left(
\dfrac{2l+1}{2^{p_1-p_q+1}}\pi
\right)
\right|
}
\\
& = &
\left|
\cos
\left(
\dfrac{2^{p_q}\pi\theta}{2}
\right)
\right|
+
\left|
\tan
\left(
\dfrac{2l+1}{2^{p_1-p_q+1}}\pi
\right)
\right|
\left|
\sin
\left(
\dfrac{2^{p_q}\pi\theta}{2}
\right)
\right|
\\
& \leq &
1+
\dfrac{2^{p_q}\pi|\theta|}{2}
\left|
\tan
\left(
\dfrac{2l+1}{2^{p_1-p_q+1}}\pi
\right)
\right|
\,,
\end{eqnarray*}
the last estimate coming by~(\ref{estimsinus2}) from Lemma~\ref{estimsinus}.

\end{proof}

Next, we give the following result that uses the alternating binary decomposition of $l$
that is guaranteed by Lemma~\ref{binaryalt} (since we have assumed that
$l\geq1$).

\begin{lemma}\label{partitionTS}

Let $l\geq1$ be defined by~(\ref{restromega0}) and let us consider its 
alternating binary decomposition
\begin{eqnarray}\label{recallbinaryaltl}
l
& = &
\sum_{i=1}^{2L}
(-1)^{i-1}2^{s_i}
\,,
\end{eqnarray}
where $L\geq1$ and
\begin{eqnarray}\label{recallbinaryaltsj}
s_1\;>\;s_2
\;>\;\cdots\;>\;s_{2L}
\;\geq\;0
\,.
\end{eqnarray}
Then $s_1\leq p_1$.
In addition, if we set
\begin{eqnarray}\label{s2L+1}
s_{2L+1}
& := &
-1
\end{eqnarray}
and
\begin{eqnarray}\label{s0}
s_0
& := &
p_1+1
\,,
\end{eqnarray}
then one still has that
\begin{eqnarray*}
s_0
\;>\;
s_1
\;>\;
\cdots
\;>\;
s_{2L}
\;>\;
s_{2_l+1}
\,,
\end{eqnarray*}
and for all $q=2,\ldots,n$, or
\begin{eqnarray}\label{defT}
p_q\;\in\;
T
\;:=\;
\bigcup_{j=0}^{2L}
\left\{
p_1-s_j
\;\leq\;
i
\;\leq\;
p_1-s_{j+1}-2
\right\}
\end{eqnarray}
(with the convention that the subsets for which
$s_{j+1}=s_j-1$ are empty),
or else
\begin{eqnarray}\label{defS}
p_q
\;\in\;
S
\;:=\;
\left\{
p_1-s_j-1\,,\;j=1,\ldots,2L
\right\}
\,.
\end{eqnarray}

\end{lemma}

\begin{proof}

First, the decomposition~(\ref{recallbinaryaltl}) yields
\begin{eqnarray*}
l
& = &
\sum_{i=1}^{2L}
(-1)^{i-1}2^{s_i}
\;\geq\;
2^{s_1}
-
\left|
\sum_{i=2}^{2L}
(-1)^{i-1}2^{s_i}
\right|
\;\geq\;
2^{s_1}
-2^{s_2}
\,,
\end{eqnarray*}
the last estimate being justified by Lemma~\ref{estimrestalt}
(because $s_j$ is decreasing).
Since $s_2\leq s_1-1$ by~(\ref{recallbinaryaltsj}), 
it follows by~(\ref{restromega0}) that
\begin{eqnarray*}
2^{p_1}-1
\;\geq\;
l
& \geq &
2^{s_1}-2^{s_1-1}
\;=\;
2^{s_1-1}
\,.
\end{eqnarray*}
If we had that $s_1\geq p_1+1$, this would yield
\begin{eqnarray*}
2^{p_1}-1
\;\geq\;
2^{(p_1+1)-1}
\;=\;
2^{p_1}
\end{eqnarray*}
and this is impossible. Then $s_1\leq p_1$.

Now let fix
$q=2,\ldots,n$. We know by~(\ref{restrN}) and~(\ref{s2L+1}) that 
\begin{eqnarray*}
p_q
\;\leq\;
p_2
\;\leq\;
p_1-1
\;=\;
p_1-s_{2L+1}-2
\,.
\end{eqnarray*}
Similarly, we have by~(\ref{restrN}) and~(\ref{s0}) that 
\begin{eqnarray*}
p_q
\;\geq\;
p_n
\;\geq\;
0
\;>\;
p_1-s_0
\,.
\end{eqnarray*}
It follows that
$p_1-s_0\leq p_q\leq p_1-s_{2L+1}-2$, i.e.
$p_q\in T\cup S$ and the lemma is proved (since
$T$ and $S$ are incompatible).

\end{proof}

\medskip

Now that we have defined the sets $T$ and $S$, we will give auxiliary results
which deal with these sets. We begin with $T$.

\begin{lemma}\label{estimquotcospi/4}

Let be $q=2,\ldots,n$,
and $l=1,\ldots,2^{p_1}-1$.
If $p_q\in T$ then for all $|\theta|\leq1$, one has
\begin{eqnarray*}
\frac{
\left|
\cos
\left(
\dfrac{2l+1}{2^{p_1-p_q+1}}\pi
-
\dfrac{2^{p_q}\pi\theta}{2}
\right)
\right|
}{
\left|
\cos
\left(
\dfrac{2l+1}{2^{p_1-p_q+1}}\pi
\right)
\right|
}
& \leq &
1+
\dfrac{2^{p_q}\pi|\theta|}{2}
\,.
\end{eqnarray*}

\end{lemma}

\begin{proof}

First, since
$p_q\in T$, by~(\ref{defT}) there is $j_q$ with $0\leq j_q\leq 2L$
such that
\begin{eqnarray}\label{pi/4}
p_1-s_{j_q}
\;\leq\;
p_q
\;\leq\;
p_1-s_{j_q+1}-2
\,,
\end{eqnarray}
then (by~(\ref{recallbinaryaltl}) and the convention~(\ref{s2L+1}) )
\begin{eqnarray*}
\frac{2l+1}{2^{p_1-p_q+1}}\pi
& = &
\frac{\pi}{2^{p_1-p_q}}
\left(
l+2^{s_{2L+1}}
\right)
\;=\;
\frac{\pi}{2^{p_1-p_q}}
\sum_{i=1}^{2L+1}
(-1)^{i-1}2^{s_i}
\\
& = &
\pi\sum_{i=1}^{j_q}
(-1)^{i-1}2^{s_i-p_1+p_q}
+
\pi\sum_{i=j_q+1}^{2L+1}
\frac{(-1)^{i-1}}{2^{p_1-p_q-s_i}}
\end{eqnarray*}
In the first sum, one has
$s_i-p_1+p_q\geq s_{j_q}-p_1+p_q\geq0$ 
by~(\ref{recallbinaryaltsj}) and~(\ref{pi/4})
for all $i=1,\ldots,j_q$, then
\begin{eqnarray}
b_q
\;:=\;
\sum_{i=1}^{j_q}
(-1)^{i-1}2^{s_i-p_1+p_q}
& \in &
\mathbb{Z}
\end{eqnarray}
(notice that the sum
$\sum_{i=1}^{j_q}
(-1)^{i-1}2^{s_i-p_1+p_q}$ may be empty if $j_q=0$; even in this case,
the above assertion holds true since
$b_q=\sum_{i\in\emptyset}(-1)^{i-1}2^{s_i-p_1+p_q}=0\in\Z$).
In the second sum (that cannot be empty since
$j_q+1\leq2L+1$), one has
\begin{eqnarray*}
\pi\sum_{i=j_q+1}^{2L+1}
\frac{(-1)^{i-1}}{2^{p_1-p_q-s_i}}
& = &
\frac{(-1)^{j_q}\pi}{2^{p_1-p_q-s_{j_q+1}}}
\sum_{i=j_q+1}^{2L+1}
\frac{(-1)^{i-j_q-1}}{2^{s_{j_q+1}-s_i}}
\,.
\end{eqnarray*}
It follows that
\begin{eqnarray}\nonumber
\left|
\tan
\left(
\frac{2l+1}{2^{p_1-p_q+1}}\pi
\right)
\right|
& = &
\left|
\tan
\left(
b_q\pi
+
\frac{(-1)^{j_q}\pi}{2^{p_1-p_q-s_{j_q+1}}}
\sum_{i=j_q+1}^{2L+1}
\frac{(-1)^{i-j_q-1}}{2^{s_{j_q+1}-s_i}}
\right)
\right|
\\\label{reductanpi/4}
& = &
\left|
\tan
\left(
\frac{\pi}{2^{p_1-p_q-s_{j_q+1}}}
\sum_{i=j_q+1}^{2L+1}
\frac{(-1)^{i-j_q-1}}{2^{s_{j_q+1}-s_i}}
\right)
\right|
\,.
\end{eqnarray}

Next, we claim that
\begin{eqnarray}\label{pi/4encadre}
0\;<\;
\frac{\pi}{2^{p_1-p_q-s_{j_q+1}}}
\sum_{i=j_q+1}^{2L+1}
\frac{(-1)^{i-j_q-1}}{2^{s_{j_q+1}-s_i}}
\;\leq\;
\frac{\pi}{4}
\,.
\end{eqnarray}
Indeed, one has on the one hand that
$p_1-p_q-s_{j_q+1}\geq2$ by~(\ref{pi/4}), then
\begin{eqnarray}\label{pi/4encadre1}
0
\;<\;
\frac{\pi}{2^{p_1-p_q-s_{j_q+1}}}
\;\leq\;
\frac{\pi}{4}
\,.
\end{eqnarray}
On the other hand, one has by Lemma~\ref{estimrestalt} that
\begin{eqnarray*}
\sum_{i=j_q+1}^{2L+1}
\frac{(-1)^{i-j_q-1}}{2^{s_{j_q+1}-s_i}}
& \leq &
\left|
\sum_{i=j_q+1}^{2L+1}
\frac{(-1)^{i-j_q-1}}{2^{s_{j_q+1}-s_i}}
\right|
\;\leq\;
\frac{1}{2^{s_{j_q+1}-s_{j_q+1}}}
\;=\;
1
\,,
\end{eqnarray*}
and
\begin{eqnarray*}
\sum_{i=j_q+1}^{2L+1}
\frac{(-1)^{i-j_q-1}}{2^{s_{j_q+1}-s_i}}
& = &
1-
\sum_{i=j_q+2}^{2L+1}
\frac{(-1)^{i-j_q}}{2^{s_{j_q+1}-s_i}}
\;\geq\;
1-
\left|
\sum_{i=j_q+2}^{2L+1}
\frac{(-1)^{i-j_q}}{2^{s_{j_q+1}-s_i}}
\right|
\\
& \geq &
1-\frac{1}{2^{s_{j_q+1}-s_{j_q+2}}}
\;\geq\;
1-\frac{1}{2}
\;=\;
\frac{1}{2}
\,,
\end{eqnarray*}
the last estimate coming from~(\ref{recallbinaryaltsj})
(notice that the sum
$\sum_{i=j_q+2}^{2L+1}
(-1)^{i-j_q}/2^{s_{j_q+1}-s_i}$
may be empty if $j_q=2L$; even in this case,
one still has that\\
$\sum_{i=j_q+1}^{2L+1}
(-1)^{i-j_q-1}/2^{s_{j_q+1}-s_i}=1-0\geq1/2$).
Hence
\begin{eqnarray}\label{pi/4encadre2}
\frac{1}{2}
\;\leq\;
\sum_{i=j_q+1}^{2L+1}
\frac{(-1)^{i-j_q-1}}{2^{s_{j_q+1}-s_i}}
\;\leq\;
1
\,.
\end{eqnarray}
The claim follows by~(\ref{pi/4encadre1}) and~(\ref{pi/4encadre2}).

In particular, 
$\dfrac{\pi}{2^{p_1-p_q-s_{j_q+1}}}
\sum_{i=j_q+1}^{2L+1}
(-1)^{i-j_q-1}/2^{s_{j_q+1}-s_i}
\in\,\left]0,\dfrac{\pi}{2}\right[$
where the function $\tan$ is positive and increasing.
It follows by~(\ref{reductanpi/4}) that
\begin{eqnarray*}
\left|
\tan
\left(
\frac{2l+1}{2^{p_1-p_q+1}}\pi
\right)
\right|
& = &
\left|
\tan
\left(
\frac{\pi}{2^{p_1-p_q-s_{j_q+1}}}
\sum_{i=j_q+1}^{2L+1}
\frac{(-1)^{i-j_q-1}}{2^{s_{j_q+1}-s_i}}
\right)
\right|
\\
& = &
\tan
\left(
\frac{\pi}{2^{p_1-p_q-s_{j_q+1}}}
\sum_{i=j_q+1}^{2L+1}
\frac{(-1)^{i-j_q-1}}{2^{s_{j_q+1}-s_i}}
\right)
\;\leq\;
\tan
\left(
\frac{\pi}{4}
\right)
\;=\;1
\,,
\end{eqnarray*}
the estimate coming by applying~(\ref{pi/4encadre}) again.
It follows by Lemma~\ref{estimquotcos} that
\begin{eqnarray*}
\frac{
\left|
\cos
\left(
\dfrac{2l+1}{2^{p_1-p_q+1}}\pi
-
\dfrac{2^{p_q}\pi\theta}{2}
\right)
\right|
}{
\left|
\cos
\left(
\dfrac{2l+1}{2^{p_1-p_q+1}}\pi
\right)
\right|
}
& \leq &
1+
\dfrac{2^{p_q}\pi|\theta|}{2}
\left|
\tan
\left(
\dfrac{2l+1}{2^{p_1-p_q+1}}\pi
\right)
\right|
\\
& \leq &
1+
\dfrac{2^{p_q}\pi|\theta|}{2}
\,,
\end{eqnarray*}
and this proves the lemma.

\end{proof}

Next, we deal with the set $S$.

\begin{lemma}\label{estimquotcospi/2}

Let be $q=2,\ldots,n$, and $l=1,\ldots,2^{p_1}-1$. Assume that $p_q\in S$, i.e.
by~(\ref{defS}) from Lemma~\ref{partitionTS} there is
$j_q$ with
$1\leq j_q\leq 2L$ such that
\begin{eqnarray}\label{pi/2}
p_q
& = &
p_1-s_{j_q}-1
\,.
\end{eqnarray}
Then for all $|\theta|\leq1$,
\begin{eqnarray*}
\frac{
\left|
\cos
\left(
\dfrac{2l+1}{2^{p_1-p_q+1}}\pi
-
\dfrac{2^{p_q}\pi\theta}{2}
\right)
\right|
}{
\left|
\cos
\left(
\dfrac{2l+1}{2^{p_1-p_q+1}}\pi
\right)
\right|
}
& \leq &
1+\frac{2^{p_1}|\theta|\pi}{2^{1+s_{j_q+1}}}
\end{eqnarray*}
($s_{j_q+1}$ makes sense since
$2\leq j_q+1\leq 2L+1$).

\end{lemma}

\begin{proof}

First, $q$ and the associated $j_q$ being fixed,
one has by~(\ref{recallbinaryaltl}) and the convention~(\ref{s2L+1})
from Lemma~\ref{partitionTS} that
\begin{eqnarray*}
\frac{2l+1}{2^{p_1-p_q+1}}\pi
& = &
\frac{\pi}{2^{p_1-p_q}}
\sum_{i=1}^{2L+1}
(-1)^{i-1}2^{s_i}
\\
& = &
\pi\sum_{i=1}^{j_q-1}
(-1)^{i-1}2^{s_i-p_1+p_q}
+
\frac{(-1)^{j_q-1}\pi}{2^{p_1-p_q-s_{j_q}}}
+\pi
\sum_{i=j_q+1}^{2L+1}
\frac{(-1)^{i-1}}{2^{p_1-p_q-s_i}}
\,.
\end{eqnarray*}
As before, if the first sum is not empty (otherwise it gives $0\in\,\pi\mathbb{Z}$), one has
by~(\ref{recallbinaryaltsj}) and~(\ref{pi/2}) for all $i=1,\ldots,j_q-1$, that
\begin{eqnarray*}
s_i-p_1+p_q
\;\geq\;
s_{j_q-1}-p_1+p_q
\;\geq\;
s_{j_q}+1-p_1+p_q
\;=\;
-1+1
\;=\;
0
\,,
\end{eqnarray*}
then
\begin{eqnarray*}
\widetilde{b}_q
\;:=\;
\sum_{i=1}^{j_q-1}
(-1)^{i-1}2^{s_i-p_1+p_q}
& \in &
\mathbb{Z}
\,.
\end{eqnarray*}
Similarly, for all $i=j_q+1,\ldots,2L+1$,
\begin{eqnarray*}
\pi
\sum_{i=j_q+1}^{2L+1}
\frac{(-1)^{i-1}}{2^{p_1-p_q-s_i}}
& = &
\frac{(-1)^{j_q}\pi}{2^{p_1-p_q-s_{j_q+1}}}
\sum_{i=j_q+1}^{2L+1}
\frac{(-1)^{i-j_q-1}}{2^{s_{j_q+1}-s_i}}
\,.
\end{eqnarray*}
Since by~(\ref{pi/2}) again, $p_1-p_q-s_{j_q}=1$, it follows that
\begin{eqnarray}\nonumber
\left|
\tan
\left(
\frac{2l+1}{2^{p_1-p_q+1}}\pi
\right)
\right|
& = &
\left|
\tan
\left(
\widetilde{b}_q\pi
+(-1)^{j_q-1}\frac{\pi}{2}
+
\frac{(-1)^{j_q}\pi}{2^{p_1-p_q-s_{j_q+1}}}
\sum_{i=j_q+1}^{2L+1}
\frac{(-1)^{i-j_q-1}}{2^{s_{j_q+1}-s_i}}
\right)
\right|
\\\nonumber
& = &
\left|
\cot
\left(
\frac{\pi}{2^{p_1-p_q-s_{j_q+1}}}
\sum_{i=j_q+1}^{2L+1}
\frac{(-1)^{i-j_q-1}}{2^{s_{j_q+1}-s_i}}
\right)
\right|
\\\label{reductanpi/2}
& \leq &
\frac{1}{
\left|
\sin
\left(
\dfrac{\pi}{2^{p_1-p_q-s_{j_q+1}}}
\sum_{i=j_q+1}^{2L+1}
(-1)^{i-j_q-1}/2^{s_{j_q+1}-s_i}
\right)
\right|
}
\,.
\end{eqnarray}

Next, as in the proof of Lemma~\ref{estimquotcospi/4}, we claim that
\begin{eqnarray}\label{pi/2encadre}
\frac{\pi/2}{2^{p_1-p_q-s_{j_q+1}}}
\;\leq\;
\frac{\pi}{2^{p_1-p_q-s_{j_q+1}}}
\sum_{i=j_q+1}^{2L+1}
\frac{(-1)^{i-j_q-1}}{2^{s_{j_q+1}-s_i}}
\;\leq\;
\frac{\pi}{4}
\,.
\end{eqnarray}
Indeed, one has on the one hand by~(\ref{recallbinaryaltsj}) and~(\ref{pi/2}) that
$p_1-p_q-s_{j_q+1}\geq p_1-p_q-s_{j_q}+1=1+1=2$ then
\begin{eqnarray}\label{pi/2encadre1}
\frac{\pi}{2^{p_1-p_q-s_{j_q+1}}}
& \leq &
\frac{\pi}{4}
\,.
\end{eqnarray}
On the other hand, one has by Lemma~\ref{estimrestalt} that
\begin{eqnarray*}
\sum_{i=j_q+1}^{2L+1}
\frac{(-1)^{i-j_q-1}}{2^{s_{j_q+1}-s_i}}
& \leq &
\left|
\sum_{i=j_q+1}^{2L+1}
\frac{(-1)^{i-j_q-1}}{2^{s_{j_q+1}-s_i}}
\right|
\;\leq\;
\frac{1}{2^{s_{j_q+1}-s_{j_q+1}}}
\;=\;
1
\,,
\end{eqnarray*}
and
\begin{eqnarray*}
\sum_{i=j_q+1}^{2L+1}
\frac{(-1)^{i-j_q-1}}{2^{s_{j_q+1}-s_i}}
& = &
1-
\sum_{i=j_q+2}^{2L+1}
\frac{(-1)^{i-j_q}}{2^{s_{j_q+1}-s_i}}
\;\geq\;
1-
\left|
\sum_{i=j_q+2}^{2L+1}
\frac{(-1)^{i-j_q}}{2^{s_{j_q+1}-s_i}}
\right|
\\
& \geq &
1-
\frac{1}{2^{s_{j_q+1}-s_{j_q+2}}}
\;\geq\;
1-\frac{1}{2}
\;=\;
\frac{1}{2}
\,,
\end{eqnarray*}
the last estimate coming from~(\ref{recallbinaryaltsj}).
Once again, if the sum
$\sum_{i=j_q+2}^{2L+1}
(-1)^{i-j_q}/2^{s_{j_q+1}-s_i}$ is empty (only if
$j_q=2L$), then
one still has that
$\sum_{i=j_q+1}^{2L+1}
(-1)^{i-j_q-1}/2^{s_{j_q+1}-s_i}=1\geq1/2$.
Hence
\begin{eqnarray}\label{pi/2encadre2}
\frac{1}{2}
\;\leq\;
\sum_{i=j_q+1}^{2L+1}
\frac{(-1)^{i-j_q-1}}{2^{s_{j_q+1}-s_i}}
\;\leq\;
1
\,.
\end{eqnarray}
The claim follows by~(\ref{pi/2encadre1}) and~(\ref{pi/2encadre2}).

In particular, 
$\dfrac{\pi}{2^{p_1-p_q-s_{j_q+1}}}
\sum_{i=j_q+1}^{2L+1}
(-1)^{i-j_q-1}/2^{s_{j_q+1}-s_i}
\;\in\;\left]0,\dfrac{\pi}{2}\right[$
where the function $\sin$ is positive and increasing.
It follows by~(\ref{reductanpi/2}) and~(\ref{pi/2encadre}) that
\begin{eqnarray*}
\left|
\tan
\left(
\frac{2l+1}{2^{p_1-p_q+1}}\pi
\right)
\right|
& \leq &
\frac{1}{
\left|
\sin
\left(
\dfrac{\pi}{2^{p_1-p_q-s_{j_q+1}}}
\sum_{i=j_q+1}^{2L+1}
(-1)^{i-j_q-1}/2^{s_{j_q+1}-s_i}
\right)
\right|
}
\\
& = &
\frac{1}{
\sin
\left(
\dfrac{\pi}{2^{p_1-p_q-s_{j_q+1}}}
\sum_{i=j_q+1}^{2L+1}
(-1)^{i-j_q-1}/2^{s_{j_q+1}-s_i}
\right)
}
\\
& \leq &
\frac{1}{
\sin
\left(
\dfrac{\pi/2}{2^{p_1-p_q-s_{j_q+1}}}
\right)
}
\;\leq\;
\frac{1}{
\dfrac{2}{\pi}
\times
\dfrac{\pi/2}{2^{p_1-p_q-s_{j_q+1}}}
}
\;=\;
2^{p_1-p_q-s_{j_q+1}}
\,,
\end{eqnarray*}
the last estimate being justified by~(\ref{estimsinus1})
from Lemma~\ref{estimsinus}.

Finally, an application of Lemma~\ref{estimquotcos} yields
\begin{eqnarray*}
\frac{
\left|
\cos
\left(
\dfrac{2l+1}{2^{p_1-p_q+1}}\pi
-
\dfrac{2^{p_q}\pi\theta}{2}
\right)
\right|
}{
\left|
\cos
\left(
\dfrac{2l+1}{2^{p_1-p_q+1}}\pi
\right)
\right|
}
& \leq &
1+
\dfrac{2^{p_q}\pi|\theta|}{2}
\left|
\tan
\left(
\dfrac{2l+1}{2^{p_1-p_q+1}}\pi
\right)
\right|
\\
& \leq &
1+
\dfrac{2^{p_q}\pi|\theta|}{2}
\times
2^{p_1-p_q-s_{j_q+1}}
\;=\;
1+\frac{2^{p_1}\pi|\theta|}{2^{1+s_{j_q+1}}}
\,,
\end{eqnarray*}
and the lemma is proved.

\end{proof}

We finish the subsection with this result about
the application $q\mapsto j_q$.

\begin{lemma}\label{jqinj}

The following application (that is well-defined
by~(\ref{pi/2}) )
\begin{eqnarray}
\left\{
2\leq q\leq n,\,p_q\in S
\right\}
& \xrightarrow[]{} &
\left\{
1\leq j\leq 2L
\right\}
\\\nonumber
q
& \mapsto &
j_q
\,,
\end{eqnarray}
is strictly decreasing. It is in particular
injective.

\end{lemma}

\begin{proof}

Indeed, let be $q,\,q'$ 
with $2\leq q<q'\leq n$ 
and such that
$p_q,\,p_{q'}\in S$.
By~(\ref{defS}), there are 
$1\leq j_q,\,j_{q'}\leq 2L$ such that
$p_q=p_1-s_{j_q}-1$ and
$p_{q'}=p_1-s_{j_{q'}}-1$.
Since one has by~(\ref{restrN}) that
\begin{eqnarray*}
p_1-s_{j_q}-1
\;=\;
p_q
\;>\;
p_{q'}
\;=\;
p_1-s_{j_{q'}}-1
\,,
\end{eqnarray*}
it follows that
$s_{j_q}<s_{j_{q'}}$ and by~(\ref{recallbinaryaltsj}) 
this gives
\begin{eqnarray*}
j_q
& > &
j_{q'}
\,.
\end{eqnarray*}

\end{proof}

\bigskip

\subsection{First case: $|\theta|\leq1/2^{p_1}$}\label{step1}

In this subsection, we want to prove the estimate~(\ref{estimatespecial}) when
$|\theta|\leq1/2^{p_1}$. We can assume in all
the following that $\theta\neq0$ since
\begin{eqnarray*}
l_1^{(N)}(1)
\;=\;
l_1^{(N)}(z_1)
\;=\;
1
\,.
\end{eqnarray*}
First, one has by~(\ref{descrl1genz0}), (\ref{restromega0}) and~(\ref{restrz}) that
\begin{eqnarray*}\nonumber
\left|
l_1^{(N)}
\left(
z
\right)
\right|
\;=\;
\left|
l_1^{(N)}
\left(
\exp(i\pi\theta)
\right)
\right|
& = &
\;\;\;\;\;\;\;\;\;\;\;\;\;\;\;\;\;\;\;\;\;\;\;\;\;\;\;\;\;\;\;\;\;\;\;\;
\;\;\;\;\;\;\;\;\;\;\;\;\;\;\;\;\;\;\;\;\;\;\;\;\;\;\;\;\;\;\;\;\;\;\;\;
\end{eqnarray*}
\begin{eqnarray*}
& = &
\frac{1}{2^{p_1}}
\left|
\sum_{j=0}^{2^{p_1}-1}
\exp
\left(
2^ji\pi\theta
\right)
\right|
\prod_{q=2}^n
\frac{
\left|
\exp
\left(
2^{p_q}i\pi\theta
\right)
+
\exp\left(\dfrac{2l+1}{2^{p_1-p_q}}i\pi\right)
\right|
}{
\left|1+
\exp\left(\dfrac{2l+1}{2^{p_1-p_q}}i\pi\right)
\right|
}
\\
& \leq &
\frac{1}{2^{p_1}}
\left(
\sum_{j=0}^{2^{p_1}-1}1
\right)
\prod_{q=2}^n
\frac{
\left|
1
+
\exp\left(
\dfrac{2l+1}{2^{p_1-p_q}}i\pi
-
2^{p_q}i\pi\theta
\right)
\right|
}{
\left|1+
\exp\left(\dfrac{2l+1}{2^{p_1-p_q}}i\pi\right)
\right|
}
\\%
\;\;\;\;\;\;\;\;\;\;\;
& = &
\prod_{q=2}^n
\frac{
2\left|
\cos
\left(
\dfrac{2l+1}{2^{p_1-p_q}}
\dfrac{\pi}{2}
-
\dfrac{2^{p_q}\pi\theta}{2}
\right)
\right|
}{
2\left|
\cos
\left(
\dfrac{2l+1}{2^{p_1-p_q}}
\dfrac{\pi}{2}
\right)
\right|
}
\;=\;
\prod_{q=2}^n
\frac{
\left|
\cos
\left(
\dfrac{2l+1}{2^{p_1-p_q+1}}\pi
-
\dfrac{2^{p_q}\pi\theta}{2}
\right)
\right|
}{
\left|
\cos
\left(
\dfrac{2l+1}{2^{p_1-p_q+1}}\pi
\right)
\right|
}
\,.
\end{eqnarray*}
Next, an application of Lemma~\ref{partitionTS} yields
\begin{eqnarray}\label{reducl1p1}
& &
\left|
l_1^{(N)}
\left(
z
\right)
\right|
\;\leq\;
\left[
\prod_{2\leq q\leq n,p_q\in T}
\times
\prod_{2\leq q\leq n,p_q\in S}
\right]
\frac{
\left|
\cos
\left(
\dfrac{2l+1}{2^{p_1-p_q+1}}\pi
-
\dfrac{2^{p_q}\pi\theta}{2}
\right)
\right|
}{
\left|
\cos
\left(
\dfrac{2l+1}{2^{p_1-p_q+1}}\pi
\right)
\right|
}
\,.
\end{eqnarray}

\medskip

We first deal with the product associated with $T$.
One has by Lemma~\ref{estimquotcospi/4} that
\begin{eqnarray*}
\prod_{2\leq q\leq n,p_q\in T}
\frac{
\left|
\cos
\left(
\dfrac{2l+1}{2^{p_1-p_q+1}}\pi
-
\dfrac{2^{p_q}\pi\theta}{2}
\right)
\right|
}{
\left|
\cos
\left(
\dfrac{2l+1}{2^{p_1-p_q+1}}\pi
\right)
\right|
}
& \leq &
\prod_{2\leq q\leq n,p_q\in T}
\left(
1+
\dfrac{2^{p_q}\pi|\theta|}{2}
\right)
\\
& \leq &
\prod_{q=2}^n
\left(
1+
\dfrac{2^{p_q}\pi|\theta|}{2}
\right)
\;\leq\;
\prod_{q=2}^n
\left(
1+
\dfrac{\pi/2}{2^{p_1-p_q}}
\right)
\,,
\end{eqnarray*}
the last estimate being justified by the condition that
$|\theta|\leq1/2^{p_1}$.
On the other hand, 
an immediate induction on
$q=1,\ldots,n$, that uses~(\ref{restrN}), gives that
\begin{eqnarray*}
p_1-p_q
& \geq &
q-1
\,.
\end{eqnarray*}
In addition, an application of Lemma~\ref{estimexp} yields
\begin{eqnarray}\nonumber
\prod_{2\leq q\leq n,p_q\in T}
\frac{
\left|
\cos
\left(
\dfrac{2l+1}{2^{p_1-p_q+1}}\pi
-
\dfrac{2^{p_q}\pi\theta}{2}
\right)
\right|
}{
\left|
\cos
\left(
\dfrac{2l+1}{2^{p_1-p_q+1}}\pi
\right)
\right|
}
& \leq &
\prod_{q=2}^n
\left(
1+
\frac{\pi/2}{2^{q-1}}
\right)
\;\leq\;
\prod_{q=2}^n
\exp
\left(
\frac{\pi/2}{2^{q-1}}
\right)
\\\label{estimprodp1pi/4}
& \leq &
\exp
\left(
\sum_{q\geq1}
\frac{\pi/2}{2^{q}}
\right)
\;=\;
\exp
\left(
\frac{\pi}{2}
\right)
\,.
\end{eqnarray}

\medskip

Now we deal with the product associated with $S$. Lemma~\ref{estimquotcospi/2}
leads to
\begin{eqnarray*}
\prod_{2\leq q\leq n,p_q\in S}
\frac{
\left|
\cos
\left(
\dfrac{2l+1}{2^{p_1-p_q+1}}\pi
-
\dfrac{2^{p_q}\pi\theta}{2}
\right)
\right|
}{
\left|
\cos
\left(
\dfrac{2l+1}{2^{p_1-p_q+1}}\pi
\right)
\right|
}
& \leq &
\prod_{2\leq q\leq n,p_q\in S}
\left(
1+\frac{2^{p_1}|\theta|\pi}{2^{1+s_{j_q+1}}}
\right)
\\
& \leq &
\prod_{2\leq q\leq n,p_q\in S}
\left(
1+\frac{\pi}{2^{1+s_{j_q+1}}}
\right)
\,,
\end{eqnarray*}
the last estimate being justified by the condition that
$|\theta|\leq1/2^{p_1}$.
On the other hand, 
the application $q\mapsto j_q$ is injective by Lemma~\ref{jqinj},
then so is 
\begin{eqnarray*}
\left\{2\leq q\leq n
,\,
p_q\in S
\right\}
& \xrightarrow[]{} &
\left\{2\leq j\leq 2L+1\right\}
\\
q
& \mapsto &
j_q+1
\,.
\end{eqnarray*}
It follows that
\begin{eqnarray*}
\prod_{2\leq q\leq n,p_q\in S}
\frac{
\left|
\cos
\left(
\dfrac{2l+1}{2^{p_1-p_q+1}}\pi
-
\dfrac{2^{p_q}\pi\theta}{2}
\right)
\right|
}{
\left|
\cos
\left(
\dfrac{2l+1}{2^{p_1-p_q+1}}\pi
\right)
\right|
}
& \leq &
\prod_{j=2}^{2L+1}
\left(
1+\frac{\pi}{2^{1+s_j}}
\right)
\,.
\end{eqnarray*}
In addition,  an immediate descending induction on
$j=2L+1,\ldots,2$, with~(\ref{recallbinaryaltsj}) 
and the convention~(\ref{s2L+1}) together yield
\begin{eqnarray*}
1+s_j
& \geq &
2L+1-j
\,.
\end{eqnarray*}
This fact and an application of Lemma~\ref{estimexp} lead to
\begin{eqnarray}\nonumber
\prod_{2\leq q\leq n,p_q\in S}
\frac{
\left|
\cos
\left(
\dfrac{2l+1}{2^{p_1-p_q+1}}\pi
-
\dfrac{2^{p_q}\pi\theta}{2}
\right)
\right|
}{
\left|
\cos
\left(
\dfrac{2l+1}{2^{p_1-p_q+1}}\pi
\right)
\right|
}
& \leq &
\prod_{j=2}^{2L+1}
\left(
1+\frac{\pi}{2^{2L+1-j}}
\right)
\;\leq\;
\prod_{j=2}^{2L+1}
\exp
\left(
\frac{\pi}{2^{2L+1-j}}
\right)
\\\label{estimprodp1pi/2}
& \leq &
\exp
\left(
\sum_{j\geq0}
\frac{\pi}{2^j}
\right)
\;=\;
\exp
\left(
2\pi
\right)
\,.
\end{eqnarray}

\bigskip

Finally, the estimates~(\ref{reducl1p1}), (\ref{estimprodp1pi/4}) and~(\ref{estimprodp1pi/2})
together yield
\begin{eqnarray*}
\left|
l_1^{(N)}
\left(
z
\right)
\right|
& \leq &
\exp
\left(
\frac{\pi}{2}
\right)
\times
\exp
\left(
2\pi
\right)
\;\leq\;
\exp\left(3\pi\right)
\,,
\end{eqnarray*}
and this proves the estimate~(\ref{estimatespecial}) in the case
$|\theta|\leq1/2^{p_1}$.

\bigskip

\subsection{Second case: $1/2^{p_1}\leq|\theta|\leq1/2^{p_{n}}$}\label{step2}

In this subsection we fix $q=q_{\theta}=2,\ldots,n$
(where we remind that $q_{\theta}$ is defined
by~(\ref{defpqtheta}) ) and
\begin{eqnarray}
\frac{1}{2^{p_{q_{\theta}-1}}}
\;\leq\;
|\theta|
\;\leq\;
\frac{1}{2^{p_{q_{\theta}}}}
\,.
\end{eqnarray}

Let us consider the following set
\begin{eqnarray}\label{defStheta}
S_{\theta}
\;:=\;
\left\{
q=q_{\theta},\ldots,n,\;p_q\in S
\,
\right\}
\;=\;
\left\{
q_{i}
\right\}_{1\leq i\leq m}
\,,
\end{eqnarray}
where (only if 
$S_{\theta}$ is non empty)
$m=\mbox{card}(S_{\theta})\geq1$ and
the numeration $(q_i)_{1\leq i\leq m}$ is chosen so that
\begin{eqnarray}\label{introqi}
q_{\theta}
\;\leq\;
q_1
\;<\;
\cdots
\;<\;
q_i
\;<\;
\cdots
\;<\;
q_m
\;\leq\;
n
\,.
\end{eqnarray}

Before dealing with this case, we want to give a couple of
auxiliary results which will be useful in all the subsection. The first one
deals with the $q$'s whose $p_q\in T$ (recall
definition~(\ref{defT}) ).

\begin{lemma}\label{estimprodqthetaT}

$q_{\theta}=2,\ldots,n$ being fixed, one has
for all $\w{q}\geq q_{\theta}$ and all
$|\theta|\leq1/2^{p_{q_{\theta}}}$,
\begin{eqnarray*}
\prod_{\w{q}\leq q\leq n,\,p_q\in T}
\frac{
\left|
\cos
\left(
\dfrac{2l+1}{2^{p_1-p_q+1}}\pi
-
\dfrac{2^{p_q}\pi\theta}{2}
\right)
\right|
}{
\left|
\cos
\left(
\dfrac{2l+1}{2^{p_1-p_q+1}}\pi
\right)
\right|
}
& \leq &
\exp\left(\pi\right)
\,.
\end{eqnarray*}

\end{lemma}

\begin{proof}

First, one can assume that
$q_{\theta}\leq\w{q}\leq n$ (indeed,
if $\w{q}>n$, then the above product is empty and
the estimate holds true).
For all $q=\w{q},\ldots,n$ such that
$p_q\in T$, one can apply Lemma~\ref{estimquotcospi/4}
to get (since
$\w{q}\geq q_{\theta}$)
\begin{eqnarray*}
\prod_{\w{q}\leq q\leq n,\,p_q\in T}
\frac{
\left|
\cos
\left(
\dfrac{2l+1}{2^{p_1-p_q+1}}\pi
-
\dfrac{2^{p_q}\pi\theta}{2}
\right)
\right|
}{
\left|
\cos
\left(
\dfrac{2l+1}{2^{p_1-p_q+1}}\pi
\right)
\right|
}
& \leq &
\prod_{\w{q}\leq q\leq n,\,p_q\in T}
\left(
1+
\frac{2^{p_q}\pi|\theta|}{2}
\right)
\;\;\;\;\;\;\;\;\;\;\;\;\;\;\;\;\;\;\;\;\;\;\;\;\;\;\;\;\;\;\;\;\;\;\;\;
\end{eqnarray*}
\begin{eqnarray*}
\;\;\;\;
& \leq &
\prod_{q_{\theta}\leq q\leq n,\,p_q\in T}
\left(
1+
\frac{2^{p_q}\pi|\theta|}{2}
\right)
\;\leq\;
\prod_{q=q_{\theta}}^n
\left(
1+
\frac{2^{p_q}\pi|\theta|}{2}
\right)
\;\leq\;
\prod_{q=q_{\theta}}^n
\left(
1+
\frac{\pi/2}{2^{p_{q_{\theta}}-p_q}}
\right)
\,,
\end{eqnarray*}
the last estimate being justified by the condition that
$|\theta|\leq1/2^{p_{q_{\theta}}}$.
On the other hand, an immediate induction on
$q=q_{\theta},\ldots,n$ that uses~(\ref{restrN}) yields
\begin{eqnarray*}
p_{q_{\theta}}-p_q
& \geq &
q-q_{\theta}
\,.
\end{eqnarray*}
It follows by applying Lemma~\ref{estimexp} that
\begin{eqnarray*}
\prod_{\w{q}\leq q\leq n,\,p_q\in T}
\frac{
\left|
\cos
\left(
\dfrac{2l+1}{2^{p_1-p_q+1}}\pi
-
\dfrac{2^{p_q}\pi\theta}{2}
\right)
\right|
}{
\left|
\cos
\left(
\dfrac{2l+1}{2^{p_1-p_q+1}}\pi
\right)
\right|
}
& \leq &
\prod_{q=q_{\theta}}^n
\left(
1+
\frac{\pi/2}{2^{q-q_{\theta}}}
\right)
\;\leq\;
\prod_{q=q_{\theta}}^n
\exp
\left(
\frac{\pi/2}{2^{q-q_{\theta}}}
\right)
\\
& \leq &
\exp
\left(
\sum_{j\geq0}
\frac{\pi/2}{2^j}
\right)
\;=\;
\exp\left(\pi\right)
\,,
\end{eqnarray*}
and the lemma is proved.

\end{proof}

The next one deals with the set
$S_{\theta}$.

\begin{lemma}\label{estimprodqthetaS}

$q_{\theta}=2,\ldots,n$ being fixed and
$q_1$ being defined by~(\ref{defStheta}), one has for all
$|\theta|\leq1/2^{p_{q_{\theta}}}$,
\begin{eqnarray*}
\prod_{q\in S_{\theta},q\neq q_1}
\frac{
\left|
\cos
\left(
\dfrac{2l+1}{2^{p_1-p_q+1}}\pi
-
\dfrac{2^{p_q}\pi\theta}{2}
\right)
\right|
}{
\left|
\cos
\left(
\dfrac{2l+1}{2^{p_1-p_q+1}}\pi
\right)
\right|
}
& \leq &
\exp
\left(
2\pi
\right)
\,.
\end{eqnarray*}

\end{lemma}

\begin{proof}

First, one can assume that neither $S_{\theta}$, nor
$S_{\theta}\setminus\left\{q_1\right\}$
are empty, i.e. $m\geq2$
in~(\ref{defStheta}), otherwise the assertion is obvious since
$\prod_{\emptyset}(\cdot)=1$.
We deduce from~(\ref{defStheta}) that
\begin{eqnarray*}
S_{\theta}
\setminus\left\{q_1\right\}
& = &
\left\{
q_i
\,,\;
i=2,\ldots,m
\right\}
\,.
\end{eqnarray*}
On the other hand, we remind by~(\ref{pi/2}) that
for all $i=2,\ldots,m$ (since every
$q_i\in S_{\theta}$ then $p_{q_i}\in S$),
\begin{eqnarray*}
p_{q_i}
& = &
p_1-s_{j_{q_i}}-1
\,.
\end{eqnarray*}
It follows by Lemma~\ref{estimquotcospi/2} that
\begin{eqnarray}\nonumber
\prod_{q\in S_{\theta},q\neq q_1}
\frac{
\left|
\cos
\left(
\dfrac{2l+1}{2^{p_1-p_q+1}}\pi
-
\dfrac{2^{p_q}\pi\theta}{2}
\right)
\right|
}{
\left|
\cos
\left(
\dfrac{2l+1}{2^{p_1-p_q+1}}\pi
\right)
\right|
}
& = &
\prod_{i=2}^m
\frac{
\left|
\cos
\left(
\dfrac{2l+1}{2^{p_1-p_{q_i}+1}}\pi
-
\dfrac{2^{p_{q_i}}\pi\theta}{2}
\right)
\right|
}{
\left|
\cos
\left(
\dfrac{2l+1}{2^{p_1-p_{q_i}+1}}\pi
\right)
\right|
}
\\\label{estimprodpqthetapi/21}
& \leq &
\prod_{i=2}^m
\left(
1+\frac{2^{p_1}|\theta|\pi}{2^{1+s_{j_{q_i}+1}}}
\right)
\,.
\end{eqnarray}
On the other hand, we know by Lemma~\ref{jqinj} that the application
$q\mapsto j_q$,
is (strictly) decreasing.
Since by~(\ref{introqi}),
$q_{i-1}<q_i$ for all
$i=2,\ldots,m$, it follows that
$j_{q_{i-1}}>j_{q_i}$, i.e.
\begin{eqnarray*}
j_{q_{i-1}}
& \geq &
j_{q_i}+1
\,.
\end{eqnarray*}
In addition, the application
$j\mapsto s_j$ being decreasing by~(\ref{recallbinaryaltsj}),
this yields 
\begin{eqnarray}
s_{j_{q_i}+1}
& \geq &
s_{j_{q_{i-1}}}
\,,
\end{eqnarray}
whose application at the estimate~(\ref{estimprodpqthetapi/21}) leads to
\begin{eqnarray}\nonumber
\prod_{q\in S_{\theta},q\neq q_1}
\frac{
\left|
\cos
\left(
\dfrac{2l+1}{2^{p_1-p_q+1}}\pi
-
\dfrac{2^{p_q}\pi\theta}{2}
\right)
\right|
}{
\left|
\cos
\left(
\dfrac{2l+1}{2^{p_1-p_q+1}}\pi
\right)
\right|
}
& \leq &
\prod_{i=2}^m
\left(
1+\frac{2^{p_1}|\theta|\pi}{2^{1+s_{j_{q_i}+1}}}
\right)
\\\label{estimprodpqthetapi/22}
& \leq &
\prod_{i=2}^m
\left(
1+\frac{2^{p_1}|\theta|\pi}{2^{1+s_{j_{q_{i-1}}}}}
\right)
\;=\;
\prod_{i=1}^{m-1}
\left(
1+\frac{2^{p_1}|\theta|\pi}{2^{1+s_{j_{q_{i}}}}}
\right)
\,.
\end{eqnarray}
Since $q_i$ satisfies~(\ref{pi/2}) for all
$i=1,\ldots,m-1$, it follows that
\begin{eqnarray*}
s_{j_{q_i}}+1
& = &
p_1-p_{q_i}
\,,
\end{eqnarray*}
then the estimate~(\ref{estimprodpqthetapi/22}) becomes by Lemma~\ref{estimexp}
\begin{eqnarray}\nonumber
\prod_{q\in S_{\theta},q\neq q_1}
\frac{
\left|
\cos
\left(
\dfrac{2l+1}{2^{p_1-p_q+1}}\pi
-
\dfrac{2^{p_q}\pi\theta}{2}
\right)
\right|
}{
\left|
\cos
\left(
\dfrac{2l+1}{2^{p_1-p_q+1}}\pi
\right)
\right|
}
& \leq &
\prod_{i=1}^{m-1}
\left(
1+\frac{2^{p_1}|\theta|\pi}{2^{p_1-p_{q_i}}}
\right)
\;\leq\;
\prod_{i=1}^{m-1}
\exp
\left(
2^{p_{q_i}}|\theta|\pi
\right)
\\\nonumber
& \leq &
\exp
\left(
\sum_{i=1}^{m}
2^{p_{q_i}}|\theta|\pi
\right)
\;=\;
\exp
\left(
\sum_{q\in S_{\theta}}
2^{p_q}|\theta|\pi
\right)
\\\label{estimprodpqthetapi/23}
& \leq &
\exp
\left(
\sum_{q=q_{\theta}}^n
2^{p_q}|\theta|\pi
\right)
\;\leq\;
\exp
\left(
\sum_{q=q_{\theta}}^n
\frac{\pi}{2^{p_{q_{\theta}}-p_q}}
\right)
\,,
\end{eqnarray}
the last estimate being justified by the condition that
$|\theta|\leq1/2^{p_{q_{\theta}}}$.
On the other hand, an immediate induction on
$q=q_{\theta},\ldots,n$, that uses~(\ref{restrN}) yields
\begin{eqnarray*}
p_{q_{\theta}}-p_q
& \geq &
q-q_{\theta}
\,.
\end{eqnarray*}
The estimate~(\ref{estimprodpqthetapi/23}) then becomes
\begin{eqnarray*}
\prod_{q\in S_{\theta},q\neq q_1}
\frac{
\left|
\cos
\left(
\dfrac{2l+1}{2^{p_1-p_q+1}}\pi
-
\dfrac{2^{p_q}\pi\theta}{2}
\right)
\right|
}{
\left|
\cos
\left(
\dfrac{2l+1}{2^{p_1-p_q+1}}\pi
\right)
\right|
}
& \leq &
\exp
\left(
\sum_{q=q_{\theta}}^n
\frac{\pi}{2^{q-q_{\theta}}}
\right)
\;\leq\;
\exp
\left(
\sum_{j\geq0}
\frac{\pi}{2^j}
\right)
\\
& = &
\exp
\left(
2\pi
\right)
\,,
\end{eqnarray*}
and the lemma is proved.

\end{proof}

\medskip

Now we can deal with the required estimate~(\ref{estimatespecial}) after fixing
$q_{\theta}=2,\ldots,n$ and
$\theta$ with
$1/2^{p_{q_{\theta}-1}}\leq|\theta|\leq1/2^{p_{q_{\theta}}}$.
First (since
$\theta\neq0$ then $z\neq1$),
one has by~(\ref{descrl1genznon0}), (\ref{restromega0}) and~(\ref{restrz}) that
\begin{eqnarray*}
\left|
l_1^{(N)}\left(z\right)
\right|
& = &
\frac{1}{2^{p_1}}
\frac{\left|z^{2^{p_1}}-1\right|}{\left|z-1\right|}
\times
\prod_{q=2}^n
\frac{
\left|z^{2^{p_q}}+\omega_0^{2^{p_q}}\right|
}{
\left|1+\omega_0^{2^{p_q}}\right|
}
\;\;\;\;\;\;\;\;\;\;\;\;\;\;\;\;\;\;\;\;\;\;\;\;\;\;\;\;\;\;\;\;\;\;\;\;\;\;\;\;
\;\;\;\;\;\;\;\;\;\;\;\;\;\;\;\;\;\;\;\;\;\;\;\;\;\;\;\;\;\;\;\;\;\;\;\;\;\;\;\;
\;\;\;\;\;\;\;\;\;\;\;\;
\end{eqnarray*}
\begin{eqnarray*}
& \leq &
\frac{1}{2^{p_1}}
\frac{\left|z^{2^{p_1}}\right|+1}{|z-1|}
\times
\prod_{q=2}^{q_{\theta}-1}
\frac{
\left|z^{2^{p_q}}\right|
+
\left|\omega_0^{2^{p_q}}\right|
}{
\left|1+\omega_0^{2^{p_q}}\right|
}
\times
\prod_{q=q_{\theta}}^n
\frac{
\left|z^{2^{p_q}}+\omega_0^{2^{p_q}}\right|
}{
\left|1+\omega_0^{2^{p_q}}\right|
}
\\\nonumber
& = &
\frac{1}{2^{p_1}}
\frac{2}{
\left|
\exp\left(i\pi\theta\right)-1
\right|
}
\prod_{q=2}^{q_{\theta}-1}
\frac{2}{
\left|
1+
\exp
\left(
\dfrac{2l+1}{2^{p_1-p_q}}
i\pi
\right)
\right|
}
\prod_{q=q_{\theta}}^n
\frac{
\left|
\exp
\left(
2^{p_q}i\pi\theta
\right)
+
\exp\left(\dfrac{2l+1}{2^{p_1-p_q}}i\pi\right)
\right|
}{
\left|1+
\exp\left(\dfrac{2l+1}{2^{p_1-p_q}}i\pi\right)
\right|
}
\,,
\end{eqnarray*}
then
\begin{eqnarray}\nonumber
\left|
l_1^{(N)}\left(z\right)
\right|
& \leq &
\frac{1}{2^{p_1}}
\frac{1}{|\sin(\pi\theta/2)|}
\prod_{q=2}^{q_{\theta}-1}
\frac{1}{
\left|
\cos
\left(
\dfrac{2l+1}{2^{p_1-p_q}}
\dfrac{\pi}{2}
\right)
\right|
}
\prod_{q=q_{\theta}}^n
\frac{
\left|
\cos
\left(
\dfrac{2l+1}{2^{p_1-p_q}}
\dfrac{\pi}{2}
-
\dfrac{2^{p_q}\pi\theta}{2}
\right)
\right|
}{
\left|
\cos
\left(
\dfrac{2l+1}{2^{p_1-p_q}}
\dfrac{\pi}{2}
\right)
\right|
}
\\\label{reducl1pqtheta15}
& \leq &
\frac{1}{2^{p_1}}
\frac{1}{|\theta|}
\;\times\;
\prod_{q=2}^{q_{\theta}-1}
\frac{1}{
\left|
\cos
\left(
\dfrac{2l+1}{2^{p_1-p_q}}
\dfrac{\pi}{2}
\right)
\right|
}
\;\times
\prod_{q=q_{\theta}}^n
\frac{
\left|
\cos
\left(
\dfrac{2l+1}{2^{p_1-p_q}}
\dfrac{\pi}{2}
-
\dfrac{2^{p_q}\pi\theta}{2}
\right)
\right|
}{
\left|
\cos
\left(
\dfrac{2l+1}{2^{p_1-p_q}}
\dfrac{\pi}{2}
\right)
\right|
}
\,,
\end{eqnarray}
the last estimate following by~(\ref{estimsinus1}) from Lemma~\ref{estimsinus}:
indeed, since $0<|\theta|\leq1$, then $\pi|\theta|/2\,\in\,\left]0,\dfrac{\pi}{2}\right]$
where the function $\sin$ is positive thus,
\begin{eqnarray}\label{minsinuspq}
& &
\left|
\sin\left(\frac{\pi\theta}{2}\right)
\right|
\;=\;
\left|
\sin
\left(
\frac{\pi|\theta|}{2}
\right)
\right|
\;=\;
\sin
\left(
\frac{\pi|\theta|}{2}
\right)
\;\geq\;
\frac{2}{\pi}\times\frac{\pi|\theta|}{2}
\;=\;
|\theta|
\,.
\end{eqnarray}

\medskip

Now  if we assume that 
\begin{eqnarray*}
S_{\theta}
& = &
\emptyset
\,,
\end{eqnarray*}
then 
$\left\{
p_q
\,,\;
q_{\theta}
\leq
q
\leq
n
\right\}
\subset
T$ 
by Lemma~\ref{partitionTS}
and
\begin{eqnarray}\nonumber
\left|
l_1^{(N)}\left(z\right)
\right|
& \leq &
\frac{1}{2^{p_1}|\theta|}
\times
\frac{1}{
\prod_{q=2}^{q_{\theta}-1}
\left|
\cos
\left(
\dfrac{2l+1}{2^{p_1-p_q+1}}\pi
\right)
\right|
}
\times
\;\;\;\;\;\;\;\;\;\;\;\;\;\;\;\;\;\;\;\;\;\;\;\;\;\;\;\;\;\;\;\;\;\;\;\;
\\\label{reducl1pqtheta1}
& &
\;\;\;\;\;\;\;\;\;\;\;\;\;\;\;\;\;\;\;\;\;\;\;\;\;\;\;\;\;\;\;\;\;\;\;\;
\times
\prod_{q_{\theta}\leq q\leq n,\,p_q\in T}
\frac{
\left|
\cos
\left(
\dfrac{2l+1}{2^{p_1-p_q+1}}\pi
-
\dfrac{2^{p_q}\pi\theta}{2}
\right)
\right|
}{
\left|
\cos
\left(
\dfrac{2l+1}{2^{p_1-p_q+1}}\pi
\right)
\right|
}
\,.
\end{eqnarray}
On the other hand, one has
\begin{eqnarray*}
\prod_{q=2}^{q_{\theta}-1}
\left|
\cos
\left(
\frac{2l+1}{2^{p_1-p_q+1}}\pi
\right)
\right|
& \geq &
\prod_{j=p_1-p_2+1}^{p_1-p_{q_{\theta}-1}+1}
\left|
\cos
\left(
\frac{2l+1}{2^{j}}\pi
\right)
\right|
\;\geq\;
\prod_{j=2}^{p_1-p_{q_{\theta}-1}+1}
\left|
\cos
\left(
\frac{2l+1}{2^{j}}\pi
\right)
\right|
\\
& = &
\left|
\prod_{j=1}^{p_1-p_{q_{\theta}-1}}
\cos
\left(
\frac{(2l+1)\pi/2}{2^{j}}
\right)
\right|
,
\end{eqnarray*}
since $p_1-p_2+1\geq2$ and
any term of the involved products is not greater than $1$.
An application of Lemma~\ref{prodcos} (with
$\alpha=(2l+1)\pi/2\notin\pi\Z$ and 
$m=p_1-p_{q_{\theta}-1}\geq p_1-p_1=0$
since $q_{\theta}\geq2$) yields
\begin{eqnarray}\label{prodpqtheta-1}
& &
\prod_{q=2}^{q_{\theta}-1}
\left|
\cos
\left(
\frac{2l+1}{2^{p_1-p_q+1}}\pi
\right)
\right|
\,\geq\,
\frac{
\left|
\sin
\left(
(2l+1)\pi/2
\right)
\right|
}{
2^{p_1-p_{q_{\theta}-1}}
\left|
\sin
\left(
\dfrac{(2l+1)\pi/2}{2^{p_1-p_{q_{\theta}-1}}}
\right)
\right|
}
\,\geq\,
\frac{1}{2^{p_1-p_{q_{\theta}-1}}}
\,.
\end{eqnarray}

It follows that~(\ref{reducl1pqtheta1}),
(\ref{prodpqtheta-1}) and Lemma~\ref{estimprodqthetaT}
(with the choice of
$\w{q}=q_{\theta}$)
together yield
\begin{eqnarray}\nonumber
\left|
l_1^{(N)}\left(z\right)
\right|
& \leq &
\frac{1}{2^{p_1}|\theta|}
\times
2^{p_1-p_{q_{\theta}-1}}
\times
\exp(\pi)
\;=\;
\frac{\exp(\pi)}
{2^{p_{q_{\theta}-1}}|\theta|}
\;\leq\;
\exp\left(\pi\right)
\,,
\end{eqnarray}
the last estimate being justified by the condition that
$|\theta|\geq1/2^{p_{q_{\theta}-1}}$,
and this proves the required assertion in the case
$S_{\theta}=\emptyset$.

\medskip

Now we assume that
$S_{\theta}$ is non empty.
In particular, we can deal with $q_1$, i.e.
\begin{eqnarray*}
\frac{
\left|
\cos
\left(
\dfrac{2l+1}{2^{p_1-p_{q_1}+1}}\pi
-
\dfrac{2^{p_{q_1}}\pi\theta}{2}
\right)
\right|
}{
\left|
\cos
\left(
\dfrac{2l+1}{2^{p_1-p_{q_1}+1}}\pi
\right)
\right|
}
& \leq &
\;\;\;\;\;\;\;\;\;\;\;\;\;\;\;\;\;\;\;\;\;\;\;\;\;\;\;\;\;\;\;\;\;\;\;\;
\;\;\;\;\;\;\;\;\;\;\;\;\;\;\;\;\;\;\;\;\;\;\;\;\;\;\;\;\;\;\;\;\;\;\;\;
\end{eqnarray*}
\begin{eqnarray*}
& \leq &
\frac{
\left|
\cos
\left(
\dfrac{2l+1}{2^{p_1-p_{q_1}+1}}\pi
\right)
\right|
\left|
\cos
\left(
\dfrac{2^{p_{q_1}}\pi\theta}{2}
\right)
\right|
+
\left|
\sin
\left(
\dfrac{2l+1}{2^{p_1-p_{q_1}+1}}\pi
\right)
\right|
\left|
\sin
\left(
\dfrac{2^{p_{q_1}}\pi\theta}{2}
\right)
\right|
}{
\left|
\cos
\left(
\dfrac{2l+1}{2^{p_1-p_{q_1}+1}}\pi
\right)
\right|
}
\\
& \leq &
\frac{
\left|
\cos
\left(
\dfrac{2l+1}{2^{p_1-p_{q_1}+1}}\pi
\right)
\right|
+
\left|
\sin
\left(
\dfrac{2^{p_{q_1}}\pi\theta}{2}
\right)
\right|
}{
\left|
\cos
\left(
\dfrac{2l+1}{2^{p_1-p_{q_1}+1}}\pi
\right)
\right|
}
\\
& \leq &
\frac{2
\max
\left[
\left|
\cos
\left(
\dfrac{2l+1}{2^{p_1-p_{q_1}+1}}\pi
\right)
\right|
\,,\,
\left|
\sin
\left(
\dfrac{2^{p_{q_1}}\pi\theta}{2}
\right)
\right|
\right]
}{
\left|
\cos
\left(
\dfrac{2l+1}{2^{p_1-p_{q_1}+1}}\pi
\right)
\right|
}
\,,
\end{eqnarray*}
then by applying~(\ref{estimsinus2}) from Lemma~\ref{estimsinus},
\begin{eqnarray}\label{estimprodpqthetapi/26}
\;\;\;\;\;
& &
\frac{
\left|
\cos
\left(
\dfrac{2l+1}{2^{p_1-p_{q_1}+1}}\pi
-
\dfrac{2^{p_{q_1}}\pi\theta}{2}
\right)
\right|
}{
\left|
\cos
\left(
\dfrac{2l+1}{2^{p_1-p_{q_1}+1}}\pi
\right)
\right|
}
\,\leq\,
\frac{2
\max
\left[
\left|
\cos
\left(
\dfrac{2l+1}{2^{p_1-p_{q_1}+1}}\pi
\right)
\right|
\,,\,
\dfrac{2^{p_{q_1}}\pi|\theta|}{2}
\right]
}{
\left|
\cos
\left(
\dfrac{2l+1}{2^{p_1-p_{q_1}+1}}\pi
\right)
\right|
}
\,.
\end{eqnarray}

\medskip

Now, let assume that
\begin{eqnarray*}
\max
\left[
\left|
\cos
\left(
\dfrac{2l+1}{2^{p_1-p_{q_1}+1}}\pi
\right)
\right|
\,,\,
\dfrac{2^{p_{q_1}}\pi|\theta|}{2}
\right]
& = &
\left|
\cos
\left(
\dfrac{2l+1}{2^{p_1-p_{q_1}+1}}\pi
\right)
\right|
\,,
\end{eqnarray*}
then~(\ref{estimprodpqthetapi/26}) becomes
\begin{eqnarray}\label{estimprodpqthetapi/26case1}
\frac{
\left|
\cos
\left(
\dfrac{2l+1}{2^{p_1-p_{q_1}+1}}\pi
-
\dfrac{2^{p_{q_1}}\pi\theta}{2}
\right)
\right|
}{
\left|
\cos
\left(
\dfrac{2l+1}{2^{p_1-p_{q_1}+1}}\pi
\right)
\right|
}
& \leq &
2
\,.
\end{eqnarray}
It follows by Lemma~\ref{estimprodqthetaT} 
(with the choice of
$\w{q}=q_{\theta}$) and Lemma~\ref{estimprodqthetaS} that
\begin{eqnarray}\nonumber
& &
\prod_{q=q_{\theta}}^n
\frac{
\left|
\cos
\left(
\dfrac{2l+1}{2^{p_1-p_q+1}}\pi
-
\dfrac{2^{p_q}\pi\theta}{2}
\right)
\right|
}{
\left|
\cos
\left(
\dfrac{2l+1}{2^{p_1-p_q+1}}\pi
\right)
\right|
}
\;=\;
\frac{
\left|
\cos
\left(
\dfrac{2l+1}{2^{p_1-p_{q_1}+1}}\pi
-
\dfrac{2^{p_{q_1}}\pi\theta}{2}
\right)
\right|
}{
\left|
\cos
\left(
\dfrac{2l+1}{2^{p_1-p_{q_1}+1}}\pi
\right)
\right|
}
\times
\;\;\;\;\;\;\;
\\\nonumber
&  &
\;\;\;\;\;\;\;\;\;
\times
\left[
\prod_{q_{\theta}\leq q\leq n,\,p_q\in T}
\times
\prod_{q\in S_{\theta},q\neq q_1}
\right]
\frac{
\left|
\cos
\left(
\dfrac{2l+1}{2^{p_1-p_q+1}}\pi
-
\dfrac{2^{p_q}\pi\theta}{2}
\right)
\right|
}{
\left|
\cos
\left(
\dfrac{2l+1}{2^{p_1-p_q+1}}\pi
\right)
\right|
}
\\\label{estimprodpqthetapi/27case1}
\;\;\;\;\;\;
& &
\leq\;
2
\times
\exp\left(\pi\right)
\times
\exp\left(2\pi\right)
\;\leq\;
\pi\exp\left(3\pi\right)
\,.
\end{eqnarray}
Then the estimates~(\ref{reducl1pqtheta15}),
(\ref{prodpqtheta-1}) and~(\ref{estimprodpqthetapi/27case1}) together yield
\begin{eqnarray*}
\left|
l_1^{(N)}\left(z\right)
\right|
& \leq &
\frac{1}{2^{p_1}|\theta|}
\times
2^{p_1-p_{q_{\theta}-1}}
\times
\pi\exp\left(3\pi\right)
\;=\;
\frac{\pi\exp\left(3\pi\right)}{2^{p_{q_{\theta}-1}}|\theta|}
\;\leq\;
\pi\exp\left(3\pi\right)
\,,
\end{eqnarray*}
the last estimate being justified by the condition that
$|\theta|\geq1/2^{p_{q_{\theta}-1}}$,
and this proves the required assertion in this case.

\medskip

The remaining case is the one for which
\begin{eqnarray}\label{case2}
\max
\left[
\left|
\cos
\left(
\dfrac{2l+1}{2^{p_1-p_{q_1}+1}}\pi
\right)
\right|
\,,\,
\dfrac{2^{p_{q_1}}\pi|\theta|}{2}
\right]
& = &
\frac{2^{p_{q_1}}\pi|\theta|}{2}
\,.
\end{eqnarray}
We prove an estimate similar to~(\ref{reducl1pqtheta15}) with
$q_{\theta}$ replaced with $q_1$.
Since $z\neq1$, one still has by~(\ref{descrl1genznon0}), (\ref{restromega0}) and~(\ref{restrz}) that
\begin{eqnarray*}
\left|
l_1^{(N)}\left(z\right)
\right|
& = &
\frac{\left|z^{2^{p_1}}-1\right|}{2^{p_1}\left|z-1\right|}
\prod_{q=2}^n
\frac{
\left|z^{2^{p_q}}+\omega_0^{2^{p_q}}\right|
}{
\left|1+\omega_0^{2^{p_q}}\right|
}
\,\leq\,
\frac{2}{2^{p_1}|z-1|}
\prod_{q=2}^{q_1-1}
\frac{2}
{
\left|1+\omega_0^{2^{p_q}}\right|
}
\prod_{q=q_1}^n
\frac{
\left|z^{2^{p_q}}+\omega_0^{2^{p_q}}\right|
}{
\left|1+\omega_0^{2^{p_q}}\right|
}
\\
& = &
\frac{1}{2^{p_1}}
\frac{2}{
\left|
\exp\left(i\pi\theta\right)-1
\right|
}
\times
\prod_{q=2}^{q_1-1}
\frac{2}{
\left|
1+
\exp
\left(
\dfrac{2l+1}{2^{p_1-p_q}}
i\pi
\right)
\right|
}
\;\times
\\
& &
\times
\,
\frac{
\left|
\exp
\left(
2^{p_{q_1}}i\pi\theta
\right)
+
\exp\left(\dfrac{2l+1}{2^{p_1-p_{q_1}}}i\pi\right)
\right|
}{
\left|1+
\exp\left(\dfrac{2l+1}{2^{p_1-p_{q_1}}}i\pi\right)
\right|
}
\prod_{q=q_1+1}^n
\frac{
\left|
\exp
\left(
2^{p_q}i\pi\theta
\right)
+
\exp\left(\dfrac{2l+1}{2^{p_1-p_q}}i\pi\right)
\right|
}{
\left|1+
\exp\left(\dfrac{2l+1}{2^{p_1-p_q}}i\pi\right)
\right|
}
\,,
\end{eqnarray*}
then
\begin{eqnarray*}
\left|
l_1^{(N)}\left(z\right)
\right|
& \leq &
\frac{1}{2^{p_1}}
\frac{1}{|\sin(\pi\theta/2)|}
\times
\prod_{q=2}^{q_1-1}
\frac{1}{
\left|
\cos
\left(
\dfrac{2l+1}{2^{p_1-p_q}}
\dfrac{\pi}{2}
\right)
\right|
}
\\
& &
\times\;
\frac{
\left|
\cos
\left(
\dfrac{2l+1}{2^{p_1-p_{q_1}+1}}\pi
-
\dfrac{2^{p_{q_1}}\pi\theta}{2}
\right)
\right|}{
\left|
\cos
\left(
\dfrac{2l+1}{2^{p_1-p_{q_1}}}
\dfrac{\pi}{2}
\right)
\right|
}
\times
\prod_{q=q_1+1}^n
\frac{
\left|
\cos
\left(
\dfrac{2l+1}{2^{p_1-p_q}}
\dfrac{\pi}{2}
-
\dfrac{2^{p_q}\pi\theta}{2}
\right)
\right|
}{
\left|
\cos
\left(
\dfrac{2l+1}{2^{p_1-p_q}}
\dfrac{\pi}{2}
\right)
\right|
}
\,,
\end{eqnarray*}
On the other hand, one has
by~(\ref{estimprodpqthetapi/26})
and~(\ref{case2}) that
\begin{eqnarray*}
\frac{
\left|
\cos
\left(
\dfrac{2l+1}{2^{p_1-p_{q_1}+1}}\pi
-
\dfrac{2^{p_{q_1}}\pi\theta}{2}
\right)
\right|
}{
\left|
\cos
\left(
\dfrac{2l+1}{2^{p_1-p_{q_1}+1}}\pi
\right)
\right|
}
& \leq &
\frac{2
\max
\left[
\left|
\cos
\left(
\dfrac{2l+1}{2^{p_1-p_{q_1}+1}}\pi
\right)
\right|
\,,\,
\dfrac{2^{p_{q_1}}\pi|\theta|}{2}
\right]
}{
\left|
\cos
\left(
\dfrac{2l+1}{2^{p_1-p_{q_1}+1}}\pi
\right)
\right|
}
\\
& \leq &
\frac{2
\times
2^{p_{q_1}}\pi|\theta|/2
}{
\left|
\cos
\left(
\dfrac{2l+1}{2^{p_1-p_{q_1}+1}}\pi
\right)
\right|
}
\,.
\end{eqnarray*}
It follows by also applying~(\ref{minsinuspq}) that
\begin{eqnarray}\nonumber
\left|
l_1^{(N)}\left(z\right)
\right|
& \leq &
\frac{1}{2^{p_1}}
\frac{1}{|\theta|}
\times
\frac{1}{
\prod_{q=2}^{q_1}
\left|
\cos
\left(
\dfrac{2l+1}{2^{p_1-p_q+1}}\pi
\right)
\right|
}
\;\times
\\\label{reducl1pqtheta1tris}
& &
\times\;
2^{p_{q_1}}\pi|\theta|
\times
\prod_{q=q_1+1}^n
\frac{
\left|
\cos
\left(
\dfrac{2l+1}{2^{p_1-p_q+1}}\pi
-
\dfrac{2^{p_q}\pi\theta}{2}
\right)
\right|
}{
\left|
\cos
\left(
\dfrac{2l+1}{2^{p_1-p_q+1}}\pi
\right)
\right|
}
\,.
\end{eqnarray}

Now one has
\begin{eqnarray*}
\prod_{q=2}^{q_1}
\left|
\cos
\left(
\frac{2l+1}{2^{p_1-p_q+1}}\pi
\right)
\right|
& \geq &
\prod_{j=p_1-p_2+1}^{p_1-p_{q_1}+1}
\left|
\cos
\left(
\frac{2l+1}{2^{j}}\pi
\right)
\right|
\;\geq\;
\prod_{j=2}^{p_1-p_{q_1}+1}
\left|
\cos
\left(
\frac{2l+1}{2^{j}}\pi
\right)
\right|
\\
& = &
\left|
\prod_{j=1}^{p_1-p_{q_1}}
\cos
\left(
\frac{(2l+1)\pi/2}{2^{j}}
\right)
\right|
,
\end{eqnarray*}
since $p_1-p_2+1\geq2$ and 
any term of the involved products is not greater than $1$.
An application of Lemma~\ref{prodcos} (with
$\alpha=(2l+1)\pi/2\notin\pi\Z$ and 
$m=p_1-p_{q_1}\geq p_1-p_2>0$
since $q_1\geq q_{\theta}\geq2$) yields
\begin{eqnarray}\label{prodpq1}
& &
\prod_{q=2}^{q_1}
\left|
\cos
\left(
\frac{2l+1}{2^{p_1-p_q+1}}\pi
\right)
\right|
\,\geq\,
\frac{
\left|
\sin
\left(
(2l+1)\pi/2
\right)
\right|
}{
2^{p_1-p_{q_1}}
\left|
\sin
\left(
\dfrac{(2l+1)\pi/2}{2^{p_1-p_{q_1}}}
\right)
\right|
}
\,\geq\,
\frac{1}{2^{p_1-p_{q_1}}}
\,.
\end{eqnarray}

Next, since
by~(\ref{defStheta}),
$\left\{q_1+1\leq q\leq n
\,,\;
p_q\in \,S\right\}
=\left\{q_i,\,i=2,\ldots,m\right\}
=S_{\theta}\setminus\left\{q_1\right\}$,
one can apply Lemma~\ref{partitionTS}
(for all 
$q=q_1+1,\ldots,n$), Lemma~\ref{estimprodqthetaT} 
(with the choice of
$\w{q}=q_1+1>q_1\geq q_{\theta}$) and Lemma~\ref{estimprodqthetaS} to get
\begin{eqnarray*}
\prod_{q=q_1+1}^n
\frac{
\left|
\cos
\left(
\dfrac{2l+1}{2^{p_1-p_q+1}}\pi
-
\dfrac{2^{p_q}\pi\theta}{2}
\right)
\right|
}{
\left|
\cos
\left(
\dfrac{2l+1}{2^{p_1-p_q+1}}\pi
\right)
\right|
}
& = &
\;\;\;\;\;\;\;\;\;\;\;\;\;\;\;\;\;\;\;\;\;\;\;\;\;\;\;\;\;\;\;\;\;\;\;\;
\;\;\;\;\;\;\;\;\;\;\;\;\;\;\;\;\;\;\;\;\;\;\;\;\;\;\;\;\;\;\;\;\;\;\;\;
\end{eqnarray*}
\begin{eqnarray}\nonumber
\;\;\;\;\;\;\;\;\;\;\;\;\;\;\;\;\;\;\;\;\;\;\;\;
& = &
\left[
\prod_{q_1+1\leq q\leq n,\,p_q\in T}
\times
\prod_{q\in S_{\theta},q\neq q_1}
\right]
\frac{
\left|
\cos
\left(
\dfrac{2l+1}{2^{p_1-p_q+1}}\pi
-
\dfrac{2^{p_q}\pi\theta}{2}
\right)
\right|
}{
\left|
\cos
\left(
\dfrac{2l+1}{2^{p_1-p_q+1}}\pi
\right)
\right|
}
\\\label{estimprodpqthetapi/27case2}
& \leq &
\exp\left(\pi\right)
\times
\exp\left(2\pi\right)
\;=\;
\exp\left(3\pi\right)
\,.
\end{eqnarray}

Finally, the estimates~(\ref{reducl1pqtheta1tris}), 
(\ref{prodpq1}) and~(\ref{estimprodpqthetapi/27case2})
together yield
\begin{eqnarray*}
\left|
l_1^{(N)}\left(z\right)
\right|
& \leq &
\frac{1}{2^{p_1}|\theta|}
\times
2^{p_1-p_{q_1}}
\times
2^{p_{q_1}}\pi|\theta|
\times
\exp\left(3\pi\right)
\;=\;
\pi\exp\left(3\pi\right)
\,,
\end{eqnarray*}
and this proves the required estimate~(\ref{estimatespecial}) in the case
$1/2^{p_{q_{\theta}-1}}\leq|\theta|\leq1/2^{p_{q_{\theta}}}$.

\medskip

The assertion being true for all $\theta$ with
$1/2^{p_{q_{\theta}-1}}\leq|\theta|\leq1/2^{p_{q_{\theta}}}$,
and all
$q_{\theta}=2,\ldots,n$,
the required estimate~(\ref{estimatespecial}) is then proved for all
$\theta$ with
$1/2^{p_1}\leq|\theta|\leq1/2^{p_n}$.

\bigskip

\subsection{Third case: $|\theta|\geq1/2^{p_n}$}\label{step3}

Now we fix $\theta$ with
$1/2^{p_n}\leq|\theta|\leq1$. In particular, $z\neq1$ then 
one has by~(\ref{descrl1genznon0}), 
(\ref{restromega0}) and~(\ref{restrz}) that
\begin{eqnarray}\nonumber
\left|
l_1^{(N)}
\left(
z
\right)
\right|
& = &
\frac{1}{2^{p_1}}
\frac{\left|z^{2^{p_1}}-1\right|}{\left|z-1\right|}
\times
\prod_{q=2}^n
\frac{
\left|z^{2^{p_q}}+\omega_0^{2^{p_q}}\right|
}{
\left|1+\omega_0^{2^{p_q}}\right|
}
\;\leq\;
\frac{1}{2^{p_1}}
\frac{
\left|z^{2^{p_q}}\right|+1
}{|z-1|}
\prod_{q=2}^n
\frac{
\left|z^{2^{p_q}}\right|+
\left|\omega_0^{2^{p_q}}\right|
}{
\left|1+\omega_0^{2^{p_q}}\right|
}
\\\nonumber
& = &
\frac{1}{2^{p_1}}
\frac{
2
}{
\left|
\exp\left(i\pi\theta\right)-1
\right|
}
\prod_{q=2}^n
\frac{2}{
\left|
1+
\exp
\left(
\dfrac{2l+1}{2^{p_1-p_q}}
i\pi
\right)
\right|
}
\\\nonumber
& = &
\frac{1}{2^{p_1}|\sin(\pi\theta/2)|}
\prod_{q=2}^n
\frac{1}{
\left|
\cos
\left(
\dfrac{2l+1}{2^{p_1-p_q}}
\dfrac{\pi}{2}
\right)
\right|
}
\,\leq\,
\frac{1}{2^{p_1}|\theta|}
\prod_{q=2}^n
\frac{1}{
\left|
\cos
\left(
\dfrac{2l+1}{2^{p_1-p_q+1}}\pi
\right)
\right|
}
\,,
\end{eqnarray}
the last estimate being an application of~(\ref{minsinuspq}), i.e.
$\left|\sin\left(\pi\theta/2\right)\right|\geq|\theta|$
(that is still valid since
$\pi|\theta|/2\,\in\,\left]0,\pi/2\right]$).
Since we have
$|\theta|\geq1/2^{p_n}$, it follows that
\begin{eqnarray}\label{reducl1pn}
\left|
l_1^{(N)}
\left(
z
\right)
\right|
& \leq &
\frac{1}{2^{p_1-p_n}}
\prod_{q=2}^n
\frac{1}{
\left|
\cos
\left(
\dfrac{2l+1}{2^{p_1-p_q+1}}\pi
\right)
\right|
}
\,.
\end{eqnarray}
On the other hand,
for all $q=2,\ldots,n$, one has by~(\ref{restrN})
\begin{eqnarray*}
2
\;\leq\;
p_1-p_2+1
\;\leq\;
p_1-p_q+1
\;\leq\;
p_1-p_n+1
\,,
\end{eqnarray*}
then
\begin{eqnarray*}
\prod_{q=2}^n
\left|
\cos
\left(
\frac{2l+1}{2^{p_1-p_q+1}}\pi
\right)
\right|
& \geq &
\prod_{j=2}^{p_1-p_n+1}
\left|
\cos
\left(
\frac{2l+1}{2^j}\pi
\right)
\right|
\;=\;
\left|
\prod_{j=1}^{p_1-p_n}
\cos
\left(
\frac{(2l+1)\pi/2}{2^j}
\right)
\right|
\,,
\end{eqnarray*}
since any term of the involved products is not greater than $1$.
An application of Lemma~\ref{prodcos} (with
$\alpha=(2l+1)\pi/2\notin\pi\Z$ and $m=p_1-p_n$) yields
\begin{eqnarray}\label{prodpn}
& &
\prod_{q=2}^n
\left|
\cos
\left(
\frac{2l+1}{2^{p_1-p_q+1}}\pi
\right)
\right|
\;\geq\;
\frac{
\left|
\sin
\left(
(2l+1)\pi/2
\right)
\right|
}{
2^{p_1-p_n}
\left|
\sin
\left(
\dfrac{(2l+1)\pi/2}{2^{p_1-p_n}}
\right)
\right|
}
\;\geq\;
\frac{1}{2^{p_1-p_n}}
\,.
\end{eqnarray}
Thus, the estimates~(\ref{reducl1pn}) and~(\ref{prodpn})
together lead to
\begin{eqnarray*}
\left|
l_1^{(N)}
\left(z\right)
\right|
& \leq &
\frac{1}{2^{p_1-p_n}}
\times
2^{p_1-p_n}
\;=\;
1
\,,
\end{eqnarray*}
and this proves the required estimate~(\ref{estimatespecial}) in this
last case, and finally completes its whole proof for all $|\theta|\leq1$.

\bigskip

\section{Proof of Theorem~\ref{unifestim} in the general case}

\subsection{A couple of auxiliary results for the proof of Theorem~\ref{unifestim}}

In the previous section, we have considered the special case of
$l_1^{(N)}$ (i.e. the FLIP associated with
$z_1=1$) and $l=1,\ldots,2^{p_1}-1$ for $\omega_0$.
The following result gives a way to extend the required estimate~(\ref{estimatespecial})
for every FLIP associated with
$z_k$ where $k=2,\ldots,2^{p_1}$. We remind from~(\ref{restrN}) that,
in all the subsection, we consider $N$ defined as follows:
\begin{eqnarray}\label{recallrestrN}
N
\;=\;
2^{p_1}+\cdots+2^{p_n}
& \mbox{with} &
p_1>p_2>\cdots>p_n\geq0
\mbox{ and }
n\geq2
\,.
\end{eqnarray}

We will also use in all the following the simplified notation (for any function $f$
defined on the closed disk):
\begin{eqnarray*}
\left\|f\right\|_{\overline{\D}}
& := &
\sup_{z\in\overline{\D}}
\left|f(z)\right|
\,.
\end{eqnarray*}

We can begin with the following result.

\begin{lemma}\label{changetoz1}

Let $\mathcal{L}_N$ be the $N$-Leja section of any
fixed Leja sequence $\mathcal{L}$ (that starts at
$z_1=1$). 
For all $k=1,\ldots,2^{p_1}$, let us consider 
$z_k\in\mathcal{L}$ (i.e.
$z_k\in\Omega_{2^{p_1}}$ is any
$2^{p_1}$-th root of the unity) and its associated
FLIP
$l_k^{(N)}$.
Then for all $z\in\C$,
\begin{eqnarray}\label{changetoz11}
\left|
l_k^{(N)}
\left(z_kz\right)
\right|
& = &
\left|
\widetilde{l}_1^{(N)}(z)
\right|
\,,
\end{eqnarray}
where
$\widetilde{l}_1^{(N)}$ is the FLIP associated with
$\widetilde{z}_1=1$ of (possibly) another
$N$-Leja section $\widetilde{\mathcal{L}}_N$
(i.e. the $N$-Leja section of a possibly another Leja sequence
that also starts at $z_1=1$).

In particular,
\begin{eqnarray}\label{changetoz12}
\left\|
l_k^{(N)}
\right\|_{\overline{\D}}
& = &
\left\|
\widetilde{l}_1^{(N)}
\right\|_{\overline{\D}}
\,.
\end{eqnarray}

\end{lemma}

\begin{proof}

Let be $z_k\in\Omega_{2^{p_1}}$ (i.e.
a $2^{p_1}$-th root of the unity). We know by~(\ref{descrlkgen})
from Lemma~\ref{lkgeneral}
that there is $\omega_0$ a $2^{p_1}$-th root
of $-1$ such that, for all $z\in\C$ with $z\neq z_k$,
\begin{eqnarray*}
\left|
l_k^{(N)}
\left(z\right)
\right|
& = &
\frac{1}{2^{p_1}}
\frac{\left|z^{2^{p_1}}-1\right|}{\left|z-z_k\right|}
\times
\prod_{q=2}^n
\frac{
\left|z^{2^{p_q}}+\omega_0^{2^{p_q}}\right|
}{
\left|z_k^{2^{p_q}}+\omega_0^{2^{p_q}}\right|
}
\,,
\end{eqnarray*}
then for all $z\neq1$,
\begin{eqnarray*}
\left|
l_k^{(N)}
\left(z_kz\right)
\right|
& = &
\frac{1}{2^{p_1}}
\frac{\left|z_k^{2^{p_1}}z^{2^{p_1}}-1\right|}{\left|zz_k-z_k\right|}
\times
\prod_{q=2}^n
\frac{
\left|z_k^{2^{p_q}}z^{2^{p_q}}+\omega_0^{2^{p_q}}\right|
}{
\left|z_k^{2^{p_q}}+\omega_0^{2^{p_q}}\right|
}
\\
& = &
\frac{1}{2^{p_1}}
\frac{\left|1\times
z^{2^{p_1}}-1\right|}
{\left|z_k\right|\left|z-1\right|}
\times
\prod_{q=2}^n
\frac{\left|z_k^{2^{p_q}}\right|
\left|z^{2^{p_q}}
+
\left(\omega_0/z_k\right)^{2^{p_q}}\right|
}{
\left|z_k^{2^{p_q}}\right|
\left|1+(\omega_0/z_k)^{2^{p_q}}\right|
}
\\
& = &
\frac{1}{2^{p_1}}
\frac{\left|z^{2^{p_1}}-1\right|}
{1\times\left|z-1\right|}
\times
\prod_{q=2}^n
\frac{
\left|z^{2^{p_q}}
+
\left(\omega_0/z_k\right)^{2^{p_q}}\right|
}{
\left|1+(\omega_0/z_k)^{2^{p_q}}\right|
}
\,.
\end{eqnarray*}
It follows that for all $z\neq1$,
\begin{eqnarray*}
\left|
l_k^{(N,\omega_0)}
\left(z_kz\right)
\right|
& = &
\left|
l_1^{(N,\omega_1)}
\left(z\right)
\right|
\,,
\end{eqnarray*}
where
\begin{eqnarray*}
\omega_1
& := &
\frac{\omega_0}{z_k}
\end{eqnarray*}
is still a $2^{p_1}$-th root of $-1$
and 
$l_k^{(N,\omega_0)}$
(resp., $l_1^{(N,\omega_1)}$)
is the FLIP associated with
$z_k$ (resp., $\w{z_1}=1$) and
the $2^{p_1}$-th root $\omega_0$
(resp., $\omega_1$).
We remind as specified by~(\ref{omega0toLN}) from
Remark~\ref{omega0} that
the data of $\Omega_{2^{p_1}}$ and $\omega_1$
conversely gives a $N$-Leja section (that starts at $\w{z_1}=1$)
and whose first FLIP is exactly
$l_1^{(N,\omega_1)}$.
This proves~(\ref{changetoz11})
by setting
$\widetilde{l}_1^{(N)}:=l_1^{(N,\omega_1)}$, 
and~(\ref{changetoz12})
follows since
$|z_k|=1$.

\end{proof}

We finish this subsection with the following result that is
the proof of the required estimate~(\ref{estimatespecial})
for $l_1^{(N)}$ and the unique case of
$\omega_0$ that was not considered in the previous section.

\begin{lemma}\label{l=0}

Let fix $k=1$ (i.e. $z_k=z_1=1$) and
$l=0$ in~(\ref{restromega0}), i.e.
\begin{eqnarray*}
\omega_0
& = &
\exp
\left(
\frac{i\pi}{2^{p_1}}
\right)
\,.
\end{eqnarray*}
Then
\begin{eqnarray*}
\left\|
l_1^{(N)}
\right\|_{\overline{\D}}
& \leq &
\frac{\pi}{2}
\,.
\end{eqnarray*}
In particular, the estimate~(\ref{estimatespecial}) is still valid in this case.

\end{lemma}

\begin{proof}

For all $z\in\C$ with
$|z|\leq1$, one has by~(\ref{descrl1genz0}) that
\begin{eqnarray*}
\left|
l_1^{(N)}
(z)
\right|
& = &
\frac{1}{2^{p_1}}
\left|
\sum_{j=0}^{2^{p_1}-1}
z^j
\right|
\times
\prod_{q=2}^n
\frac{
\left|z^{2^{p_q}}+\omega_0^{2^{p_q}}\right|
}{
\left|1+\omega_0^{2^{p_q}}\right|
}
\;\leq\;
\frac{1}{2^{p_1}}
\sum_{j=0}^{2^{p_1}-1}
\left|
z^j
\right|
\times
\prod_{q=2}^n
\frac{
\left|z^{2^{p_q}}\right|
+
\left|\omega_0^{2^{p_q}}\right|
}{
\left|1+\omega_0^{2^{p_q}}\right|
}
\\
& \leq &
1
\times
\prod_{q=2}^n
\frac{2
}{
\left|1+
\exp
\left(
i\pi/2^{p_1-p_q}
\right)
\right|
}
\;=\;
\prod_{q=2}^n
\;
\frac{2}{
2
\left|
\cos
\left(
\dfrac{1}{2^{p_1-p_q}}
\dfrac{\pi}{2}
\right)
\right|
}
\,,
\end{eqnarray*}
then
\begin{eqnarray}\label{l=01}
\sup_{|z|\leq1}
\left|
l_1^{(N)}
(z)
\right|
& \leq &
\frac{1}{\prod_{q=2}^n
\left|
\cos
\left(
\pi/2^{p_1-p_q+1}
\right)
\right|}
\,.
\end{eqnarray}
On the other hand, for all
$q=2,\ldots,n$, one has by~(\ref{recallrestrN}) that
\begin{eqnarray*}
2
\;\leq\;
p_1-p_2+1
\;\leq\;
p_1-p_q+1
\;\leq\;
p_1-p_n+1
\;\leq\;
p_1+1
\,,
\end{eqnarray*}
then
\begin{eqnarray*}
\prod_{q=2}^n
\left|
\cos
\left(
\pi/2^{p_1-p_q+1}
\right)
\right|
& \geq &
\prod_{j=p_1-p_2+1}^{p_1-p_n+1}
\left|
\cos
\left(
\pi/2^j
\right)
\right|
\;\geq\;
\prod_{j=2}^{p_1+1}
\left|
\cos
\left(
\pi/2^j
\right)
\right|
\,,
\end{eqnarray*}
since any term in the involved products is not greater than $1$.
It follows by Lemma~\ref{prodcos}
with $\alpha=\pi/2$ ($\notin\pi\Z$) and $m=p_1$
that
\begin{eqnarray*}
\prod_{q=2}^n
\left|
\cos
\left(
\pi/2^{p_1-p_q+1}
\right)
\right|
\;\geq\;
\left|
\prod_{j=1}^{p_1}
\cos
\left(
\frac{\pi/2}{2^j}
\right)
\right|
\,=\,
\frac{
\left|
\sin\left(\pi/2\right)
\right|
}{
2^{p_1}
\left|
\sin\left(
\dfrac{\pi/2}{2^{p_1}}
\right)
\right|
}
\,\geq\,
\frac{1}{2^{p_1}
\times
\dfrac{\pi/2}{2^{p_1}}}
\,=\,
\frac{2}{\pi}
\,,
\end{eqnarray*}
where the last estimate is justified by~(\ref{estimsinus2}) from
Lemma~\ref{estimsinus}, 
and the proof is finished by~(\ref{l=01}).

\end{proof}

\bigskip

\subsection{Proof of Theorem~\ref{unifestim}}

Before giving the proof of Theorem~\ref{unifestim}, we need the following
preliminar result in which we deal with the other FLIPs
associated with $z_k\in\mathcal{L}_N\setminus\Omega_{2^{p_1}}$.

\begin{lemma}\label{prelimpfthm}

Let $\mathcal{L}_N$ be a $N$-Leja section whose first point
$z_1$ starts at $1$.
There is an $\left(N-2^{p_1}\right)$-Leja section $\w{\mathcal{L}}_{N-2^{p_1}}$ that also
starts at $\w{z}_1=1$, with the following property:
for all $k=2^{p_1}+1,\ldots,N$, there is 
$k'$ with
$1\leq k'\leq N-2^{p_1}$ such that
\begin{eqnarray*}
\left\|l_k^{(N)}(z)\right\|_{\overline{\D}}
& \leq &
\left\|
\w{l}_{k'}^{(N-2^{p_1})}
\right\|_{\overline{\D}}
\,,
\end{eqnarray*}
where $\w{l}_{k'}^{(N-2^{p_1})}$ is the FLIP
associated with $\w{z_{k'}}\in\w{\mathcal{L}}_{N-2^{p_1}}$.

In addition, the correspondence
\begin{eqnarray*}
k,\,2^{p_1}+1\leq k\leq N
& \mapsto &
k',\,1\leq k'\leq N-2^{p_1}
\,,
\end{eqnarray*}
is well-defined and one-to-one.

\end{lemma}

\begin{proof}

First, let us consider
$k=2^{p_1}+1,\ldots, N$ and the FLIP
$l_k^{(N)}$ associated with
$z_k$. $\mathcal{L}_N$ being a
$N$-Leja section that starts at $z_1=1$, one necessarily has
by Theorem~5 from~\cite{bialascalvi} (or~(\ref{splitleja}) ) 
that $z_k\notin\Omega_{2^{p_1}}$, i.e.
$z_k$ is a $2^{p_1}$-th root of $-1$. It follows that
\begin{eqnarray}\nonumber
\left|l_k^{(N)}(z)\right|
& = &
\left|
\prod_{z_j\in\mathcal{L}_N,z_j\neq z_k}
\frac{z-z_j}{z_k-z_j}
\right|
\\\label{pfthm1}
& = &
\left|
\prod_{z_j\in\Omega_{2^{p_1}}}
\frac{z-z_j}{z_k-z_j}
\right|
\times
\left|
\prod_{z_j\in\mathcal{L}_N\setminus\Omega_{2^{p_1}},z_j\neq z_k}
\frac{z-z_j}{z_k-z_j}
\right|
\,.
\end{eqnarray}

On the one hand, one has for all $|z|\leq1$,
\begin{eqnarray}\label{pfthm2}
\left|
\prod_{z_j\in\Omega_{2^{p_1}}}
\frac{z-z_j}{z_k-z_j}
\right|
\;=\;
\frac{
\left|
z^{2^{p_1}}-1
\right|
}{
\left|
z_k^{2^{p_1}}
-1
\right|
}
& \leq &
\frac{|z|^{2^{p_1}}+1}{|-1-1|}
\;=\;
\frac{2}{2}
\;=\;
1
\,,
\end{eqnarray}
since $z_k^{2^{p_1}}=-1$.

On the other hand, one has (again by Theorem~5 from~\cite{bialascalvi},
or~(\ref{splitleja}) ) that
\begin{eqnarray*}
\mathcal{L}_N\setminus\Omega_{2^{p_1}}
& = &
\omega_1\widetilde{\mathcal{L}}_{N-2^{p_1}}
\,,
\end{eqnarray*}
where
$\omega_1$ is a $2^{p_1}$-th root of $-1$,
$\widetilde{\mathcal{L}}_{N-2^{p_1}}$ is the
$\left(N-2^{p_1}\right)$-section of (maybe) another
Leja sequence
$\widetilde{\mathcal{L}}
=
\left\{
\w{z_1},\w{z_2},\ldots,\w{z_j},\ldots
\right\}$
with
$\w{z_1}=1$, and the above equality is meant as sets.
In particular,
$z_k\in\omega_1\w{\mathcal{L}}_{N-2^{p_1}}$ can be written as
$z_k=\omega_1\w{z_{k'}}$
with
$1\leq k'\leq N-2^{p_1}$. This proves that
the correspondence:
\begin{eqnarray*}
z_k,\,2^{p_1}+1\leq k\leq N
& \mapsto &
\w{z_{k'}},\,1\leq k'\leq N-2^{p_1}
\end{eqnarray*}
(then $k\mapsto k'$ as well),
 is well-defined and injective.
Since
$\mbox{card}\left(\mathcal{L}_N\setminus\Omega_{2^{p_1}}\right)
=N-2^{p_1}=
\mbox{card}\left(\omega_1\w{\mathcal{L}}_{N-2^{p_1}}\right)$, it is also
one-to-one.

In particular, this leads to
\begin{eqnarray*}
\left|
\prod_{z_j\in\mathcal{L}_N\setminus\Omega_{2^{p_1}},z_j\neq z_k}
\frac{z-z_j}{z_k-z_j}
\right|
& = &
\;\;\;\;\;\;\;\;\;\;\;\;\;\;\;\;\;\;\;\;\;\;\;\;\;\;\;\;\;\;\;\;\;\;\;\;
\;\;\;\;\;\;\;\;\;\;\;\;\;\;\;\;\;\;\;\;\;\;\;\;\;\;\;\;\;\;\;\;\;\;\;\;
\;\;\;\;\;\;\;\;\;\;\;\;\;\;\;\;
\end{eqnarray*}
\begin{eqnarray}\nonumber
\;\;\;\;\;\;\;\;\;\;\;\;\;\;\;\;
& = &
\left|
\prod_{z_j\in\omega_1\widetilde{\mathcal{L}}_{N-2^{p_1}},z_j\neq z_k}
\frac{z-z_j}{z_k-z_j}
\right|
\;=\;
\left|
\prod_{\w{z_j}\in\widetilde{\mathcal{L}}_{N-2^{p_1}},\w{z_j}\neq \w{z_{k'}}}
\frac{z-\omega_1\w{z_j}}{\omega_1\w{z_{k'}}-\omega_1\w{z_j}}
\right|
\\\label{pfthm3}
& = &
\left|
\prod_{\w{z_j}\in\widetilde{\mathcal{L}}_{N-2^{p_1}},\w{z_j}\neq \w{z_{k'}}}
\frac{z/\omega_1-\w{z_j}}{\w{z_{k'}}-\w{z_j}}
\right|
\;=\;
\left|
\w{l}_{k'}^{(N-2^{p_1})}
\left(
\frac{z}{\omega_1}
\right)
\right|
\,,
\end{eqnarray}
where
$\w{l}_{k'}^{(N-2^{p_1})}$ is the FLIP
associated with
$\w{z_{k'}}$ from the
$\left(N-2^{p_1}\right)$-Leja section
$\w{\mathcal{L}}_{N-2^{p_1}}$.

Finally, the estimates~(\ref{pfthm1}), (\ref{pfthm2})
and~(\ref{pfthm3}) together yield
\begin{eqnarray*}
\sup_{|z|\leq1}
\left|l_k^{(N)}(z)\right|
& \leq &
1\times
\sup_{|z|\leq1}
\left|
\w{l}_{k'}^{(N-2^{p_1})}
\left(
\frac{z}{\omega_1}
\right)
\right|
\,,
\end{eqnarray*}
and the lemma is proved since
$\left|\omega_1\right|=1$.

\end{proof}

Now we can finally give the proof of Theorem~\ref{unifestim}.
We then consider $\mathcal{L}_N$
the $N$-section of any fixed Leja sequence, where
\begin{eqnarray}\label{recallNgen}
& &
N
\;=\;
2^{p_1}+2^{p_2}+\cdots+2^{p_n}
\,\mbox{ with }\,
p_1>p_2>\cdots>p_n\geq0
\;\mbox{ and }\;
n\geq1
\,.
\end{eqnarray}

\begin{proof}

First, by the symmetry of the disk, we can wlog assume that (the first Leja point) $z_1=1$.
Next, by the maximum modulus principle, it suffices to prove the required estimate for all
$|z|=1$, i.e. $z=\exp(i\pi\theta)$ with
$\theta\in\,]-1,1]$.
The proof is by induction on $n\geq1$, where $n$ is defined
from~(\ref{recallNgen}). 

\medskip

The special case of $n=1$ means that $N=2^{p_1}$ with $p_1\geq0$.
Then $\mathcal{L}_{2^{p_1}}=\Omega_{2^{p_1}}$
by Theorem~5 from~\cite{bialascalvi}, 
and~(\ref{estim13}) from Lemma~\ref{estim1} yields
for all $k=1,\ldots,2^{p_1}$,
\begin{eqnarray*}
\sup_{z\in\overline{\D}}
\left|
l_k^{(2^{p_1})}(z)
\right|
& = &
1
\,.
\end{eqnarray*}
An alternative argument is that the set $\Omega_{2^{p_1}}$ is a
$2^{p_1}$-Fekete set for the unit disk 
(as specified by~(\ref{unifboundfekete}) from the Introduction,
see also~\cite{fekete1}). Thus, all the FLIPs
are bounded by $1$.

\medskip

Now let be $N$ with $n\geq2$
(i.e. $2^{p_1}<N<2^{p_1+1}$), let be
$\mathcal{L}_N$ and consider the associated $\omega_0$
defined from Lemma~\ref{lkgeneral} and that can be written as follows:
\begin{eqnarray*}
\omega_0
\;=\;
\exp\left(
\frac{2l+1}{2^{p_1}}i\pi
\right)
& \mbox{ with } &
l=0,\ldots,2^{p_1}-1
\,.
\end{eqnarray*}
We first prove the theorem for the FLIP
$l_1^{(N)}$ associated with
$z_1=1$,
and we fix $|z|=1$. If $l=0$, then the theorem is a consequence of
Lemma~\ref{l=0}.
Otherwise, $1\leq l\leq 2^{p_1}-1$ then the assertion is a consequence of
what has been proved in Section~\ref{specialessential}.
The theorem is then true for
$l_1^{(N)}$ with any $N$-Leja section $\mathcal{L}_N$.

\medskip

Next, if $2\leq k\leq2^{p_1}$, then
$z_k\in\Omega_{2^{p_1}}$ (i.e. $z_k$ is a
$2^{p_1}$-th root of the unity,
see Theorem~1 from~\cite{calviphung1},
or~(\ref{splitleja}) ). An application of~(\ref{changetoz12})
from Lemma~\ref{changetoz1}
yields
\begin{eqnarray*}
\left\|
l_k^{(N)}
\right\|_{\overline{\D}}
& = &
\left\|
\widetilde{l}_1^{(N)}
\right\|_{\overline{\D}}
\,,
\end{eqnarray*}
where $\widetilde{l}_1^{(N)}$ is the FLIP associated
with $\widetilde{z}_1=1$ from (possibly) another $N$-Leja section
$\widetilde{\mathcal{L}}_N$. Since the theorem is valid
for
$l_1^{(N)}$ and any $N$-Leja section, it follows that it holds true for
$l_k^{(N)}$. Thus, it is proved for
all $k=1,\ldots,2^{p_1}$ (and for any $N$-Leja section).

\medskip

Lastly, let us consider
$k=2^{p_1}+1,\ldots, N$ and the FLIP
$l_k^{(N)}$ associated with
$z_k$. $\mathcal{L}_N$ being a
$N$-Leja section that starts at $z_1=1$, necessarily
$z_k\notin\Omega_{2^{p_1}}$, i.e.
$z_k$ is a $2^{p_1}$-th root of $-1$
(see Theorem~5 from~\cite{bialascalvi}). An application of
Lemma~\ref{prelimpfthm} gives that
\begin{eqnarray*}
\left\|l_k^{(N)}(z)\right\|_{\overline{\D}}
& \leq &
\left\|
\w{l}_{k'}^{(N-2^{p_1})}
\right\|_{\overline{\D}}
\,,
\end{eqnarray*}
where $\w{l}_{k'}^{(N-2^{p_1})}$ is the FLIP
associated with $\w{z_{k'}}\in\w{\mathcal{L}}_{N-2^{p_1}}$
and
$\w{\mathcal{L}}_{N-2^{p_1}}$ 
is an $\left(N-2^{p_1}\right)$-Leja section that also
starts at $\w{z}_1=1$. Since by~(\ref{recallNgen}),
\begin{eqnarray*}
N-2^{p_1}
& = &
2^{p_2}+2^{p_3}+\cdots+2^{p_n}
\;=\;
\sum_{q=1}^{n-1}2^{p_{q+1}}
\,,
\end{eqnarray*}
it follows that the induction hypothesis can be applied to
$\w{\mathcal{L}}_{N-2^{p_1}}$ with
$n-1$, and the above inequality becomes
\begin{eqnarray*}
\left\|l_k^{(N)}(z)\right\|_{\overline{\D}}
& \leq &
\left\|
\w{l}_{k'}^{(N-2^{p_1})}
\right\|_{\overline{\D}}
\;\leq\;
\pi\exp(3\pi)
\,.
\end{eqnarray*}
This finally achieves the induction and the whole proof of the theorem.

\end{proof}

\bigskip

\section{On the special case of $N=2^{p}-1$}\label{secspecialN}

As it has been pointed in the Introduction, the bound from Theorem~\ref{unifestim} may not be optimal.
That is why we want to consider here the special case of
$N=2^p-1$ where the associated estimate can be considerably improved. Indeed,
it is first proved (see Proposition~\ref{unifestimspecialN} below) that
$\sup_{z\in\overline{\D}}\left|l_k^{\left(2^p-1\right)}(z)\right|\leq2$
for all $k=1,\ldots,2^p-1$ (and for exceptional values of $k$,
the bound cannot be better than $4/\pi$).

On the other hand, numerical simulations let us think that for
{\em almost} $k=1,\ldots,2^p-1$, the associated bound
for 
$l_k^{\left(2^p-1\right)}$
is {\em almost} $1$, i.e. any $\left(2^p-1\right)$-Leja section
is {\em almost} a $\left(2^p-1\right)$-Fekete set for the unit disk.
We may explain what is meant by using {\em almost} and this is
specified by the following result that has been mentioned in the Introduction.

\begin{theorem}\label{specialN}

Let $\mathcal{L}$ be any Leja sequence for the unit disk and
let us consider for all
$p\geq1$, $\left\{l_k^{\left(2^p-1\right)}\right\}_{1\leq k\leq 2^p-1}$ the associated
family of FLIPs. Then
\begin{eqnarray}\label{specialN1}
\lim_{p\rightarrow+\infty}
\left[
\frac{1}{2^p-1}
\sum_{k=1}^{2^p-1}
\sup_{z\in\overline{\D}}
\left|
l_k^{\left(2^p-1\right)}(z)
\right|
\right]
& = &
1
\,.
\end{eqnarray}

More precisely, for all $p\geq2$,
\begin{eqnarray}\label{specialN2}
2^p-1
& < &
\sum_{k=1}^{2^p-1}
\sup_{z\in\overline{\D}}
\left|
l_k^{\left(2^p-1\right)}(z)
\right|
\;\leq\;
\left(
1+\varepsilon(1/p)\right)
\times
\left(2^p-1\right)
\,,
\end{eqnarray}
where
$\lim_{p\rightarrow+\infty}
\varepsilon(1/p)=0$.

\end{theorem}

We first notice that,
except for an asymptotically negligible number of values for $k=1,\ldots,N$,
the FLIPs are asymptotically bounded by $1$. 

Next, as an application, it is
a heuristic confirmation of Theorem~8 from~\cite{calviphung1} where
it is proved that
$\Lambda_{2^p-1}=2^p-1$ for all $p\geq1$.
Although we think that this way will not allow us to prove the conjecture
from~\cite{calviphung1} that
$\Lambda_{N}\leq N$ for all $N\geq1$.
If we wanted to get this last result as an application of Theorem~\ref{specialN}, 
we should then prove a better estimate,
like for example
$\sum_{k=1}^{2^p-1}
\sup_{z\in\overline{\D}}
\left|
l_k^{\left(2^p-1\right)}(z)
\right|
\leq
2^p-1$.
Unfortunately, and as it could be suspected, this cannot
be possible as specified by the left-hand side of~(\ref{specialN2}).

Finally, this makes us think that this way will not allow us to prove the conjecture
from~\cite{calviphung1} that
$\Lambda_{N}\leq N$ for all $N\geq1$ (although
we think that the worst values of $\Lambda_{N}$ appear for 
$N=2^p-1$, as  it has been pointed out in the last part 
of~\cite{chkifa1}, {\em Numerical illustration},
p. 198--199).

Before giving the proof, we first remind the notation
$N=2^{p_1+1}-1
=2^{p_1}+2^{p_1-1}+\cdots+2+1$, i.e.
\begin{eqnarray}\label{dataspecialN}
n\;=\;p_1+1
& \mbox{and} &
p_q\;=\;p_1-q+1
\,,\;
\mbox{ for all }\,
q\;=\;1,\ldots,p_1+1
\,.
\end{eqnarray}

We will first deal with a subfamily of FLIPs:
for all $l=0,\ldots,2^{p_1}-1$, we consider the
$2^{p_1}$-th root of $-1$,
\begin{eqnarray}\label{defomegal}
\omega_l
& := &
\exp
\left(
\frac{2l+1}{2^{p_1}}i\pi
\right)
\,,
\end{eqnarray}
and for all $z\neq1,\omega_l$, we set
\begin{eqnarray}\label{defl1omegal}
l_{\omega_l}^{(N)}(z)
& := &
\frac{
1-\omega_l
}{
2^{p_1+1}
}
\times
\frac{
z^{2^{p_1+1}}-1
}{
(z-1)
\,
\left(z-\omega_l\right)
}
\,.
\end{eqnarray}

As we will see in the following, $l_{\omega_l}^{(N)}$ is indeed a FLIP
(at least in modulus) for all $l=0,\ldots,2^{p_1}-1$. 
We then want to begin with preliminar results
for the estimate of the $l_{\omega_l}^{(N)}$'s for
{\em almost} every $l=1,\ldots,2^{p_1}-1$ (as it will be specified below).

\subsection{A preliminar estimate for $l_{\omega_l}^{(N)}$ and for {\em almost} every $l$}

We begin with an estimate of $\left|1-\omega_l\right|$.

\begin{lemma}\label{estim1-omegal}

Let fix
$\varepsilon_0>$ small enough such that 
$\varepsilon_02^{p_1}\leq\left(1-\varepsilon_0\right)2^{p_1}-1$
for all $p_1\geq1$
(for example, if $\varepsilon_0\leq1/4$). Then for all $l$ with
$\varepsilon_02^{p_1}\leq l\leq\left(1-\varepsilon_0\right)2^{p_1}-1$
(where
$\omega_l=\exp\left[(2l+1)i\pi/2^{p_1}\right]$), one has
\begin{eqnarray*}
\left|
1-\omega_l
\right|
& \geq &
4\varepsilon_0
\,.
\end{eqnarray*}

\end{lemma}

\begin{proof}

We know from hypothesis about $l$ that
\begin{eqnarray*}
\frac{2l+1}{2^{p_1}}
& \geq &
\frac{
2\times\varepsilon_02^{p_1}+1
}{2^{p_1}}
\;\geq\;
\frac{\varepsilon_02^{p_1+1}}{2^{p_1}}
\;=\;2\varepsilon_0
\,.
\end{eqnarray*}
Similarly,
\begin{eqnarray*}
\frac{2l+1}{2^{p_1}}
& \leq &
\frac{
2\times
\left(
\left(1-\varepsilon_0\right)2^{p_1}-1
\right)+1
}{2^{p_1}}
\;=\;
\frac{
\left(1-\varepsilon_0\right)2^{p_1+1}-1
}{2^{p_1}}
\;\leq\;
2\left(1-\varepsilon_0\right)
\,.
\end{eqnarray*}
Thus,
\begin{eqnarray*}
0
\;<\;
\pi\varepsilon_0
\;\leq\;
\frac{2l+1}{2^{p_1}}
\frac{\pi}{2}
\;\leq\;
\pi\left(1-\varepsilon_0\right)
\;<\;
\pi
\,.
\end{eqnarray*}
In particular,
$\dfrac{2l+1}{2^{p_1}}
\dfrac{\pi}{2}
\in
\left[
0,\pi
\right]$ then
\begin{eqnarray*}
\left|
\sin
\left(
\frac{2l+1}{2^{p_1}}
\frac{\pi}{2}
\right)
\right|
& = &
\sin
\left(
\frac{2l+1}{2^{p_1}}
\frac{\pi}{2}
\right)
\\
& \geq &
\min
\left[
\sin
\left(
\pi\varepsilon_0
\right)
\,,\,
\sin
\left(
\pi-
\pi\varepsilon_0
\right)
\right]
\;=\;
\sin
\left(
\pi\varepsilon_0
\right)
\,.
\end{eqnarray*}
It follows that
\begin{eqnarray*}
\left|
1-\omega_l
\right|
\;=\;
\left|
1-\exp\left(
\frac{2l+1}{2^{p_1}}i\pi
\right)
\right|
\;=\;
2\left|
\sin
\left(
\frac{2l+1}{2^{p_1}}
\frac{\pi}{2}
\right)
\right|
\;\geq\;
2
\sin
\left(
\pi\varepsilon_0
\right)
\,.
\end{eqnarray*}
On the other hand, since
$0<\pi\varepsilon_0\leq\pi/4<\pi/2$,
an application
of~(\ref{estimsinus1}) from Lemma~\ref{estimsinus}
yields
\begin{eqnarray*}
\left|
1-\omega_l
\right|
\;\geq\;
2
\sin
\left(
\pi\varepsilon_0
\right)
\;\geq\;
2\times\frac{2}{\pi}
\times
\pi\varepsilon_0
\;=\;
4\varepsilon_0
\,,
\end{eqnarray*}
and this proves the lemma.

\end{proof}

The following result gives an estimate of
$l_{\omega_l}^{(N)}$ for the $z$'s which are close to
$1$ or $\omega_l$.

\begin{lemma}\label{estimlomegalmax}

Let fix
$\varepsilon_0>$ small enough such that 
$\varepsilon_02^{p_1}\leq\left(1-\varepsilon_0\right)2^{p_1}-1$
for all $p_1\geq1$
(for example, if $\varepsilon_0\leq1/4$). Then for all $l$ with
$\varepsilon_02^{p_1}\leq l\leq\left(1-\varepsilon_0\right)2^{p_1}-1$ 
and all
$z\in\overline{\D}$ such that
$|z-1|\leq\varepsilon_0^2$
or
$|z-\omega_l|\leq\varepsilon_0^2$
(where
$\omega_l=\exp\left[(2l+1)i\pi/2^{p_1}\right]$, see~(\ref{defomegal}) ), one has
\begin{eqnarray*}
\left|
l_{\omega_l}^{(N)}(z)
\right|
& \leq &
\frac{1}{
1-
\varepsilon_0/4
}
\,.
\end{eqnarray*}

\end{lemma}

\begin{proof}

First, let notice that $l_{\omega_l}^{(N)}$
is a polynomial since by~(\ref{defl1omegal}), for all $z\in\C$,
\begin{eqnarray*}
l_{\omega_l}^{(N)}(z)
& = &
\frac{
1-\omega_l
}{
2^{p_1+1}
}
\times
\frac{
\left(z^{2^{p_1}}\right)^2-1
}{
(z-1)
\left(z-\omega_l\right)
}
\;=\;
\frac{
1-\omega_l
}{
2^{p_1+1}
}
\times
\frac{
\left(z^{2^{p_1}}-1\right)
\times
\left(z^{2^{p_1}}+1\right)
}{
(z-1)
\,
\left(z-\omega_l\right)
}
\\
& = &
\frac{
1-\omega_l
}{
2^{p_1+1}
}
\,
\frac{z^{2^{p_1}}-1}{z-1}
\,
\frac{z^{2^{p_1}}-\omega_l^{2^{p_1}}}{z-\omega_l}
\;=\;
\frac{
1-\omega_l
}{
2^{p_1+1}
}
\left(
\sum_{j=0}^{2^{p_1}-1}z^j
\right)
\left(
\sum_{j=0}^{2^{p_1}-1}
\omega_l^{2^{p_1}-1-j}z^j
\right)
\end{eqnarray*}
(because $\omega_l^{2^{p_1}}=-1$).
In particular, $l_{\omega_l}^{(N)}$ is continuous
then it suffices to prove the required estimate for all
$z\neq1,\omega_l$.

Next, if $|z-1|\leq\varepsilon_0^2$, then by Lemma~\ref{estim1-omegal}, 
\begin{eqnarray*}
\frac{
\left|
z-1
\right|
}{
\left|
1-\omega_l
\right|
}
& \leq &
\frac{\varepsilon_0^2}{
4\varepsilon_0
}
\;=\;
\frac{\varepsilon_0}{4}
\end{eqnarray*}
thus,
\begin{eqnarray*}
\frac{
\left|1-\omega_l\right|
}{
\left|z-\omega_l\right|
}
& \leq &
\frac{
\left|1-\omega_l\right|
}{
|1-\omega_l|-|z-1|
}
\;=\;
\frac{1}{
1-
|z-1|/|1-\omega_l|
}
\;\leq\;
\frac{1}{
1-
\varepsilon_0/4
}
\,.
\end{eqnarray*}
Since $z\neq1,\omega_l$, and $|z|\leq1$, it follows by~(\ref{defl1omegal}) that
\begin{eqnarray*}
\left|
l_{\omega_l}^{(N)}(z)
\right|
& = &
\frac{
1}{
2^{p_1+1}
}
\left|
\frac{
z^{2^{p_1+1}}-1
}{
z-1
}
\right|
\frac{
\left|1-\omega_l\right|
}{
\left|z-\omega_l\right|
}
\;=\;
\frac{1}{2^{p_1+1}}
\left|
\sum_{j=0}^{2^{p_1+1}-1}
z^j
\right|
\frac{
\left|1-\omega_l\right|
}{
\left|z-\omega_l\right|
}
\\
& \leq &
\frac{1}{2^{p_1+1}}
\times
2^{p_1+1}
\times
\frac{
\left|1-\omega_l\right|
}{
\left|z-\omega_l\right|
}
\;\leq\;
1
\times
\frac{1}{
1-
\varepsilon_0/4
}
\,,
\end{eqnarray*}
and this proves the lemma in this case.

Now let assume that
$\left|z-\omega_l\right|\leq\varepsilon_0^2$.
Since $\omega_l^{2^{p_1+1}}=1$, notice that for all
$z\neq1,\omega_l$,
\begin{eqnarray*}
\left|
l_{\omega_l}^{(N)}(z)
\right|
& = &
\frac{
\left|1/\omega_l-1\right|
}{
2^{p_1+1}
}
\frac{
\left|
z^{2^{p_1+1}}/\omega_l^{2^{p_1+1}}-1
\right|
}{
\left|z/\omega_l-1/\omega_l\right|
\,
\left|z/\omega_l-1\right|
}
\;=\;
\left|
l_{1/\omega_l}^{(N)}
\left(
z/\omega_l
\right)
\right|
\,,
\end{eqnarray*}
where $l_{1/\omega_l}^{(N)}$ is associated with
$1/\omega_l$ (that is still a $2^{p_1}$-th root of
$-1$).
In addition,
\begin{eqnarray*}
\frac{1}{\omega_l}
& = &
\exp
\left(
-\frac{2l+1}{2^{p_1}}i\pi
\right)
\;=\;
\exp
\left(2i\pi
-\frac{2l+1}{2^{p_1}}i\pi
\right)
\;=\;
\exp
\left(
\frac{2^{p_1+1}-2l-1}{2^{p_1}}i\pi
\right)
\\
& = &
\exp
\left(
\frac{2\times
\left(
2^{p_1}-l-1
\right)+1
}{2^{p_1}}i\pi
\right)
\;=\;
\exp
\left(
\frac{2\w{l}+1
}{2^{p_1}}i\pi
\right)
\;=\;
\omega_{\w{l}}
\,,
\end{eqnarray*}
with
$\w{l}:=2^{p_1}-l-1$. On the other hand,
\begin{eqnarray*}
\w{l}
\;=\;
2^{p_1}-l-1
\;\geq\;
2^{p_1}-
\left(\left(1-\varepsilon_0\right)2^{p_1}-1\right)
-1
\;=\;
2^{p_1}-2^{p_1}
+\varepsilon_02^{p_1}+1-1
\;=\;
\varepsilon_02^{p_1}
\end{eqnarray*}
and
\begin{eqnarray*}
\w{l}
\;=\;
2^{p_1}-l-1
\;\leq\;
2^{p_1}-\varepsilon_02^{p_1}-1
\;=\;
\left(1-\varepsilon_0\right)2^{p_1}-1
\,,
\end{eqnarray*}
i.e. the integer $\w{l}$ associated with $\omega_{\w{l}}=1/\omega_l$ satisfies
the same conditions as $l$. Since
$\left|z/\omega_l-1\right|=\left|z-\omega_l\right|\leq\varepsilon_0^2$,
one can apply the previous case with
$l_{1/\omega_l}^{(N)}=l_{\omega_{\w{l}}}^{(N)},\,1/\omega_l=\omega_{\w{l}},\,\w{l}$ 
and $z/\omega_l$ to get
\begin{eqnarray*}
\left|
l_{\omega_l}^{(N)}(z)
\right|
& = &
\left|
l_{1/\omega_l}^{(N)}
\left(
z/\omega_l
\right)
\right|
\;=\;
\left|
l_{\omega_{\w{l}}}^{(N)}
\left(
z/\omega_l
\right)
\right|
\;\leq\;
\frac{1}{
1-
\varepsilon_0/4
}
\,,
\end{eqnarray*}
and this proves the lemma in the second case.

\end{proof}

\begin{remark}

We can numerically check that $1$ and $\omega_l$ are almost
the maxima of
$\left|l_{\omega_l}^{(N)}\right|$ on $\overline{\D}$.
It follows that the estimate from the above lemma may hold in the
whole disk, as confirmed by the following result.

\end{remark}

\begin{lemma}\label{estimlomegaldisk}

Let fix $\varepsilon_0>0$ small enough. Then
there is $P_1\left(\varepsilon_0\right)\geq1$
(that only depends on $\varepsilon_0$) such that,
for all $p_1\geq P_1\left(\varepsilon_0\right)$ and
all $l$ with
$\varepsilon_02^{p_1}\leq l\leq\left(1-\varepsilon_0\right)2^{p_1}-1$,
one has
\begin{eqnarray*}
\left\|
l_{\omega_l}^{(N)}
\right\|_{\overline{\D}}
& \leq &
\frac{1}{
1-
\varepsilon_0/4
}
\,.
\end{eqnarray*}

\end{lemma}

\begin{proof}

Let fix any $z\in\C$ with
$|z|\leq1$.
If $|z-1|\geq\varepsilon_0^2$
and
$\left|z-\omega_l\right|\geq\varepsilon_0^2$, then
in particular $z\neq1,\omega_l$, and one
has that (since
$|\omega_l|=1$ and $|z|\leq1$)
\begin{eqnarray*}
\left|
l_{\omega_l}^{(N)}(z)
\right|
& = &
\frac{
\left|1-\omega_l\right|
}{
2^{p_1+1}
}
\frac{
\left|
z^{2^{p_1+1}}-1
\right|
}{
\left|z-1\right|
\,
\left|z-\omega_l\right|
}
\;\leq\;
\frac{2}{2^{p_1+1}}
\frac{1+1}{
\varepsilon_0^2\times\varepsilon_0^2
}
\;=\;
\frac{2}{2^{p_1}\varepsilon_0^4}
\,.
\end{eqnarray*}
Since
$\lim_{p_1\rightarrow+\infty}\dfrac{2}{2^{p_1}\varepsilon_0^4}=0$,
there is $P_1\left(\varepsilon_0\right)\geq1$ such that 
$\dfrac{2}{2^{p_1}\varepsilon_0^4}\leq1$
for all
$p_1\geq P_1\left(\varepsilon_0\right)$, then
\begin{eqnarray*}
\left|
l_{\omega_l}^{(N)}(z)
\right|
& \leq &
\frac{2}{2^{p_1}\varepsilon_0^4}
\;\leq\;
1
\;\leq\;
\frac{1}{
1-
\varepsilon_0/4
}
\,.
\end{eqnarray*}

Otherwise, 
$|z-1|\leq\varepsilon_0^2$ or
$\left|z-\omega_l\right|\leq\varepsilon_0^2$
thus, Lemma~\ref{estimlomegalmax} yields for all
$p_1\geq1$ (then for all
$p_1\geq P_1\left(\varepsilon_0\right)$ in particular),
\begin{eqnarray*}
\left|
l_{\omega_l}^{(N)}(z)
\right|
& \leq &
\frac{1}{
1-
\varepsilon_0/4
}
\,,
\end{eqnarray*}
and the lemma follows.

\end{proof}

\subsection{Proof of Theorem~\ref{specialN} for the {\em first half}
of the $l_k^{(N)}$'s}

In this subsection, we give a proof of the theorem in the special
case of the $l_k^{(N)}$'s for
$k=1,\ldots,2^{p_1}$. We begin with the following preliminar property
for polynomials.

\begin{lemma}\label{prelimpol}

For all $m\geq0$, one has
\begin{eqnarray*}
\left(
X-Y
\right)
\prod_{j=0}^m
\left(
X^{2^j}+Y^{2^j}
\right)
& = &
X^{2^{m+1}}-Y^{2^{m+1}}
\,.
\end{eqnarray*}

\end{lemma}

\begin{proof}

The proof is by induction on $m\geq0$. For $m=0$, one has that
\begin{eqnarray*}
\left(
X-Y
\right)
\prod_{j=0}^0
\left(
X^{2^j}+Y^{2^j}
\right)
\,=\,
(X-Y)\left(
X^{2^0}+Y^{2^0}
\right)
\,=\,
(X-Y)(X+Y)
\,=\,
X^{2^1}-Y^{2^1}
.
\end{eqnarray*}
Now if $m\geq0$, then the induction hypothesis yields
\begin{eqnarray*}
\left(
X-Y
\right)
\prod_{j=0}^{m+1}
\left(
X^{2^j}+Y^{2^j}
\right)
& = &
\left(
X-Y
\right)
\prod_{j=0}^{m}
\left(
X^{2^j}+Y^{2^j}
\right)
\times
\left(
X^{2^{m+1}}+Y^{2^{m+1}}
\right)
\\
& = &
\left(
X^{2^{m+1}}-Y^{2^{m+1}}
\right)
\times
\left(
X^{2^{m+1}}+Y^{2^{m+1}}
\right)
\\
& = &
\left(X^{2^{m+1}}\right)^2
-
\left(Y^{2^{m+1}}\right)^2
\;=\;
X^{2^{m+2}}-Y^{2^{m+2}}
\,,
\end{eqnarray*}
and the induction is achieved.

\end{proof}

The following lemma gives the link between the $l_k^{(N)}$'s and the
$l_{\omega_l}^{(N)}$'s introduced at the beginning of this section.

\begin{lemma}\label{lktolomegal}

Let $\mathcal{L}_N$ be a $N$-Leja section for the unit disk
(that starts at $1$).
For all $k=1,\ldots,2^{p_1}$, let $l_k^{(N)}$ be the FLIP
associated with $z_k\in\mathcal{L}_N$ (that is a $2^{p_1}$-th root of the unity
since it belongs to $\Omega_{2^{p_1}}$). Then there is 
$l(k)$ with $0\leq l(k)\leq2^{p_1}-1$ such that
\begin{eqnarray*}
\left\|
l_k^{(N)}
\right\|_{\overline{\D}}
\;=\;
\left\|
l_{\omega_{l(k)}}^{(N)}
\right\|_{\overline{\D}}
\,,
& \mbox{with} &
\omega_{l(k)}
\;=\;
\exp
\left(
\frac{2l(k)+1}{2^{p_1}}\,i\pi
\right)
\,.
\end{eqnarray*}

In addition, the correspondence
\begin{eqnarray*}
k,\,1\leq k\leq2^{p_1}
& \mapsto &
l(k),\,
0\leq l(k)\leq2^{p_1}-1
\,,
\end{eqnarray*}
is well-defined and one-to-one.

\end{lemma}

\begin{proof}

First, we know by Lemma~\ref{lkgeneral} that there is
$\omega$ a $2^{p_1}$-th root of $-1$ 
(we designate it by $\omega$ instead of $\omega_0$ to
not confuse it with the notation from~(\ref{defomegal}) 
for $l=0$)
such that for all $z\neq z_k,\,\omega$ (by also recalling~(\ref{dataspecialN}) ),
\begin{eqnarray*}
\left|
l_k^{(N)}(z)
\right|
& = &
\frac{1}{2^{p_1}}
\frac{\left|z^{2^{p_1}}-1\right|}{\left|z-z_k\right|}
\times
\prod_{q=2}^n
\frac{
\left|z^{2^{p_q}}+\omega^{2^{p_q}}\right|
}{
\left|z_k^{2^{p_q}}+\omega^{2^{p_q}}\right|
}
\\
& = &
\frac{1}{2^{p_1}}
\frac{\left|z^{2^{p_1}}+\omega^{2^{p_1}}\right|}{\left|z-z_k\right|}
\times
\prod_{q=2}^{p_1+1}
\frac{
\left|z^{2^{p_1-q+1}}+\omega^{2^{p_1-q+1}}\right|
}{
\left|z_k^{2^{p_1-q+1}}+\omega^{2^{p_1-q+1}}\right|
}
\\
& = &
\frac{1}{2^{p_1}}
\frac{\left|z^{2^{p_1}}+\omega^{2^{p_1}}\right|}{\left|z-z_k\right|}
\times
\prod_{q=0}^{p_1-1}
\frac{
\left|z^{2^{q}}+\omega^{2^{q}}\right|
}{
\left|z_k^{2^{q}}+\omega^{2^{q}}\right|
}
\,.
\end{eqnarray*}
Now successive applications of Lemma~\ref{prelimpol} give
(we remind that $z_k$ is a $2^{p_1}$-th root of the unity,
and since $\omega$ is a
$2^{p_1}$-th root of $-1$, then it is a $2^{p_1+1}$-th root
of the unity)
\begin{eqnarray*}
\prod_{q=0}^{p_1-1}
\left|z_k^{2^q}+\omega^{2^{q}}\right|
\;=\;
\frac{
\left|z_k-\omega\right|
\times
\prod_{q=0}^{p_1-1}
\left|z_k^{2^q}+\omega^{2^{q}}\right|
}{
\left|z_k-\omega\right|
}
\;=\;
\frac{
\left|z_k^{2^{p_1}}-\omega^{2^{p_1}}\right|
}{
\left|z_k-\omega\right|
}
\;=\;
\frac{
2
}{
\left|z_k-\omega\right|
}
\end{eqnarray*}
(notice that we cannot have $z_k=\omega$
since
$\omega\notin\Omega_{2^{p_1}}\ni z_k$),
and for all $z\neq\omega$,
\begin{eqnarray*}
\left|z^{2^{p_1}}+\omega^{2^{p_1}}\right|
\times
\prod_{q=0}^{p_1-1}
\left|z^{2^{q}}+\omega^{2^{q}}\right|
& = &
\frac{
\left|z-\omega\right|
\times
\prod_{q=0}^{p_1}
\left|z^{2^{q}}+\omega^{2^{q}}\right|
}{
\left|z-\omega\right|
}
\\
& = &
\frac{
\left|z^{2^{p_1+1}}-\omega^{2^{p_1+1}}\right|
}{
\left|z-\omega\right|
}
\;=\;
\frac{
\left|z^{2^{p_1+1}}-1\right|
}{
\left|z-\omega\right|
}
\,.
\end{eqnarray*}
It follows that for all $z\neq z_k,\,\omega$,
\begin{eqnarray*}
\left|
l_k^{(N)}(z)
\right|
& = &
\frac{1}{2^{p_1}}
\times
\frac{1}{|z-z_k|}
\times
\frac{
\left|z^{2^{p_1+1}}-1\right|
}{
\left|z-\omega\right|
}
\times
\frac{
\left|z_k-\omega\right|
}{2}
\\
& = &
\frac{\left|1-\omega/z_k\right|}{2^{p_1+1}}
\frac{
\left|\left(z/z_k\right)^{2^{p_1+1}}-1\right|
}{
\left|z/z_k-1\right|
\,
\left|
z/z_k-\omega/z_k\right|
}
\;=\;
\left|
l_{\omega/z_k}^{(N)}\left(z/z_k\right)
\right|
\end{eqnarray*}
by~(\ref{defl1omegal}) (and because $z_k^{2^{p_1+1}}=1^2=1$).
$\omega/z_k$ being still a $2^{p_1}$-th root of $-1$, there is
$l(k)$ with $0\leq l(k)\leq 2^{p_1}-1$ such that
$\omega/z_k=\omega_{l(k)}$. In particular,
\begin{eqnarray*}
\sup_{|z|\leq1}
\left|
l_k^{(N)}(z)
\right|
& = &
\sup_{|z|\leq1}
\left|
l_{\omega/z_k}^{(N)}\left(z/z_k\right)
\right|
\;=\;
\left\|
l_{\omega/z_k}^{(N)}
\right\|_{\overline{\D}}
\,,
\end{eqnarray*}
and this proves the first assertion.

In order to prove the other one, we first remind from Remark~\ref{omega0}
that $\omega$ is any of the $2^{p_n}$ different choices for
$z_{N+1}$. Here by~(\ref{dataspecialN}), $p_n=0$ and $N+1=2^{p_1+1}$,
then $\omega$ is the only possible choice for $z_{N+1}$ and
$\mathcal{L}_N\cup\left\{\omega\right\}
=\mathcal{L}_{N+1}=\mathcal{L}_{2^{p_1+1}}=\Omega_{2^{p_1+1}}$ 
by Theorem~5 from~\cite{bialascalvi} (the equality
being meant as sets).
In particular, the uniqueness of $\omega$ implies that
for any $\omega_l$, there is
at most one $z_k$ such that $\omega_l=\omega/z_k$.
This proves that the application $k\mapsto l(k)$
is injective and the one-to-one correspondence follows
because
$\mbox{card}\left\{k,\,1\leq k\leq2^{p_1}\right\}
=2^{p_1}
=\mbox{card}\left\{l,\,0\leq l\leq2^{p_1}-1\right\}$.

\end{proof}

Now we can prove a special case of Theorem~\ref{specialN} for the {\em first half}
of the $l_k^{(N)}$'s, i.e. for $k=1,\ldots,2^{p_1}$.

\begin{lemma}\label{specialNhalflk}

Let fix $\varepsilon_0>0$ small enough. Then there is
$P_1\left(\varepsilon_0\right)\geq1$ large enough such that
for all $p_1\geq P_1\left(\varepsilon_0\right)$, one has
\begin{eqnarray*}
\frac{1}{2^{p_1}}
\sum_{k=1}^{2^{p_1}}
\left\|
l_k^{(N)}
\right\|_{\overline{\D}}
& \leq &
2\pi\exp(3\pi)
\left(
\varepsilon_0+\frac{1}{2^{p_1}}
\right)
+
\frac{1-2\varepsilon_0}{1-\varepsilon_0/4}
\,.
\end{eqnarray*}

\end{lemma}

\begin{proof}

First, we know by Lemma~\ref{lktolomegal} that
\begin{eqnarray*}
\sum_{k=1}^{2^{p_1}}
\left\|
l_k^{(N)}
\right\|_{\overline{\D}}
& = &
\sum_{k=1}^{2^{p_1}}
\left\|
l_{\omega_{l(k)}}^{(N)}
\right\|_{\overline{\D}}
\;=\;
\sum_{l=0}^{2^{p_1}-1}
\left\|
l_{\omega_{l}}^{(N)}
\right\|_{\overline{\D}}
\,.
\end{eqnarray*}
Next, an application of Lemma~\ref{estimlomegaldisk} gives
$P_1\left(\varepsilon_0\right)\geq1$ such that
for all $p_1\geq P_1\left(\varepsilon_0\right)$,
\begin{eqnarray*}
\sum_{l=0}^{2^{p_1}-1}
\left\|
l_{\omega_{l}}^{(N)}
\right\|_{\overline{\D}}
& = &
\left[
\sum_{0\leq l<\varepsilon_02^{p_1}}
+
\sum_{\varepsilon_02^{p_1}\leq l\leq\left(1-\varepsilon_0\right)2^{p_1}-1}
+
\sum_{\left(1-\varepsilon_0\right)2^{p_1}-1< l\leq2^{p_1}-1}
\right]
\left\|
l_{\omega_{l}}^{(N)}
\right\|_{\overline{\D}}
\\
& \leq &
\left(
\varepsilon_02^{p_1}+1
\right)
\times
\pi\exp(3\pi)
+
\frac{
\left(1-2\varepsilon_0\right)2^{p_1}
}{1-\varepsilon_0/4}
+
\left(
\varepsilon_02^{p_1}+1
\right)
\times
\pi\exp(3\pi)
\,,
\end{eqnarray*}
the other estimates being applications of Theorem~\ref{unifestim}. 
Hence
\begin{eqnarray*}
\frac{1}{2^{p_1}}
\sum_{k=1}^{2^{p_1}}
\left\|
l_k^{(N)}
\right\|_{\overline{\D}}
& \leq &
\frac{1}{2^{p_1}}
\left[
2\pi\exp(3\pi)
\times
\left(\varepsilon_02^{p_1}+1\right)
+
\frac{
\left(1-2\varepsilon_0\right)2^{p_1}
}{1-\varepsilon_0/4}
\right]
\\
& = &
2\pi\exp(3\pi)
\left(
\varepsilon_0+\frac{1}{2^{p_1}}
\right)
+
\frac{1-2\varepsilon_0}{1-\varepsilon_0/4}
\,,
\end{eqnarray*}
and the lemma is proved.

\end{proof}

\begin{remark}

We could have used in the proof the sharper estimate from
Proposition~\ref{unifestimspecialN} below in order to replace
the bound $\pi\exp(3\pi)$ with $2$ (since
$N=2^{p_1+1}-1$). However, we will see that it is
useless for the proof of Theorem~\ref{specialN}.

\end{remark}

\subsection{Proof of Theorem~\ref{specialN}}

In order to give the proof, we first want to begin with
a preliminar result in which we deal with the FLIPs
$l_k^{(N)}$ for $k=2^{p_1}+1,\ldots,N$.

\begin{lemma}\label{prelimspecial}

Let $\mathcal{L}_N$ be any $N$-Leja section for the unit disk
that starts at $z_1=1$.
For all $k=1,\ldots,N$, one has
\begin{eqnarray*}
\left\|
l_k^{(N)}
\right\|_{\overline{\D}}
& \leq &
\left\|
\w{l}_{k'}^{(2^{q}-1)}
\right\|_{\overline{\D}}
\,,
\end{eqnarray*}
where $1\leq q\leq p_1+1$, $1\leq k'\leq 2^{q-1}$ and
$\w{l}_{k'}^{(2^q-1)}$ is the FLIP
associated with $\w{z_{k'}}\in\w{\mathcal{L}}_{2^q-1}$
and
$\w{\mathcal{L}}_{2^q-1}$ 
is a $\left(2^q-1\right)$-Leja section that also
starts at $\w{z_1}=1$.

In addition, the correspondence between
$k$ and $\left(q,k'\right)$ is one-to-one: more precisely
the application defined by
\begin{eqnarray}\label{linkq}
& &
\begin{cases}
\;\;\;k\,,\;\;\;\;\,
1\leq k\leq 2^{p_1}
\;\mapsto\;
\left(p_1+1,k\right)
,\,1\leq k\leq 2^{p_1}
\,,
\\
k,\,2^{p_1}+1\leq k\leq N
\;\mapsto\; 
\left(q,k'\right)\,,\;
1\leq q\leq p_1\,,\;
1\leq k'\leq2^{q-1}
\,,
\end{cases}
\end{eqnarray}
is well-defined and is a one-to-one correspondence.

\end{lemma}

\begin{proof}

The proof is by induction on $p_1\geq0$ where
$N=2^{p_1+1}-1$. If $p_1=0$, then
$\mathcal{L}_N=\mathcal{L}_{2^{p_1+1}-1}=\mathcal{L}_1
=\left\{z_1\right\}=\{1\}$. Necessarily, $k=1$ and the lemma is obvious by
taking $q=1$,  the same $1$-Leja section
$\mathcal{L}_1=\{1\}$ and $k'=k=1$ (and the 
correspondence~(\ref{linkq}) is of course well-defined and one-to-one).

Now let be $p_1\geq1$.
First, if $1\leq k\leq2^{p_1}$, then we take
$q=p_1+1$, the same $N$-Leja section
$\mathcal{L}_N$ and $k'=k$.
In addition, the application
\begin{eqnarray}\label{linkq1}
k,\,1\leq k\leq 2^{p_1}
& \mapsto &
(p_1+1,k),\,1\leq k\leq 2^{(p_1+1)-1}
\,,
\end{eqnarray}
is obviously well-defined and is a one-to-one correspondence.

Otherwise, $2^{p_1}+1\leq k\leq N$,
then an application of
Lemma~\ref{prelimpfthm} yields
\begin{eqnarray}\label{prelimspecial1}
\left\|l_k^{(N)}\right\|_{\overline{\D}}
& \leq &
\left\|
\w{l}_{k'}^{(N-2^{p_1})}
\right\|_{\overline{\D}}
\,,
\end{eqnarray}
where $\w{l}_{k'}^{(N-2^{p_1})}$ is the FLIP
associated with $\w{z_{k'}}\in\w{\mathcal{L}}_{N-2^{p_1}}$
and
$\w{\mathcal{L}}_{N-2^{p_1}}$ 
is an $\left(N-2^{p_1}\right)$-Leja section that also
starts at $\w{z_1}=1$. 
In addition, the application
\begin{eqnarray}\label{linkq2}
k,\,2^{p_1}+1\leq k\leq N
& \mapsto &
k',\,1\leq k'\leq N-2^{p_1}
\,,
\end{eqnarray}
is well-defined and is a one-to-one correspondence.

Since
$N-2^{p_1}=2^{p_1+1}-1-2^{p_1}=2^{p_1}-1$,
the induction hypothesis applied to $p_1-1$ and the 
$\left(2^{p_1}-1\right)$-Leja section
$\w{\mathcal{L}}_{2^{p_1}-1}$, leads to
\begin{eqnarray}\label{prelimspecial2}
\left\|
\w{l}_{k'}^{(N-2^{p_1})}
\right\|_{\overline{\D}}
& \leq &
\left\|
\w{\w{l}}_{k''}^{\,(2^q-1)}
\right\|_{\overline{\D}}
\,,
\end{eqnarray}
where $1\leq q\leq p_1\leq p_1+1$, $1\leq k'\leq 2^{q-1}$ and
${\w{\w{l}}_{k''}}^{\,(2^q-1)}$ is the FLIP
associated with $\w{\w{z_{k'}}}\in\w{\w{\mathcal{L}}\,}_{2^q-1}$,
where $\w{\w{\mathcal{L}}\,}_{2^q-1}$ 
is a $\left(2^q-1\right)$-Leja section that
starts at $\w{\w{z_1}}=1$. 
Moreover, the application defined by
\begin{eqnarray}\label{linkq25}
& &
k,\,1\leq k'\leq 2^{p_1}-1
\;\mapsto\;
\left(q,k''\right)\,,\;
1\leq q\leq \left(p_1-1\right)+1\,,\;
1\leq k''\leq2^{q-1}
\,,
\end{eqnarray}
is well-defined and is a one-to-one correspondence, then so is
the following one by composition of~(\ref{linkq2}) and~(\ref{linkq25}):
\begin{eqnarray}\label{linkq3}
k,\,2^{p_1}+1\leq k\leq N
& \mapsto &
\left(q,k''\right)\,,\;
1\leq q\leq p_1\,,\;
1\leq k''\leq2^{q-1}
\,.
\end{eqnarray}
Finally, the partial applications~(\ref{linkq1}) and~(\ref{linkq3})
yield~(\ref{linkq}).
On the other hand, the estimates~(\ref{prelimspecial1})
and~(\ref{prelimspecial2}) achieve the induction and the proof of the lemma.

\end{proof}

As a first consequence, we have the following improvement for the bound
of $l_k^{(N)}$ in Theorem~\ref{unifestim}.

\begin{proposition}\label{unifestimspecialN}

One has for the special case of
$N=2^{p_1+1}-1$ that
\begin{eqnarray*}
\frac{4}{\pi}
\left(
1+\varepsilon(1/N)
\right)
\;\leq\;
\max_{1\leq k\leq N}
\left\|l_k^{(N)}\right\|_{\overline{\D}}
\;\leq\;
2
\,
\end{eqnarray*}
where $\lim_{N\rightarrow+\infty}\varepsilon(1/N)=0$.
In addition, $\varepsilon(1/N)$ can be chosen so that for all $N\geq3$,
\begin{eqnarray}\label{unifestimspecialNepsilon}
\frac{4}{\pi}
\left(
1+\varepsilon(1/N)
\right)
& > &
1
\,.
\end{eqnarray}

\end{proposition}

Of course, it can be numerically checked that the upper bound is not optimal: except for
exceptional values of $k$, the bound of $l_k^{(N)}$ is almost $1$. On the other hand,
the lower bound is reached only for some of these exceptional values of $k$.

\begin{proof}

First, we can wlog assume that the Leja sequence $\mathcal{L}$
starts at $z_1=1$.

We begin with the upper bound. The proof is by induction on $p_1\geq0$. If $p_1=0$, i.e.
$N=2^{p_1+1}-1=1$, then necessarily $k=1$ and the estimate is obvious since
$\left\|l_1^{(1)}\right\|_{\overline{\D}}
=\left\|1\right\|_{\overline{\D}}=1$.

If $p_1\geq1$ and $N=2^{p_1+1}-1$,
then Lemma~\ref{prelimpfthm} and the induction hypothesis applied to
$2^{\left(p_1-1\right)+1}-1=2^{p_1+1}-2^{p_1}-1=N-2^{p_1}$, yield
\begin{eqnarray}\nonumber
\max_{2^{p_1}+1\leq k\leq N}
\left\|
l_k^{(N)}
\right\|_{\overline{\D}}
& \leq &
\max_{1\leq k'\leq N-2^{p_1}}
\left\|
\w{l}_{k'}^{\left(N-2^{p_1}\right)}
\right\|_{\overline{\D}}
\\\label{unifestimspecialN1}
& = &
\max_{1\leq k'\leq 2^{p_1}-1}
\left\|
\w{l}_{k'}^{\left(2^{p_1}-1\right)}
\right\|_{\overline{\D}}
\;\leq\;
2
\,.
\end{eqnarray}
On the other hand, one has by Lemma~\ref{lktolomegal}
that
\begin{eqnarray}\label{unifestimspecialN2}
\max_{1\leq k\leq2^{p_1}}
\left\|
l_k^{(N)}
\right\|_{\overline{\D}}
& = &
\max_{1\leq k\leq2^{p_1}}
\left\|
l_{\omega_{l(k)}}^{(N)}
\right\|_{\overline{\D}}
\;=\;
\max_{0\leq l\leq2^{p_1}-1}
\left\|
l_{\omega_l}^{(N)}
\right\|_{\overline{\D}}
\,.
\end{eqnarray}
Now for all $l=0,\ldots,2^{p_1}-1$ and all
$z\in\overline{\D}$ with $z\neq1,\,\omega_l$, one has by~(\ref{defl1omegal}) that 
\begin{eqnarray*}
\left|
l_{\omega_l}^{(N)}(z)
\right|
& = &
\frac{
\left|
1-\omega_l
\right|
}{
2^{p_1+1}
}
\times
\frac{
\left|
z^{2^{p_1+1}}-1
\right|
}{
\left|z-1\right|
\,
\left|z-\omega_l\right|
}
\;\leq\;
\frac{
\left|
1-z
\right|
+
\left|
z-\omega_l
\right|
}{
2^{p_1+1}
}
\times
\frac{
\left|
z^{2^{p_1}}-1
\right|
\,
\left|
z^{2^{p_1}}+1
\right|
}{
\left|z-1\right|
\,
\left|z-\omega_l\right|
}
\\
& = &
\frac{
1
}{
2^{p_1+1}
}
\frac{
\left|
z^{2^{p_1}}-1
\right|
\,
\left|
z^{2^{p_1}}+1
\right|
}{
\left|z-\omega_l\right|
}
+
\frac{
1
}{
2^{p_1+1}
}
\frac{
\left|
z^{2^{p_1}}-1
\right|
\,
\left|
z^{2^{p_1}}+1
\right|
}{
\left|z-1\right|
}
\\
& \leq &
\frac{
1
}{
2^{p_1+1}
}
\frac{
(1+1)
\left|
z^{2^{p_1}}+1
\right|
}{
\left|z-\omega_l\right|
}
+
\frac{
1
}{
2^{p_1+1}
}
\frac{
\left|
z^{2^{p_1}}-1
\right|
(1+1)
}{
\left|z-1\right|
}
\\
& = &
\frac{
1
}{
2^{p_1}
}
\frac{
\left|
z^{2^{p_1}}
-
\omega_l^{2^{p_1}}
\right|
}{
\left|z-\omega_l\right|
}
+
\frac{
1
}{
2^{p_1}
}
\frac{
\left|
z^{2^{p_1}}-1
\right|
}{
\left|z-1\right|
}
\;=\;
\frac{
1
}{
2^{p_1}
}
\frac{
\left|
(z/\omega_l)^{2^{p_1}}-1
\right|
}{
\left|z/\omega_l-1\right|
}
+
\frac{
1
}{
2^{p_1}
}
\frac{
\left|
z^{2^{p_1}}-1
\right|
}{
\left|z-1\right|
}
\,,
\end{eqnarray*}
since $\omega_l$ is a $2^{p_1}$-th root of $-1$.
The symmetry of the unit disk and Lemma~\ref{estim1} yield
\begin{eqnarray*}
\left\|
\frac{
1
}{
2^{p_1}
}
\frac{
(z/\omega_l)^{2^{p_1}}-1
}{
z/\omega_l-1
}
\right\|_{\overline{\D}}
& = &
\left\|
\frac{
1
}{
2^{p_1}
}
\frac{
z^{2^{p_1}}-1
}{
z-1
}
\right\|_{\overline{\D}}
\;=\;
\left\|
l_1^{\left(2^{p_1}\right)}
\right\|_{\overline{\D}}
\;=\;
1
\,.
\end{eqnarray*}
It follows that
\begin{eqnarray*}
\left\|
l_{\omega_l}^{(N)}
\right\|_{\overline{\D}}
& \leq &
\left\|
\frac{
1
}{
2^{p_1}
}
\frac{
(z/\omega_l)^{2^{p_1}}-1
}{
z/\omega_l-1
}
\right\|_{\overline{\D}}
+
\left\|
\frac{
1
}{
2^{p_1}
}
\frac{
z^{2^{p_1}}-1
}{
z-1
}
\right\|_{\overline{\D}}
\;=\;
2
\,,
\end{eqnarray*}
and~(\ref{unifestimspecialN2}) becomes
\begin{eqnarray}\label{unifestimspecialN3}
\max_{1\leq k\leq2^{p_1}}
\left\|
l_k^{(N)}
\right\|_{\overline{\D}}
& = &
\max_{0\leq l\leq2^{p_1}-1}
\left\|
l_{\omega_l}^{(N)}
\right\|_{\overline{\D}}
\;\leq\;
2
\,.
\end{eqnarray}
Finally,
(\ref{unifestimspecialN1}) and~(\ref{unifestimspecialN3}) prove the induction.

\medskip

In order to prove the lower bound, we first have by~(\ref{unifestimspecialN2})
that for all $N\geq1$,
\begin{eqnarray}\nonumber
\max_{1\leq k\leq N}
\left\|l_k^{(N)}\right\|_{\overline{\D}}
& \geq &
\max_{1\leq k\leq 2^{p_1}}
\left\|l_k^{(N)}\right\|_{\overline{\D}}
\;=\;
\max_{0\leq l\leq2^{p_1}-1}
\left\|
l_{\omega_l}^{(N)}
\right\|_{\overline{\D}}
\\\label{unifestimspecialN4}
& \geq &
\left\|
l_{\omega_0}^{(N)}
\right\|_{\overline{\D}}
\;\geq\;
\left|
l_{\omega_0}^{(N)}
\left(
\exp\left(i\pi/2^{p_1+1}\right)
\right)
\right|
\,.
\end{eqnarray}
On the other hand, one has by~(\ref{defomegal}) and~(\ref{defl1omegal}) that
\begin{eqnarray}\nonumber
\left|
l_{\omega_0}^{(N)}
\left(
\exp\left(
\frac{i\pi}{2^{p_1+1}}
\right)
\right)
\right|
& = &
\frac{
\left|
1-
\exp
\left(
\dfrac{i\pi}{2^{p_1}}
\right)
\right|
\times
\left|
\exp\left(
2^{p_1+1}\times 
\dfrac{i\pi}{2^{p_1+1}}
\right)
-1
\right|
}{
2^{p_1+1}
\times
\left|
\exp\left(
\dfrac{i\pi}{2^{p_1+1}}
\right)
-
1
\right|
\times
\left|
\exp\left(
\dfrac{i\pi}{2^{p_1+1}}
\right)
-
\exp\left(
\dfrac{i\pi}{2^{p_1}}
\right)
\right|
}
\\\nonumber
& = &
\frac{
2\sin\left(\pi/2^{p_1+1}\right)
\times
2
}{
2^{p_1+1}
\times
2\sin\left(\pi/2^{p_1+2}\right)
\times
2\sin\left(\pi/2^{p_1+2}\right)
}
\\\label{unifestimspecialN5}
& = &
\frac{
2\sin\left(\pi/2^{p_1+2}\right)
\times\cos\left(\pi/2^{p_1+2}\right)
}{
2^{p_1+1}
\sin^2\left(\pi/2^{p_1+2}\right)
}
\;=\;
\frac{
\cos\left(\pi/2^{p_1+2}\right)
}{
2^{p_1}
\sin\left(\pi/2^{p_1+2}\right)
}
\\\label{unifestimspecialN6}
& \sim &
\frac{1}{2^{p_1}
\times
\pi/2^{p_1+2}}
\;\xrightarrow[p_1\rightarrow+\infty]{}\;
\frac{4}{\pi}
\,.
\end{eqnarray}
Since $N\rightarrow+\infty$ iff $p_1\rightarrow+\infty$, 
it follows that~(\ref{unifestimspecialN4}) and~(\ref{unifestimspecialN6}) lead to
\begin{eqnarray*}
\max_{1\leq k\leq N}
\left\|l_k^{(N)}\right\|_{\overline{\D}}
& \geq &
\left|
l_{\omega_0}^{(N)}
\left(
\exp\left(i\pi/2^{p_1+1}
\right)
\right)
\right|
\;=\;
\frac{4}{\pi}
\left(
1+\varepsilon
\left(1/N
\right)
\right)
\,,
\end{eqnarray*}
where $\lim_{N\rightarrow+\infty}\varepsilon(1/N)=0$, and 
this yields the required lower bound.
In addition, one also has by~(\ref{unifestimspecialN4}),
(\ref{unifestimspecialN5}) and the estimate~(\ref{estimsinus1})
from Lemma~\ref{estimsinus} that for all $p_1\geq1$,
\begin{eqnarray*}
\max_{1\leq k\leq N}
\left\|l_k^{(N)}\right\|_{\overline{\D}}
\;\geq\;
\frac{
\cos\left(\pi/2^{p_1+2}\right)
}{
2^{p_1}
\sin\left(\pi/2^{p_1+2}\right)
}
\;\geq\;
\frac{
\cos\left(\pi/2^{p_1+2}\right)
}{
2^{p_1}
\times
\pi/2^{p_1+2}
}
\;\geq\;
\frac{4\cos(\pi/8)}{\pi}
\;>\;
1
\,,
\end{eqnarray*}
and this proves~(\ref{unifestimspecialNepsilon}) 
(since $p_1\geq1$ iff $N=2^{p_1+1}-1\geq3$) and
completes the whole proof of the proposition.

\end{proof}

Now we can give the proof of Theorem~\ref{specialN}.

\begin{proof}

First, 
$N=2^{p_1+1}-1$ and $\mathcal{L}_N$ a
$N$-Leja section being given, by symmetry of the unit disk,
we can wlog assume that $\mathcal{L}_N$ starts at $z_1=1$. Next, 
let us consider $k_N=1,\ldots,N$, that satisfies
\begin{eqnarray*}
\left\|l_{k_N}^{(N)}\right\|_{\overline{\D}}
& = &
\max_{1\leq k\leq N}
\left\|l_k^{(N)}\right\|_{\overline{\D}}
\,.
\end{eqnarray*}
We know by Proposition~\ref{unifestimspecialN} and~(\ref{unifestimspecialNepsilon}) that
for all $N\geq3$,
\begin{eqnarray*}
\left\|l_{k_N}^{(N)}\right\|_{\overline{\D}}
& \geq &
\frac{4}{\pi}
\left(
1+\varepsilon(1/N)
\right)
\;>\;1
\,,
\end{eqnarray*}
where $\lim_{N\rightarrow+\infty}\varepsilon(1/N)=0$.
On the other hand, since for all $k=1,\ldots,N$,
$\left\|l_k^{(N)}\right\|_{\overline{\D}}\geq\left|l_k^{(N)}\left(z_k\right)\right|=1$,
it follows that
\begin{eqnarray}\nonumber
\sum_{k=1}^{N}
\left\|l_k^{(N)}\right\|_{\overline{\D}}
& = &
\left\|l_{k_N}^{(N)}\right\|_{\overline{\D}}
+
\sum_{k=1,\,k\neq k_N}^{N}
\left\|l_k^{(N)}\right\|_{\overline{\D}}
\;\geq\;
\frac{4}{\pi}
\left(
1+\varepsilon(1/N)
\right)
+
\sum_{k=1,\,k\neq k_N}^{N}
1
\\\label{specialNlow1}
& = &
\frac{4}{\pi}
\left(
1+\varepsilon(1/N)
\right)
+N-1
\;>\;
1+(N-1)
\;=\;
N
\,.
\end{eqnarray}
In particular, this last estimate yields
\begin{eqnarray}\label{pfthspecial0}
\liminf_{N\rightarrow+\infty}
\left[
\frac{1}{N}
\sum_{k=1}^{N}
\left\|l_k^{(N)}\right\|_{\overline{\D}}
\right]
& \geq &
1
\,.
\end{eqnarray}

The essential part of the proof is to deal with
the other inequality. We first have that
\begin{eqnarray*}
\frac{1}{N}
\sum_{k=1}^N
\left\|
l_k^{(N)}
\right\|_{\overline{\D}}
& = &
\frac{1}{2^{p_1+1}-1}
\sum_{k=1}^{2^{p_1+1}-1}
\left\|
l_k^{(N)}
\right\|_{\overline{\D}}
\\
& = &
\frac{2^{p_1}}{2^{p_1+1}-1}
\times
\frac{1}{2^{p_1}}
\sum_{k=1}^{2^{p_1}}
\left\|
l_k^{(N)}
\right\|_{\overline{\D}}
+
\frac{2^{p_1}}{2^{p_1+1}-1}
\times
\frac{1}{2^{p_1}}
\sum_{k=2^{p_1}+1}^{2^{p_1+1}-1}
\left\|
l_k^{(N)}
\right\|_{\overline{\D}}
\,,
\end{eqnarray*}
then (since
$N\rightarrow+\infty$ iff $p_1\rightarrow+\infty$)
\begin{eqnarray*}
\limsup_{N\rightarrow+\infty}
\left[
\frac{1}{N}
\sum_{k=1}^N
\left\|
l_k^{(N)}
\right\|_{\overline{\D}}
\right]
& \leq &
\;\;\;\;\;\;\;\;\;\;\;\;\;\;\;\;\;\;\;\;\;\;\;\;\;\;\;\;\;\;\;\;\;\;\;\;\;\;\;\;
\;\;\;\;\;\;\;\;\;\;\;\;\;\;\;\;\;\;\;\;\;\;\;\;\;\;\;\;\;\;\;\;\;\;\;\;\;\;\;\;
\end{eqnarray*}
\begin{eqnarray}\label{pfthspecial1}
&  &
\;\;\;\;\leq\;
\frac{1}{2}
\times
\limsup_{p_1\rightarrow+\infty}
\left[
\frac{1}{2^{p_1}}
\sum_{k=1}^{2^{p_1}}
\left\|
l_k^{(N)}
\right\|_{\overline{\D}}
\right]
+
\frac{1}{2}
\times
\limsup_{p_1\rightarrow+\infty}
\left[
\frac{1}{2^{p_1}}
\sum_{k=2^{p_1}+1}^{2^{p_1+1}-1}
\left\|
l_k^{(N)}
\right\|_{\overline{\D}}
\right]
\,.
\end{eqnarray}

Now let fix any $\varepsilon>0$ small enough.
On the one hand, by Lemma~\ref{specialNhalflk}, 
there is $P_1(\varepsilon)\geq1$ such that
for all $p_1\geq P_1(\varepsilon)$,
\begin{eqnarray*}
\frac{1}{2^{p_1}}
\sum_{k=1}^{2^{p_1}}
\left\|
l_k^{(N)}
\right\|_{\overline{\D}}
& \leq &
2\pi\exp(3\pi)
\left(
\varepsilon+\frac{1}{2^{p_1}}
\right)
+
\frac{1-2\varepsilon}{1-\varepsilon/4}
\,,
\end{eqnarray*}
then
\begin{eqnarray}\label{pfthspecial2}
\limsup_{p_1\rightarrow+\infty}
\left[
\frac{1}{2^{p_1}}
\sum_{k=1}^{2^{p_1}}
\left\|
l_k^{(N)}
\right\|_{\overline{\D}}
\right]
& \leq &
2\pi\varepsilon\exp(3\pi)
+
\frac{1-2\varepsilon}{1-\varepsilon/4}
\,.
\end{eqnarray}

On the other hand, the one-to-one correspondence defined
by~(\ref{linkq}) from Lemma~\ref{prelimspecial} and the associated
estimate lead to
\begin{eqnarray}\label{pfthspecial3}
& &
\sum_{k=2^{p_1}+1}^{2^{p_1+1}-1}
\left\|
l_k^{(N)}
\right\|_{\overline{\D}}
\;\leq\;
\sum_{q=1}^{p_1}
\sum_{k'=1}^{2^{q-1}}
\left\|
l_{k'}^{(2^q-1)}
\right\|_{\overline{\D}}
\;=\;
\sum_{q=0}^{p_1-1}
2^{q}\times
\frac{1}{2^{q}}
\sum_{k'=1}^{2^{q}}
\left\|
l_{k'}^{(2^{q+1}-1)}
\right\|_{\overline{\D}}
\,.
\end{eqnarray}
Now for every $q$
with
$P_1(\varepsilon)\leq q\leq p_1-1$, we can still apply
Lemma~\ref{specialNhalflk} for the
$\left(2^{q+1}-1\right)$-Leja section
$\w{\mathcal{L}}_{2^{q+1}-1}$
(with $p_1$ replaced with $q$ and
$N=2^{p_1+1}-1$ replaced with $2^{q+1}-1$) to get
\begin{eqnarray*}
\frac{1}{2^{q}}
\sum_{k'=1}^{2^{q}}
\left\|
l_{k'}^{(2^{q+1}-1)}
\right\|_{\overline{\D}}
& \leq &
2\pi\exp(3\pi)
\left(
\varepsilon+\frac{1}{2^q}
\right)
+
\frac{1-2\varepsilon}{1-\varepsilon/4}
\,,
\end{eqnarray*}
then
\begin{eqnarray*}
\sum_{q=P_1(\varepsilon)}^{p_1-1}
2^{q}\times
\frac{1}{2^{q}}
\sum_{k'=1}^{2^{q}}
\left\|
l_{k'}^{(2^{q+1}-1)}
\right\|_{\overline{\D}}
& \leq &
\;\;\;\;\;\;\;\;\;\;\;\;\;\;\;\;\;\;\;\;\;\;\;\;\;\;\;\;\;\;\;\;\;\;\;\;\;\;\;\;\;\;\;\;
\;\;\;\;\;\;\;\;\;\;\;\;\;\;\;\;\;\;\;\;\;\;\;\;\;\;\;\;\;\;\;\;\;\;\;\;\;\;\;\;
\end{eqnarray*}
\begin{eqnarray}\nonumber
\;\;\;\;
& \leq &
\left(
2\pi\varepsilon\exp(3\pi)
+
\frac{1-2\varepsilon}{1-\varepsilon/4}
\right)
\sum_{q=P_1(\varepsilon)}^{p_1-1}
2^q
+
2\pi\exp(3\pi)
\sum_{q=P_1(\varepsilon)}^{p_1-1}
1
\\\nonumber
& = &
\left(
2\pi\varepsilon\exp(3\pi)
+
\frac{1-2\varepsilon}{1-\varepsilon/4}
\right)
\frac{2^{p_1}-2^{P_1(\varepsilon)}}{2-1}
+
2\pi\exp(3\pi)\times\left(p_1-P_1(\varepsilon)\right)
\\\label{pfthspecial4}
& \leq &
\left(
2\pi\varepsilon\exp(3\pi)
+
\frac{1-2\varepsilon}{1-\varepsilon/4}
\right)
\times
2^{p_1}
+
2\pi\exp(3\pi)\times
p_1
\,.
\end{eqnarray}
For all $q$ with $0\leq q\leq P_1(\varepsilon)-1$, we just have by
Theorem~\ref{unifestim} that 
\begin{eqnarray}\nonumber
\sum_{q=0}^{P_1(\varepsilon)-1}
\sum_{k'=1}^{2^{q}}
\left\|
l_{k'}^{(2^{q+1}-1)}
\right\|_{\overline{\D}}
& \leq &
\sum_{q=0}^{P_1(\varepsilon)-1}
\sum_{k'=1}^{2^{q}}
\pi\exp(3\pi)
\;=\;
\pi\exp(3\pi)
\sum_{q=0}^{P_1(\varepsilon)-1}
2^q
\\\label{pfthspecial5}
& = &
\pi\exp(3\pi)
\times
\frac{2^{P_1(\varepsilon)}-1}{2-1}
\;\leq\;
\pi\exp(3\pi)
\times
2^{P_1(\varepsilon)}
\,.
\end{eqnarray}
It follows that~(\ref{pfthspecial3}), (\ref{pfthspecial4}) 
and~(\ref{pfthspecial5}) together yield
\begin{eqnarray*}
\sum_{k=2^{p_1}+1}^{2^{p_1+1}-1}
\left\|
l_k^{(N)}
\right\|_{\overline{\D}}
& \leq &
\sum_{q=0}^{P_1(\varepsilon)-1}
\sum_{k'=1}^{2^{q}}
\left\|
l_{k'}^{(2^{q+1}-1)}
\right\|_{\overline{\D}}
+
\sum_{q=P_1(\varepsilon)}^{p_1-1}
2^{q}\times
\frac{1}{2^{q}}
\sum_{k'=1}^{2^{q}}
\left\|
l_{k'}^{(2^{q+1}-1)}
\right\|_{\overline{\D}}
\\
& \leq &
\pi\exp(3\pi)\,
2^{P_1(\varepsilon)}
+
\left(
2\pi\varepsilon\exp(3\pi)
+
\frac{1-2\varepsilon}{1-\varepsilon/4}
\right)
2^{p_1}
+
2\pi\exp(3\pi)\,
p_1
\end{eqnarray*}
thus,
\begin{eqnarray*}
\limsup_{p_1\rightarrow+\infty}
\left[
\frac{1}{2^{p_1}}
\sum_{k=2^{p_1}+1}^{2^{p_1+1}-1}
\left\|
l_k^{(N)}
\right\|_{\overline{\D}}
\right]
& \leq &
\;\;\;\;\;\;\;\;\;\;\;\;\;\;\;\;\;\;\;\;\;\;\;\;\;\;\;\;\;\;\;\;\;\;\;\;\;\;\;\;\;\;\;\;
\;\;\;\;\;\;\;\;\;\;\;\;\;\;\;\;\;\;\;\;\;\;\;\;\;\;\;\;\;\;\;\;\;\;\;\;\;\;\;\;\;\;\;\;
\end{eqnarray*}
\begin{eqnarray}\nonumber
\;\;\;\;\;\;\;\;\;\;
& \leq &
2\pi\varepsilon\exp(3\pi)
+
\frac{1-2\varepsilon}{1-\varepsilon/4}
+
\lim_{p_1\rightarrow+\infty}
\left(
\frac{
\pi\exp(3\pi)\,
2^{P_1(\varepsilon)}
}{
2^{p_1}
}
+
2\pi\exp(3\pi)\,
\frac{p_1}{2^{p_1}}
\right)
\\\label{pfthspecial6}
& = &
2\pi\varepsilon\exp(3\pi)
+
\frac{1-2\varepsilon}{1-\varepsilon/4}
\,.
\end{eqnarray}

Finally, the estimates~(\ref{pfthspecial1}), (\ref{pfthspecial2})
and~(\ref{pfthspecial6}) together yield
\begin{eqnarray*}
\limsup_{N\rightarrow+\infty}
\left[
\frac{1}{N}
\sum_{k=1}^N
\left\|
l_k^{(N)}
\right\|_{\overline{\D}}
\right]
& \leq &
\frac{1}{2}
\left(
2\pi\varepsilon\exp(3\pi)
+
\frac{1-2\varepsilon}{1-\varepsilon/4}
\right)
+
\frac{1}{2}
\left(
2\pi\varepsilon\exp(3\pi)
+
\frac{1-2\varepsilon}{1-\varepsilon/4}
\right)
\\
& = &
2\pi\varepsilon\exp(3\pi)
+
\frac{1-2\varepsilon}{1-\varepsilon/4}
\,,
\end{eqnarray*}
and $\varepsilon>0$ being arbitrary, we can deduce that
\begin{eqnarray*}
\limsup_{N\rightarrow+\infty}
\left[
\frac{1}{N}
\sum_{k=1}^N
\left\|
l_k^{(N)}
\right\|_{\overline{\D}}
\right]
& \leq &
\lim_{\varepsilon\rightarrow0,\,\varepsilon>0}
\left(
2\pi\varepsilon\exp(3\pi)
+
\frac{1-2\varepsilon}{1-\varepsilon/4}
\right)
\;=\;
1
\,.
\end{eqnarray*}

\medskip

This last estimate and~(\ref{pfthspecial0}) lead to
\begin{eqnarray*}
1
\;\leq\;
\liminf_{N\rightarrow+\infty}
\left[
\frac{1}{N}
\sum_{k=1}^N
\left\|
l_k^{(N)}
\right\|_{\overline{\D}}
\right]
\;\leq\;
\limsup_{N\rightarrow+\infty}
\left[
\frac{1}{N}
\sum_{k=1}^N
\left\|
l_k^{(N)}
\right\|_{\overline{\D}}
\right]
\;\leq\;
1
\end{eqnarray*}
and this proves~(\ref{specialN1}). In addition, this leads with~(\ref{specialNlow1}) 
for all $N\geq3$ to
\begin{eqnarray*}
N
\;<\;
\sum_{k=1}^N
\left\|
l_k^{(N)}
\right\|_{\overline{\D}}
\;=\;
N\left(1+\varepsilon(1/N)\right)
\,,
\end{eqnarray*}
and this proves~(\ref{specialN2}) and completes the whole proof
of the theorem (since $N=2^{p_1+1}-1\geq3$ iff $p_1+1\geq2$).

\end{proof}

\bigskip

\section{On the case of compact sets with Alper-smooth Jordan boundary}

In this part we deal with the case of a compact set $K$ whose boundary is an Alper-smooth
Jordan curve, and $\Phi$ denotes the exterior conformal mapping from
$\overline{\C}\setminus\D$ onto $\overline{\C}\setminus K$. We remind that
$\Gamma=\partial K$ is an Alper-smooth Jordan curve if the modulus of continuity 
$\omega$ of the angle
$\theta(s)$ between the tangent at $\Gamma(s)$ and the positive real axis
(where $s$ is the arc-length parameter), satisfies (see~\cite{kovaripom} for the terminology)
\begin{eqnarray*}
\int_0^h\dfrac{\omega(x)}{x}\left|\ln x\right|dx
\;<\;\infty
\,.
\end{eqnarray*}
In particular, twice continuously differentiable Jordan curves are
Alper-smooth.

We can then give the proof of Theorem~\ref{unifestimcompact}. 
We first remind the following result that is Lemma~3 from~\cite{bialascalvi}.

\begin{lemma}\label{lemma3bialascalvi}

Let $K$ be a compact set whose boundary is an Alper-smooth Jordan curve.
Let $\Phi$ be the conformal mapping of the exterior of the unit disk on the interior of $K$.
Lastly, let $\left(a_j\right)_{j\geq0}$ be a Leja sequence for the unit disk with $\left|a_0\right|=1$.
Then for any $z$ on the unit disk and $n\in\N^{*}$, we have
\begin{eqnarray} \label{alper2}
\dfrac{C(K)^n}{c_n}
\;\leq\;
\left|
\prod_{j=0}^{n-1}
\dfrac{\Phi(z)-\Phi(a_j)}{z-a_j}
\right|
\;\leq\;
C(K)^nc_n
\,,
\end{eqnarray}
where $C(K)$ is the logarithmic capacity of $K$,
$c_n\leq (n+1)^{A/\ln(2)}$ and $A$ is a positive constant
depending only on $K$.

\end{lemma}

We also remind an important property of $\Phi$, that can be found 
in~\cite[Sections~1 and~2]{alper1955} or~\cite[Eq.~(3) p.~45]{alper1956}.
There exist positive constants $M_1$ and $M_2$ such that for all
$z,w$ in the unit circle with $z\neq w$,
\begin{eqnarray}\label{alper1}
M_1
\;\leq\;
\left|
\dfrac{\Phi(z)-\Phi(w)}{z-w}
\right|
\;\leq\;
M_2
\,.
\end{eqnarray}

Now we can give the proof of Theorem~\ref{unifestimcompact}.

\begin{proof}

First, it is sufficient to prove that
for all $N\geq1$, $p=1,\ldots,N$ and $z$ on the unit circle, we have
\begin{eqnarray}\label{compactaux1}
\left|
\prod_{j=1,j\neq p}^N
\dfrac{\Phi(z)-\Phi(\eta_j)}{\Phi(\eta_p)-\Phi(\eta_j)}
\right|
& \leq &
M
(N+1)^{2A/\ln(2)}
\,
\left|
\prod_{j=1,j\neq p}^N
\dfrac{z-\eta_j}{\eta_p-\eta_j}
\right|
\,,
\end{eqnarray}
where $M$ and $A$ are positive constants depending only on $K$.
Indeed, let fix $N\geq1$ and $p=1,\ldots,N$. By the maximum modulus principle, we will have that
\begin{eqnarray*}
\sup_{w\in K}
\left|
\prod_{j=1,j\neq p}^N
\dfrac{w-\Phi(\eta_j)}{\Phi(\eta_p)-\Phi(\eta_j)}
\right|
& = &
\left|
\prod_{j=1,j\neq p}^N
\dfrac{w^*-\Phi(\eta_j)}{\Phi(\eta_p)-\Phi(\eta_j)}
\right|
\,,
\end{eqnarray*}
where $w^*\in\partial K$. Let be $z^*$ in the circle such that
$\Phi\left(z^*\right)=w^*$. It will follow that
\begin{eqnarray*}
\sup_{w\in K}
\left|
\prod_{j=1,j\neq p}^N
\dfrac{w-\Phi(\eta_j)}{\Phi(\eta_p)-\Phi(\eta_j)}
\right|
& = &
\left|
\prod_{j=1,j\neq p}^N
\dfrac{\Phi\left(z^*\right)-\Phi(\eta_j)}{\Phi(\eta_p)-\Phi(\eta_j)}
\right|
\\
& \leq &
M
(N+1)^{2A/\ln(2)}
\,
\left|
\prod_{j=1,j\neq p}^N
\dfrac{z^*-\eta_j}{\eta_p-\eta_j}
\right|
\\
& \leq &
\pi\exp(3\pi)\,M
(N+1)^{2A/\ln(2)}
\,,
\end{eqnarray*}
the last inequality being an application of Theorem~\ref{unifestim}
(since $\left(\eta_j\right)_{j\geq1}$ is a Leja sequence for
the unit disk with $\left|\eta_1\right|=1$).

In order to prove~(\ref{compactaux1}), we use the same method as for the proof of
Theorem~13 from~\cite{calviphung1}. We can assume that $N\geq2$ (otherwise,
the required estimate is obvious), we fix an $N$-Leja section for the unit disk
and consider $z$ on the unit circle with
$z\neq\eta_j$, $\forall\,j=1,\ldots,N$. One has for all $p=1,\ldots,N$,
\begin{eqnarray*}
\left|
\prod_{j=1,j\neq p}^N
\dfrac{\Phi(z)-\Phi(\eta_j)}{z-\eta_j}
\right|
& = &
\left|
\dfrac{z-\eta_p}{\Phi(z)-\Phi(\eta_p)}
\right|
\left|
\prod_{j=1}^N
\dfrac{\Phi(z)-\Phi(\eta_j)}{z-\eta_j}
\right|
\\\nonumber
& = &
\left|
\dfrac{z-\eta_p}{\Phi(z)-\Phi(\eta_p)}
\right|
\left|
\prod_{j=0}^{N-1}
\dfrac{\Phi(z)-\Phi(\eta_{j+1})}{z-\eta_{j+1}}
\right|
\,.
\end{eqnarray*}
On the one hand, an application of~(\ref{alper1}) with $w=\eta_p$, and on the other hand
an application of~(\ref{alper2}) with the $N$-Leja section
$\left\{a_j\right\}_{0\leq j\leq N-1}=\left\{\eta_{j+1}\right\}_{0\leq j\leq N-1}$, 
together give that
\begin{eqnarray*}
\dfrac{C(K)^N}{M_2\,c_N}
\;\leq\;
\left|
\prod_{j=1,j\neq p}^N
\dfrac{\Phi(z)-\Phi(\eta_j)}{z-\eta_j}
\right|
\;\leq\;
\dfrac{C(K)^Nc_N}{M_1}
\,.
\end{eqnarray*}
In particular, these estimates remain true for $z=\eta_p$ by continuity, and lead to
\begin{eqnarray*}
\begin{cases}
\;
\prod_{j=1,j\neq p}^N
\left|
\Phi(z)-\Phi\left(\eta_j\right)
\right|
\;
\;\leq\;
\dfrac{C(K)^Nc_N}{M_1}
\prod_{j=1,j\neq p}^N
\left|
z-\eta_j
\right|
\,,
\\
\prod_{j=1,j\neq p}^N
\left|
\Phi\left(\eta_p\right)-\Phi\left(\eta_j\right)
\right|
\;\geq\;
\dfrac{C(K)^N}{M_2\,c_N}
\prod_{j=1,j\neq p}^N
\left|
\eta_p-\eta_j
\right|
\,.
\end{cases}
\end{eqnarray*}
It follows that 
\begin{eqnarray}\label{unifestimcompactaux}
\left|
\prod_{j=1,j\neq p}^N
\dfrac{\Phi(z)-\Phi(\eta_j)}{\Phi(\eta_p)-\Phi(\eta_j)}
\right|
& \leq &
\dfrac{M_2\,c_N^2}{M_1}
\left|
\prod_{j=1,j\neq p}^N
\dfrac{z-\eta_j}{\eta_p-\eta_j}
\right|
\,.
\end{eqnarray}
Lastly, by continuity (and since
$c_N\leq(N+1)^{A /\ln(2)}$), the above inequality holds for every $z$ on the unit circle
and yields the required estimate~(\ref{compactaux1}).

\end{proof}

\bigskip

\section{Applications in multivariate Lagrange interpolation} \label{applmultivar}

\subsection{An explicit formula of the Lagrange polynomials for intertwining sequences}

For any given $N\geq1$, we remind from Subsection~\ref{applmultivarintro} of the Introduction,
the numbers $n$ and $m$ with $0\leq m\leq n$, the space
$\mathcal{P}_{n,m}$ that is the linear subspace of $\mathcal{P}_n$
spanned by $\mathcal{P}_{n-1}$ and the monomials
$z^n,z^{n-1}w,\ldots,z^{n-m}w^m$, and the set
$\Omega_N=\Omega_{n,m}=\left\{H_k\right\}_{1\leq k\leq N}
=\left\{
\left(\eta_0,\theta_0\right),\ldots,
\left(\eta_0,\theta_{n-1}\right),
\left(\eta_n,\theta_0\right),
\ldots,\left(\eta_{n-m},\theta_m\right)
\right\}$ (for fixed sequences
$\left(\eta_k\right)_{k\geq0}$ and $\left(\theta_l\right)_{l\geq0}$
of pairwise distinct elements).
We give the proof of the following result mentioned as Proposition~\ref{explicitlagrangepol}
in the Introduction, and that gives an explicit formula for the bidimensional
fundamental Lagrange polynomials associated with $\Omega_N$ (FLIPs).

\begin{proposition}\label{explicitlagrangepol}

Let be $N\geq1$ and the associated $n$, $m$, $\mathcal{P}_{n,m}$ and $\Omega_N$.
Then the multivariate fundamental Lagrange polynomials
$l_{H_k}^{(N)}$ exist, i.e. for all $k=1,\ldots,N$,
$l_{H_k}^{(N)}\in\mathcal{P}_{n,m}$ and
$l_{H_k}^{(N)}\left(H_l\right)=\delta_{k,l}$
for all $l=1,\ldots,N$.

In addition, let fix $k=1,\ldots,N$ and let be
$p,\,q$ such that $H_k=\left(\eta_p,\theta_q\right)$. We have
for all $(z,w)\in\C^2$:

\begin{itemize}

\item

if $p+q=n$, or $p+q=n-1$ and $q\geq m+1$,
then
\begin{eqnarray}\label{nn-1m+1}
l_{(\eta_p,\theta_q)}^{(N)}(z,w)
& = &
\prod_{j=0}^{p-1}\dfrac{z-\eta_j}{\eta_p-\eta_j}
\,\times\,
\prod_{i=0}^{q-1}\dfrac{w-\theta_i}{\theta_q-\theta_i}
\,;
\end{eqnarray}

\item

if $p+q=n-1$ and $q=m$, then 
\begin{eqnarray}\nonumber
l_{(\eta_p,\theta_q)}^{(N)}(z,w)
& = &
l_{(\eta_{n-m-1},\theta_m)}^{(N)}(z,w)
\\\label{n-1m}
& = &
\prod_{j=0,j\neq {n-m-1}}^{n-m}\dfrac{z-\eta_j}{\eta_{n-m-1}-\eta_j}
\,\times\,
\prod_{i=0}^{m-1}\dfrac{w-\theta_i}{\theta_m-\theta_i}
\,;
\end{eqnarray}

\item

if $p+q=n-1$ and $0\leq q\leq m-1$, 
then
\begin{eqnarray}\label{n-1m-1}
l_{(\eta_p,\theta_q)}^{(N)}(z,w)
& = &
\;\;\;\;\;\;\;\;\;\;\;\;\;\;\;\;\;\;\;\;\;\;\;\;\;\;\;\;\;\;\;\;\;\;\;\;
\;\;\;\;\;\;\;\;\;\;\;\;\;\;\;\;\;\;\;\;\;\;\;\;\;\;\;\;\;\;\;\;\;\;\;\;
\end{eqnarray}
\begin{eqnarray*}
=\;
\prod_{j=0,j\neq p}^{p+1}\dfrac{z-\eta_j}{\eta_p-\eta_j}
\,
\prod_{i=0}^{q-1}\dfrac{w-\theta_i}{\theta_q-\theta_i}
\,-\,
\prod_{j=0}^{p-1}\dfrac{z-\eta_j}{\eta_p-\eta_j}
\,
\prod_{i=0}^{q-1}\dfrac{w-\theta_i}{\theta_q-\theta_i}
\,+\,
\prod_{j=0}^{p-1}\dfrac{z-\eta_j}{\eta_p-\eta_j}
\,
\prod_{i=0,i\neq q}^{q+1}\dfrac{w-\theta_i}{\theta_q-\theta_i}
\,;
\end{eqnarray*}

\item

if $0 \leq p+q\leq n-2$, $0\leq q\leq m-1$ and $0\leq p\leq n-m-1$, then
\begin{eqnarray}\label{n-2m-1n-m-1}
l_{(\eta_p,\theta_q)}^{(N)}(z,w)
& = &
\;\;\;\;\;\;\;\;\;\;\;\;\;\;\;\;\;\;\;\;\;\;\;\;\;\;\;\;\;\;\;\;\;\;\;\;
\;\;\;\;\;\;\;\;\;\;\;\;\;\;\;\;\;\;\;\;\;\;\;\;\;\;\;\;\;\;\;\;\;\;\;\;
\end{eqnarray}
\begin{eqnarray*}
& = &
\prod_{j=0,j\neq p}^{n-q}\dfrac{z-\eta_j}{\eta_p-\eta_j}
\prod_{i=0}^{q-1}\dfrac{w-\theta_i}{\theta_q-\theta_i}
-
\prod_{j=0,j\neq p}^{n-q-1}\dfrac{z-\eta_j}{\eta_p-\eta_j}
\prod_{i=0}^{q-1}\dfrac{w-\theta_i}{\theta_q-\theta_i}
+
\prod_{j=0,j\neq p}^{n-q-1}\dfrac{z-\eta_j}{\eta_p-\eta_j}
\prod_{i=0,i\neq q}^{q+1}\dfrac{w-\theta_i}{\theta_q-\theta_i}
\\
& &
+\;
\sum_{r=1}^{m-q-1}
\left[
\prod_{j=0,j\neq p}^{n-q-r-1}\dfrac{z-\eta_j}{\eta_p-\eta_j}
\prod_{i=0,i\neq q}^{q+r+1}\dfrac{w-\theta_i}{\theta_q-\theta_i}
\;-\;
\prod_{j=0,j\neq p}^{n-q-r-1}\dfrac{z-\eta_j}{\eta_p-\eta_j}
\prod_{i=0,i\neq q}^{q+r}\dfrac{w-\theta_i}{\theta_q-\theta_i}
\right]
\\
& &
+\;
\sum_{r=m-q}^{n-p-q-2}
\left[
\prod_{j=0,j\neq p}^{n-q-r-2}\dfrac{z-\eta_j}{\eta_p-\eta_j}
\prod_{i=0,i\neq q}^{q+r+1}\dfrac{w-\theta_i}{\theta_q-\theta_i}
\;-\;
\prod_{j=0,j\neq p}^{n-q-r-2}\dfrac{z-\eta_j}{\eta_p-\eta_j}
\prod_{i=0,i\neq q}^{q+r}\dfrac{w-\theta_i}{\theta_q-\theta_i}
\right]
\,,
\end{eqnarray*}
with the convention that
$\sum_{\emptyset}=0$ if $q\geq m-1$ or $p\geq n-m-1$
(similarly, we set $\prod_{\emptyset}=1$);

\item

if $0 \leq p+q\leq n-2$, $0\leq q\leq m-1$ and $p\geq n-m$, then
\begin{eqnarray}\label{n-2m-1n-m}
l_{(\eta_p,\theta_q)}^{(N)}(z,w)
& = &
\;\;\;\;\;\;\;\;\;\;\;\;\;\;\;\;\;\;\;\;\;\;\;\;\;\;\;\;\;\;\;\;\;\;\;\;
\;\;\;\;\;\;\;\;\;\;\;\;\;\;\;\;\;\;\;\;\;\;\;\;\;\;\;\;\;\;\;\;\;\;\;\;
\end{eqnarray}
\begin{eqnarray*}
& = &
\prod_{j=0,j\neq p}^{n-q}\dfrac{z-\eta_j}{\eta_p-\eta_j}
\prod_{i=0}^{q-1}\dfrac{w-\theta_i}{\theta_q-\theta_i}
-
\prod_{j=0,j\neq p}^{n-q-1}\dfrac{z-\eta_j}{\eta_p-\eta_j}
\prod_{i=0}^{q-1}\dfrac{w-\theta_i}{\theta_q-\theta_i}
+
\prod_{j=0,j\neq p}^{n-q-1}\dfrac{z-\eta_j}{\eta_p-\eta_j}
\prod_{i=0,i\neq q}^{q+1}\dfrac{w-\theta_i}{\theta_q-\theta_i}
\\
& &
+\;
\sum_{r=1}^{n-p-q-1}
\left[
\prod_{j=0,j\neq p}^{n-q-r-1}\dfrac{z-\eta_j}{\eta_p-\eta_j}
\prod_{i=0,i\neq q}^{q+r+1}\dfrac{w-\theta_i}{\theta_q-\theta_i}
\;-\;
\prod_{j=0,j\neq p}^{n-q-r-1}\dfrac{z-\eta_j}{\eta_p-\eta_j}
\prod_{i=0,i\neq q}^{q+r}\dfrac{w-\theta_i}{\theta_q-\theta_i}
\right]
\,;
\end{eqnarray*}

\item
if $p+q\leq n-2$ and $q=m$, then 
\begin{eqnarray}\label{n-2m}
l_{(\eta_p,\theta_q)}^{(N)}(z,w)
\;=\;
l_{(\eta_p,\theta_m)}^{(N)}(z,w)
& = &
\;\;\;\;\;\;\;\;\;\;\;\;\;\;\;\;\;\;\;\;\;\;\;\;\;\;\;\;\;\;\;\;\;\;\;\;
\;\;\;\;\;\;\;\;\;\;\;\;\;\;
\end{eqnarray}
\begin{eqnarray*}
& = &
\prod_{j=0,j\neq p}^{n-m}\dfrac{z-\eta_j}{\eta_p-\eta_j}
\prod_{i=0}^{m-1}\dfrac{w-\theta_i}{\theta_m-\theta_i}
-
\prod_{j=0,j\neq p}^{n-m-2}\dfrac{z-\eta_j}{\eta_p-\eta_j}
\prod_{i=0}^{m-1}\dfrac{w-\theta_i}{\theta_m-\theta_i}
+
\prod_{j=0,j\neq p}^{n-m-2}\dfrac{z-\eta_j}{\eta_p-\eta_j}
\prod_{i=0,i\neq m}^{m+1}\dfrac{w-\theta_i}{\theta_m-\theta_i}
\\
& &
+\;
\sum_{r=1}^{n-m-p-2}
\left[
\prod_{j=0,j\neq p}^{n-m-r-2}\dfrac{z-\eta_j}{\eta_p-\eta_j}
\prod_{i=0,i\neq m}^{m+r+1}\dfrac{w-\theta_i}{\theta_m-\theta_i}
\;-\;
\prod_{j=0,j\neq p}^{n-m-r-2}\dfrac{z-\eta_j}{\eta_p-\eta_j}
\prod_{i=0,i\neq m}^{m+r}\dfrac{w-\theta_i}{\theta_m-\theta_i}
\right]
\,;
\end{eqnarray*}

\item

if $p+q\leq n-2$ and $q\geq m+1$, then 
\begin{eqnarray}\label{n-2m+1}
l_{(\eta_p,\theta_q)}^{(N)}(z,w)
& = &
\;\;\;\;\;\;\;\;\;\;\;\;\;\;\;\;\;\;\;\;\;\;\;\;\;\;\;\;\;\;\;\;\;\;\;\;
\;\;\;\;\;\;\;\;\;\;\;\;\;\;\;\;\;\;\;\;\;\;\;\;\;\;\;\;\;\;\;\;\;\;\;\;
\end{eqnarray}
\begin{eqnarray*}
& = &
\prod_{j=0,j\neq p}^{n-q-1}\dfrac{z-\eta_j}{\eta_p-\eta_j}
\prod_{i=0}^{q-1}\dfrac{w-\theta_i}{\theta_q-\theta_i}
-
\prod_{j=0,j\neq p}^{n-q-2}\dfrac{z-\eta_j}{\eta_p-\eta_j}
\prod_{i=0}^{q-1}\dfrac{w-\theta_i}{\theta_q-\theta_i}
+
\prod_{j=0,j\neq p}^{n-q-2}\dfrac{z-\eta_j}{\eta_p-\eta_j}
\prod_{i=0,i\neq q}^{q+1}\dfrac{w-\theta_i}{\theta_q-\theta_i}
\\
& &
+\;
\sum_{r=1}^{n-p-q-2}
\left[
\prod_{j=0,j\neq p}^{n-q-r-2}\dfrac{z-\eta_j}{\eta_p-\eta_j}
\prod_{i=0,i\neq q}^{q+r+1}\dfrac{w-\theta_i}{\theta_q-\theta_i}
\;-\;
\prod_{j=0,j\neq p}^{n-q-r-2}\dfrac{z-\eta_j}{\eta_p-\eta_j}
\prod_{i=0,i\neq q}^{q+r}\dfrac{w-\theta_i}{\theta_q-\theta_i}
\right]
\,.
\end{eqnarray*}

\end{itemize}

\end{proposition}

The proof of this proposition consists of two steps:
the following lemma and the next corollary.

\begin{lemma}\label{prexplicitlagrangepol}

For all $\left(\eta_p,\theta_q\right)\in\Omega_N$,
the function that appears in the claimed equality~(\ref{nn-1m+1})
(resp., (\ref{n-1m}), (\ref{n-1m-1}), (\ref{n-2m-1n-m-1}), (\ref{n-2m-1n-m}), (\ref{n-2m})
and~(\ref{n-2m+1})) satisfies the required properties for a FLIP $l_{(\eta_p,\theta_q)}^{(N)}$,
i.e.:
\begin{enumerate}

\item\label{prexplicitlagrangepol1}

$l_{(\eta_p,\theta_q)}^{(N)}
\;\in\;
\mathcal{P}_{n,m}$\,;

\item\label{prexplicitlagrangepol2}

$l_{(\eta_p,\theta_q)}^{(N)}\left(\eta_k,\theta_l\right)
\;=\;
\delta_{p,k}\,\delta_{q,l}
\;\;\;
\forall\,\left(\eta_k,\theta_l\right)\in\Omega_N$.

\end{enumerate}

\end{lemma}

\begin{proof}

First, all the involved polynomials are well-defined since the $\eta_j$'s
(resp., $\theta_i$'s) are supposed to be pairwise distinct.

\medskip

\underline{Proof of~(\ref{n-2m-1n-m-1}):}
we deal with the case $p+q\leq n-2$, $0\leq q\leq m-1$ and $p\leq n-m-1$.
We begin with part~(\ref{prexplicitlagrangepol1}) in the statement of the lemma:
we want to prove that the involved polynomial belongs to
$\mathcal{P}_{n,m}$, i.e. it has total degree at most $n$, and the 
terms whose total degree equals $n$ must have partial degree at most
$m$ with respect to the second variable $w$.
Here, all the involved products have total degree
at most $n$. Next, one has
$\deg_w\left(\prod_{j=0,j\neq p}^{n-q}\dfrac{z-\eta_j}{\eta_p-\eta_j}
\prod_{i=0}^{q-1}\dfrac{w-\theta_i}{\theta_q-\theta_i}\right)=q\leq m-1$.
Similarly, 
$\deg_w\left(\prod_{j=0,j\neq p}^{n-q-1}\dfrac{z-\eta_j}{\eta_p-\eta_j}
\prod_{i=0}^{q+1}\dfrac{w-\theta_i}{\theta_q-\theta_i}\right)=q+1\leq m$,
and
for all $r=1,\ldots,m-q-1$ (in case $m\geq q+2$),\\
$\deg_w\left(\prod_{j=0,j\neq p}^{n-q-r-1}\dfrac{z-\eta_j}{\eta_p-\eta_j}
\prod_{i=0}^{q+r+1}\dfrac{w-\theta_i}{\theta_q-\theta_i}\right)=q+r+1\leq m$.
All the remaining products have total degree at most $n-1$, then this proves~(\ref{prexplicitlagrangepol1}).

Now we prove part~(\ref{prexplicitlagrangepol2}) in the statement of the lemma.
First, one has $n-q>n-q-1\geq p+1>p$. 
On the other hand,
for all $r=1,\ldots,m-q-1$ (in case $m-q\geq2$, otherwise the associated sum does not
even appear), one has 
$n-q-r-1\geq n-m\geq p+1>p$. Similarly, 
for all $r=m-q,\ldots,n-p-q-2$ (in case $p\leq n-m-2$), one has 
$n-q-r-2\geq p>p-1$. 
It follows that the expression in~(\ref{n-2m-1n-m-1}) is divisible by
$\prod_{j=0}^{p-1}\dfrac{z-\eta_j}{\eta_p-\eta_j}$.
We similarly check that it is also divisible by 
$\prod_{i=0}^{q-1}\dfrac{w-\theta_i}{\theta_q-\theta_i}$.
Then it cancels all the points
$\left(\eta_k,\theta_l\right)$ with $0\leq k\leq p-1$ or $0\leq l\leq q-1$.
Thus in the following, it is sufficient to check~(\ref{prexplicitlagrangepol2})
for $(z,w)=\left(\eta_k,\theta_l\right)\in\Omega_N$ with
$k\geq p$ and $l\geq q$.

Next, if we fix $z=\eta_p$, the expression in~(\ref{n-2m-1n-m-1}) gives
\begin{eqnarray*}
1\times\prod_{i=0}^{q-1}\dfrac{w-\theta_i}{\theta_q-\theta_i}
-
1\times\prod_{i=0}^{q-1}\dfrac{w-\theta_i}{\theta_q-\theta_i}
+
1\times\prod_{i=0,i\neq q}^{q+1}\dfrac{w-\theta_i}{\theta_q-\theta_i}
& &
\\
+\;
\sum_{r=1}^{m-q-1}
\left[
1\times\prod_{i=0,i\neq q}^{q+r+1}\dfrac{w-\theta_i}{\theta_q-\theta_i}
-
1\times\prod_{i=0,i\neq q}^{q+r}\dfrac{w-\theta_i}{\theta_q-\theta_i}
\right]
& &
\\
+\;
\sum_{r=m-q}^{n-p-q-2}
\left[
1\times\prod_{i=0,i\neq q}^{q+r+1}\dfrac{w-\theta_i}{\theta_q-\theta_i}
-
1\times\prod_{i=0,i\neq q}^{q+r}\dfrac{w-\theta_i}{\theta_q-\theta_i}
\right]
& = &
\end{eqnarray*}
\begin{eqnarray*}
& = &
\prod_{i=0,i\neq q}^{q+1}\dfrac{w-\theta_i}{\theta_q-\theta_i}
+
\prod_{i=0,i\neq q}^{m}\dfrac{w-\theta_i}{\theta_q-\theta_i}
-
\prod_{i=0,i\neq q}^{q+1}\dfrac{w-\theta_i}{\theta_q-\theta_i}
+
\prod_{i=0,i\neq q}^{n-p-1}\dfrac{w-\theta_i}{\theta_q-\theta_i}
-
\prod_{i=0,i\neq q}^{m}\dfrac{w-\theta_i}{\theta_q-\theta_i}
\\
& = &
\prod_{i=0,i\neq q}^{n-p-1}\dfrac{w-\theta_i}{\theta_q-\theta_i}
\end{eqnarray*}
(notice that the above equalities hold if $q=m-1$ or $p=n-m-1$).
We first get $1$ for $w=\theta_q$. If $w=\theta_l$ with $l\geq q+1$ and $p+l\leq n-1$, then
$q+1\leq l\leq n-p-1$ and the above product vanishes. Lastly, if
$w=\theta_l$ with $p+l=n$, necessarily $l\leq m$, i.e.
$n-p=l\leq m$. This is impossible since $p\leq n-m-1$ by assumption.

Now if we fix $w=\theta_q$ in~(\ref{n-2m-1n-m-1}),
we get
\begin{eqnarray*}
\prod_{j=0,j\neq p}^{n-q}\dfrac{z-\eta_j}{\eta_p-\eta_j}\times1
-
\prod_{j=0,j\neq p}^{n-q-1}\dfrac{z-\eta_j}{\eta_p-\eta_j}\times1
+
\prod_{j=0,j\neq p}^{n-q-1}\dfrac{z-\eta_j}{\eta_p-\eta_j}\times1
& &
\\
+\;
\sum_{r=1}^{m-q-1}
\left[
\prod_{j=0,j\neq p}^{n-q-r-1}\dfrac{z-\eta_j}{\eta_p-\eta_j}\times1
-
\prod_{j=0,j\neq p}^{n-q-r-1}\dfrac{z-\eta_j}{\eta_p-\eta_j}\times1
\right]
& &
\\
+\;
\sum_{r=m-q}^{n-p-q-2}
\left[
\prod_{j=0,j\neq p}^{n-q-r-2}\dfrac{z-\eta_j}{\eta_p-\eta_j}\times1
-
\prod_{j=0,j\neq p}^{n-q-r-2}\dfrac{z-\eta_j}{\eta_p-\eta_j}\times1
\right]
& = &
\prod_{j=0,j\neq p}^{n-q}\dfrac{z-\eta_j}{\eta_p-\eta_j}
\,.
\end{eqnarray*}
Again, we get $1$ if $z=\eta_p$. If $z=\eta_k$ with
$k\geq p+1$ and $k+q\leq n$, then
$p+1\leq k\leq n-q$ and the above product vanishes.

The remaining case is the one for which
$(z,w)=\left(\eta_k,\theta_l\right)\in\Omega_{N}$ with
$k\geq p+1$ and $l\geq q+1$.
Since $k+l\leq n$, one has $k\leq n-l\leq n-q-1$ (with $k\geq p+1$),
then 
$\prod_{j=0,j\neq p}^{n-q-1}\dfrac{\eta_k-\eta_j}{\eta_p-\eta_j}=0$
and the first three terms in~(\ref{n-2m-1n-m-1}) vanish.
Next, let be $r$ with $1\leq r\leq m-q-1$ (we can assume 
$m-q\geq2$, otherwise the first sum disappears). If
$k\leq n-q-r-1$, then
$\prod_{j=0,j\neq p}^{n-q-r-1}\dfrac{\eta_k-\eta_j}{\eta_p-\eta_j}=0$
(since $k\geq p+1$).
Otherwise, $k\geq n-q-r$ then $l\leq n-k\leq q+r$ (with $l\geq q+1$),
and 
$\prod_{i=0,i\neq q}^{q+r}\dfrac{\theta_l-\theta_i}{\theta_q-\theta_i}=0$.
It follows that the first sum in~(\ref{n-2m-1n-m-1}) vanishes.

Lastly, let be $r$ with $m-q\leq r\leq n-p-q-2$ (similarly, we can assume that
$n-p-2\geq m$).
If $k\leq n-q-r-2$ (with $k\geq p+1$), then
$\prod_{j=0,j\neq p}^{n-q-r-2}\dfrac{\eta_k-\eta_j}{\eta_p-\eta_j}=0$.
Otherwise, $k\geq n-q-r-1$. If $k+l\leq n-1$, then
$l\leq n-k-1\leq q+r$ (with $l\geq q+1$) and
$\prod_{i=0,i\neq q}^{q+r}\dfrac{\theta_l-\theta_i}{\theta_q-\theta_i}=0$;
otherwise, $k+l=n$ then necessarily $l\leq m\leq q+r$
(since $r\geq m-q$) and one still has
$\prod_{i=0,i\neq q}^{q+r}\dfrac{\theta_l-\theta_i}{\theta_q-\theta_i}=0$.
It follows that the second sum in~(\ref{n-2m-1n-m-1}) vanishes and this completes
the proof of~(\ref{n-2m-1n-m-1}).

\medskip

\underline{Proof of~(\ref{n-2m-1n-m}):}
now we assume that $p+q\leq n-2$, $q\leq m-1$ and $p\geq n-m$.
We begin with the proof of~(\ref{prexplicitlagrangepol1}). All the involved products have total degree
at most $n$, and
$\deg_w\left(\prod_{j=0,j\neq p}^{n-q}\dfrac{z-\eta_j}{\eta_p-\eta_j}
\prod_{i=0}^{q-1}\dfrac{w-\theta_i}{\theta_q-\theta_i}\right)=q\leq m-1$,
$\deg_w\left(\prod_{j=0,j\neq p}^{n-q-1}\dfrac{z-\eta_j}{\eta_p-\eta_j}
\prod_{i=0}^{q+1}\dfrac{w-\theta_i}{\theta_q-\theta_i}\right)=q+1\leq m$,
and
for all $r=1,\ldots,n-p-q-1$,
$\deg_w\left(\prod_{j=0,j\neq p}^{n-q-r-1}\dfrac{z-\eta_j}{\eta_p-\eta_j}
\prod_{i=0}^{q+r+1}\dfrac{w-\theta_i}{\theta_q-\theta_i}\right)=q+r+1\leq n-p\leq m$.
All the remaining products have total degree at most $n-1$, then this proves~(\ref{prexplicitlagrangepol1}).

For the proof of~(\ref{prexplicitlagrangepol2}), we first have
$n-q>n-q-1\geq p+1>p$. 
On the other hand,
for all $r=1,\ldots,n-p-q-1$, one has 
$n-q-r-1\geq p$, then the expression in~(\ref{n-2m-1n-m}) is divisible by
$\prod_{j=0}^{p-1}\dfrac{z-\eta_j}{\eta_p-\eta_j}$.
We also see that it is divisible by 
$\prod_{i=0}^{q-1}\dfrac{w-\theta_i}{\theta_q-\theta_i}$.
Thus it cancels all the points
$\left(\eta_k,\theta_l\right)$ with $0\leq k\leq p-1$ or $0\leq l\leq q-1$.

Next, if we fix $z=\eta_p$, the expression in~(\ref{n-2m-1n-m}) gives
\begin{eqnarray*}
1\times\prod_{i=0}^{q-1}\dfrac{w-\theta_i}{\theta_q-\theta_i}
-
1\times\prod_{i=0}^{q-1}\dfrac{w-\theta_i}{\theta_q-\theta_i}
+
1\times\prod_{i=0,i\neq q}^{q+1}\dfrac{w-\theta_i}{\theta_q-\theta_i}
& &
\\
+\;
\sum_{r=1}^{n-p-q-1}
\left[
1\times\prod_{i=0,i\neq q}^{q+r+1}\dfrac{w-\theta_i}{\theta_q-\theta_i}
-
1\times\prod_{i=0,i\neq q}^{q+r}\dfrac{w-\theta_i}{\theta_q-\theta_i}
\right]
& = &
\;\;\;\;\;\;\;\;\;\;\;\;\;\;\;\;\;\;\;\;\;\;\;\;\;\;\;\;\;\;\;\;\;\;\;\;\;\;\;\;\;\;\;\;\;\;\;\;
\end{eqnarray*}
\begin{eqnarray*}
& = &
\prod_{i=0,i\neq q}^{q+1}\dfrac{w-\theta_i}{\theta_q-\theta_i}
+
\prod_{i=0,i\neq q}^{n-p}\dfrac{w-\theta_i}{\theta_q-\theta_i}
-
\prod_{i=0,i\neq q}^{q+1}\dfrac{w-\theta_i}{\theta_q-\theta_i}
\;=\;
\prod_{i=0,i\neq q}^{n-p}\dfrac{w-\theta_i}{\theta_q-\theta_i}
\,.
\end{eqnarray*}
We first get $1$ for $w=\theta_q$. If $w=\theta_l$ then
$p+l\leq n$, i.e. $l\leq n-p$ (with $l\geq q+1$) and the above product vanishes. 

Now if we fix $w=\theta_q$ in~(\ref{n-2m-1n-m}),
we get
\begin{eqnarray*}
\prod_{j=0,j\neq p}^{n-q}\dfrac{z-\eta_j}{\eta_p-\eta_j}\times1
-
\prod_{j=0,j\neq p}^{n-q-1}\dfrac{z-\eta_j}{\eta_p-\eta_j}\times1
+
\prod_{j=0,j\neq p}^{n-q-1}\dfrac{z-\eta_j}{\eta_p-\eta_j}\times1
& &
\\
+\;
\sum_{r=1}^{n-p-q-1}
\left[
\prod_{j=0,j\neq p}^{n-q-r-1}\dfrac{z-\eta_j}{\eta_p-\eta_j}\times1
-
\prod_{j=0,j\neq p}^{n-q-r-1}\dfrac{z-\eta_j}{\eta_p-\eta_j}\times1
\right]
& = &
\prod_{j=0,j\neq p}^{n-q}\dfrac{z-\eta_j}{\eta_p-\eta_j}
\,.
\end{eqnarray*}
Again, we get $1$ if $z=\eta_p$. If $z=\eta_k$ with
$k\geq p+1$ and $k+q\leq n$, then
$p+1\leq k\leq n-q$ and the above product vanishes.

The last case that we have to check is the one for which
$(z,w)=\left(\eta_k,\theta_l\right)\in\Omega_{N}$ with
$k\geq p+1$ and $l\geq q+1$.
Since $k+l\leq n$, one has $k\leq n-l\leq n-q-1$ (with $k\geq p+1$),
then 
$\prod_{j=0,j\neq p}^{n-q-1}\dfrac{\eta_k-\eta_j}{\eta_p-\eta_j}=0$
and the first three terms in~(\ref{n-2m-1n-m}) vanish.
Next, let be $r$ with $1\leq r\leq n-p-q-1$.
If $k\leq n-q-r-1$, then
$\prod_{j=0,j\neq p}^{n-q-r-1}\dfrac{\eta_k-\eta_j}{\eta_p-\eta_j}=0$
(since $k\geq p+1$).
Otherwise, $k\geq n-q-r$ then
$l\leq n-k\leq q+r$ (with $l\geq q+1$) and
$\prod_{i=0,i\neq q}^{q+r}\dfrac{\theta_l-\theta_i}{\theta_q-\theta_i}=0$.
It follows that the sum in~(\ref{n-2m-1n-m}) vanishes and this completes
the proof of~(\ref{n-2m-1n-m}).

\medskip

\underline{Proof of~(\ref{n-2m}):}
in this part we assume that $p+q\leq n-2$ and $q=m$. 
First, all the involved products have total degree
at most $n$, and
$\deg_w\left(\prod_{j=0,j\neq p}^{n-m}\dfrac{z-\eta_j}{\eta_p-\eta_j}
\prod_{i=0}^{m-1}\dfrac{w-\theta_i}{\theta_m-\theta_i}\right)=m$.
All the remaining products have total degree at most $n-1$, and this proves~(\ref{prexplicitlagrangepol1}).

Now we prove~(\ref{prexplicitlagrangepol2}).
First (since $p+m=p+q\leq n-2$), one has $n-m>n-m-2\geq p$. 
On the other hand,
for all $r=1,\ldots,n-m-p-2$ (in case $n-m-p\geq3$), one has 
$n-m-r-2\geq p$, then the expression in~(\ref{n-2m}) is divisible by
$\prod_{j=0}^{p-1}\dfrac{z-\eta_j}{\eta_p-\eta_j}$.
We also see that it is divisible by 
$\prod_{i=0}^{m-1}\dfrac{w-\theta_i}{\theta_m-\theta_i}$.
Thus it cancels all the points
$\left(\eta_k,\theta_l\right)$ with $0\leq k\leq p-1$ or $0\leq l\leq m-1$.

Next, if we fix $z=\eta_p$, the expression in~(\ref{n-2m}) gives
\begin{eqnarray*}
1\times\prod_{i=0}^{m-1}\dfrac{w-\theta_i}{\theta_m-\theta_i}
-
1\times\prod_{i=0}^{m-1}\dfrac{w-\theta_i}{\theta_m-\theta_i}
+
1\times\prod_{i=0,i\neq m}^{m+1}\dfrac{w-\theta_i}{\theta_m-\theta_i}
& &
\\
+\;
\sum_{r=1}^{n-m-p-2}
\left[
1\times\prod_{i=0,i\neq m}^{m+r+1}\dfrac{w-\theta_i}{\theta_m-\theta_i}
-
1\times\prod_{i=0,i\neq m}^{m+r}\dfrac{w-\theta_i}{\theta_m-\theta_i}
\right]
& = &
\;\;\;\;\;\;\;\;\;\;\;\;\;\;\;\;\;\;\;\;\;\;\;\;\;\;\;\;\;\;\;\;\;\;\;\;\;\;\;\;\;\;\;\;\;\;\;\;
\end{eqnarray*}
\begin{eqnarray*}
& = &
\prod_{i=0,i\neq m}^{m+1}\dfrac{w-\theta_i}{\theta_m-\theta_i}
+
\prod_{i=0,i\neq m}^{n-p-1}\dfrac{w-\theta_i}{\theta_m-\theta_i}
-
\prod_{i=0,i\neq m}^{m+1}\dfrac{w-\theta_i}{\theta_m-\theta_i}
\;=\;
\prod_{i=0,i\neq m}^{n-p-1}\dfrac{w-\theta_i}{\theta_m-\theta_i}
\end{eqnarray*}
(notice that this equality holds if $n-m-p-2=0$).
We first get $1$ for $w=\theta_m$. If $w=\theta_l$ then
$p+l\leq n$, i.e. $l\leq n-p$ (with $l\geq m+1$) and the above product vanishes,
unless $l=n-p$. But this last case cannot occur
since we should have $p+l=n$ then necessarily $l\leq m$,
and this would yield $n=p+l\leq p+m=p+q\leq n-2<n$.

Now if we fix $w=\theta_m$ in~(\ref{n-2m}),
we get
\begin{eqnarray*}
\prod_{j=0,j\neq p}^{n-m}\dfrac{z-\eta_j}{\eta_p-\eta_j}\times1
-
\prod_{j=0,j\neq p}^{n-m-2}\dfrac{z-\eta_j}{\eta_p-\eta_j}\times1
+
\prod_{j=0,j\neq p}^{n-m-2}\dfrac{z-\eta_j}{\eta_p-\eta_j}\times1
& &
\\
+\;
\sum_{r=1}^{n-m-p-2}
\left[
\prod_{j=0,j\neq p}^{n-m-r-2}\dfrac{z-\eta_j}{\eta_p-\eta_j}\times1
-
\prod_{j=0,j\neq p}^{n-m-r-2}\dfrac{z-\eta_j}{\eta_p-\eta_j}\times1
\right]
& = &
\prod_{j=0,j\neq p}^{n-m}\dfrac{z-\eta_j}{\eta_p-\eta_j}
\,.
\end{eqnarray*}
Again, we get $1$ if $z=\eta_p$. If $z=\eta_k$ with
$k\geq p+1$ and $k+m\leq n$, then
$p+1\leq k\leq n-m$ and the above product vanishes.

The remaining case is the one for which
$(z,w)=\left(\eta_k,\theta_l\right)\in\Omega_{N}$ with
$k\geq p+1$ and $l\geq m+1$. First, one necessarily has that
$k+l\leq n-1$: indeed, if $k+l=n$, then $l\leq m$ and this is
impossible since $l\geq m+1$.
Thus $k\leq n-l-1\leq n-m-2$ (with $k\geq p+1$),
then 
$\prod_{j=0,j\neq p}^{n-m-2}\dfrac{\eta_k-\eta_j}{\eta_p-\eta_j}=0$
and the first three terms in~(\ref{n-2m}) vanish.
Next, let be $r$ with $1\leq r\leq n-m-p-2$ (we can assume that
$n-m-p\geq3$, otherwise the associated sum does not appear).
If $k\leq n-m-r-2$, then
$\prod_{j=0,j\neq p}^{n-m-r-2}\dfrac{\eta_k-\eta_j}{\eta_p-\eta_j}=0$
(since $k\geq p+1$).
Otherwise, $k\geq n-m-r-1$ then
$l\leq n-k-1\leq m+r$ (with $l\geq m+1$) and
$\prod_{i=0,i\neq m}^{m+r}\dfrac{\theta_l-\theta_i}{\theta_m-\theta_i}=0$.
It follows that the sum in~(\ref{n-2m}) vanishes and this completes
the proof of~(\ref{n-2m}).

\medskip

\underline{Proof of~(\ref{n-2m+1}):}
now we assume that $p+q\leq n-2$ and $q\geq m+1$.
We first notice that the polynomial that appears in~(\ref{n-2m+1}) belongs to
$\mathcal{P}_{n,m}$ since all the involved products have total degree
at most $n-1$. This proves~(\ref{prexplicitlagrangepol1}).

Now we prove~(\ref{prexplicitlagrangepol2}).
First, one has $n-q-1>n-q-2\geq p$. 
On the other hand,
for all $r=1,\ldots,n-p-q-2$ (in case $n-p-q\geq3$, otherwise the sum does
not appear), one has 
$n-q-r-2\geq p$, then the expression in~(\ref{n-2m+1}) is divisible by
$\prod_{j=0}^{p-1}\dfrac{z-\eta_j}{\eta_p-\eta_j}$.
We also see that it is divisible by 
$\prod_{i=0}^{q-1}\dfrac{w-\theta_i}{\theta_q-\theta_i}$.
Thus it cancels all the points
$\left(\eta_k,\theta_l\right)$ with $0\leq k\leq p-1$ or $0\leq l\leq q-1$.

Next, if we fix $z=\eta_p$, the expression in~(\ref{n-2m+1}) gives
\begin{eqnarray*}
1\times\prod_{i=0}^{q-1}\dfrac{w-\theta_i}{\theta_q-\theta_i}
-
1\times\prod_{i=0}^{q-1}\dfrac{w-\theta_i}{\theta_q-\theta_i}
+
1\times\prod_{i=0,i\neq q}^{q+1}\dfrac{w-\theta_i}{\theta_q-\theta_i}
& &
\\
+\;
\sum_{r=1}^{n-p-q-2}
\left[
1\times\prod_{i=0,i\neq q}^{q+r+1}\dfrac{w-\theta_i}{\theta_q-\theta_i}
-
1\times\prod_{i=0,i\neq q}^{q+r}\dfrac{w-\theta_i}{\theta_q-\theta_i}
\right]
& = &
\;\;\;\;\;\;\;\;\;\;\;\;\;\;\;\;\;\;\;\;\;\;\;\;\;\;\;\;\;\;\;\;\;\;\;\;\;\;\;\;\;\;\;\;\;\;\;\;
\end{eqnarray*}
\begin{eqnarray*}
& = &
\prod_{i=0,i\neq q}^{q+1}\dfrac{w-\theta_i}{\theta_q-\theta_i}
+
\prod_{i=0,i\neq q}^{n-p-1}\dfrac{w-\theta_i}{\theta_q-\theta_i}
-
\prod_{i=0,i\neq q}^{q+1}\dfrac{w-\theta_i}{\theta_q-\theta_i}
\;=\;
\prod_{i=0,i\neq q}^{n-p-1}\dfrac{w-\theta_i}{\theta_q-\theta_i}
\,.
\end{eqnarray*}
We first get $1$ for $w=\theta_q$. If $w=\theta_l$ with $l\geq q+1$,
necessarily $p+l\leq n-1$ (otherwise, we should have $p+l=n$
then $l\leq m$, and this is impossible since 
$l\geq q+1>m+1>m$). Thus $l\leq n-p-1$ (with $l\geq q+1$)
and the above product vanishes. 

Now if we fix $w=\theta_q$ in~(\ref{n-2m+1}),
we get
\begin{eqnarray*}
\prod_{j=0,j\neq p}^{n-q-1}\dfrac{z-\eta_j}{\eta_p-\eta_j}\times1
-
\prod_{j=0,j\neq p}^{n-q-2}\dfrac{z-\eta_j}{\eta_p-\eta_j}\times1
+
\prod_{j=0,j\neq p}^{n-q-2}\dfrac{z-\eta_j}{\eta_p-\eta_j}\times1
& &
\\
+\;
\sum_{r=1}^{n-p-q-2}
\left[
\prod_{j=0,j\neq p}^{n-q-r-2}\dfrac{z-\eta_j}{\eta_p-\eta_j}\times1
-
\prod_{j=0,j\neq p}^{n-q-r-2}\dfrac{z-\eta_j}{\eta_p-\eta_j}\times1
\right]
& = &
\prod_{j=0,j\neq p}^{n-q-1}\dfrac{z-\eta_j}{\eta_p-\eta_j}
\,.
\end{eqnarray*}
Again, we get $1$ if $z=\eta_p$. If $z=\eta_k$ with
$k\geq p+1$ and $k+q\leq n$, necessarily
$k+q\leq n-1$ (otherwise, we should have $k+q=n$
then $q\leq m$, and this is impossible since
$q\geq m+1$). Thus 
$p+1\leq k\leq n-q-1$ and the above product vanishes.

The last case that we have to check is the one for which
$(z,w)=\left(\eta_k,\theta_l\right)\in\Omega_N$ with
$k\geq p+1$ and $l\geq q+1$.
Once again, we necessarily have that $k+l\leq n-1$ since
$k+l=n$ would yield $l\leq m$, and this would contradict that
$l\geq q+1>q\geq m+1$. Thus $k\leq n-l-1\leq n-q-2$ (with $k\geq p+1$),
then
$\prod_{j=0,j\neq p}^{n-q-2}\dfrac{\eta_k-\eta_j}{\eta_p-\eta_j}=0$
and the first three terms in~(\ref{n-2m+1}) vanish.
Next, let be $r$ with $1\leq r\leq n-p-q-2$ (once again,
we can assume that $n-p-q\geq3$).
If $k\leq n-q-r-2$, then
$\prod_{j=0,j\neq p}^{n-q-r-2}\dfrac{\eta_k-\eta_j}{\eta_p-\eta_j}=0$
(since $k\geq p+1$).
Otherwise, $k\geq n-q-r-1$ then
$l\leq n-k-1\leq q+r$ (with $l\geq q+1$) and
$\prod_{i=0,i\neq q}^{q+r}\dfrac{\theta_l-\theta_i}{\theta_q-\theta_i}=0$.
It follows that the sum in~(\ref{n-2m+1}) vanishes and this completes
the proof of~(\ref{n-2m+1}).

\medskip

\underline{Proof of~(\ref{n-1m-1}):}
we deal with the case $p+q=n-1$ and $q\leq m-1$.
First, all the involved products have total degree at most $n$, and
$\deg_w\left(\prod_{j=0,j\neq p}^{p+1}\dfrac{z-\eta_j}{\eta_p-\eta_j}
\prod_{i=0}^{q-1}\dfrac{w-\theta_i}{\theta_q-\theta_i}\right)=q\leq m-1$,
$\deg_w\left(\prod_{j=0,j\neq p}^{p-1}\dfrac{z-\eta_j}{\eta_p-\eta_j}
\prod_{i=0}^{q+1}\dfrac{w-\theta_i}{\theta_q-\theta_i}\right)=q+1\leq m$
(the remaining one having total degree at most $n-1$), and this 
proves~(\ref{prexplicitlagrangepol1}).

Now we prove~(\ref{prexplicitlagrangepol2}).
First, we see that all the involved products are divisible by
$\prod_{j=0}^{p-1}\dfrac{z-\eta_j}{\eta_p-\eta_j}
\prod_{i=0}^{q-1}\dfrac{w-\theta_i}{\theta_q-\theta_i}$,
then they cancel all the points
$\left(\eta_k,\theta_l\right)$ with $0\leq k\leq p-1$
or $0\leq l\leq q-1$.

Next, if we fix $z=\eta_p$, the expression~(\ref{n-1m-1}) gives
\begin{eqnarray*}
1\times
\prod_{i=0}^{q-1}\dfrac{w-\theta_i}{\theta_q-\theta_i}
\,-\,
1\times
\prod_{i=0}^{q-1}\dfrac{w-\theta_i}{\theta_q-\theta_i}
\,+\,
1\times
\prod_{i=0,i\neq q}^{q+1}\dfrac{w-\theta_i}{\theta_q-\theta_i}
& = &
\prod_{i=0,i\neq q}^{q+1}\dfrac{w-\theta_i}{\theta_q-\theta_i}
\,.
\end{eqnarray*}
We get $1$ for $w=\theta_q$. If $w=\theta_l$ with $l\geq q+1$,
then $l\leq n-p=q+1$, i.e. $l=q+1$ and the above product vanishes.

Now if we fix $w=\theta_q$ in~(\ref{n-1m-1}), we get
\begin{eqnarray*}
\prod_{j=0,j\neq p}^{p+1}\dfrac{z-\eta_j}{\eta_p-\eta_j}
\times1
\,-\,
\prod_{j=0}^{p-1}\dfrac{z-\eta_j}{\eta_p-\eta_j}
\times1
\,+\,
\prod_{j=0}^{p-1}\dfrac{z-\eta_j}{\eta_p-\eta_j}
\times1
& = &
\prod_{j=0,j\neq p}^{p+1}\dfrac{z-\eta_j}{\eta_p-\eta_j}
\,.
\end{eqnarray*}
We get $1$ for $z=\eta_p$. If $z=\eta_k$ with $k\geq p+1$,
then $k\leq n-q=p+1$, i.e. $k=p+1$ and the above product vanishes.

The remaining case is the one for which 
$(z,w)=\left(\eta_k,\theta_l\right)$ with
$k\geq p+1$ and $l\geq q+1$. But this yields
$k+l\geq p+q+2=n+1>n$, and this is impossible.

\medskip

\underline{Proof of~(\ref{n-1m}):}
we deal with the case $p+q=n-1$ and $q=m$ (then $p=n-m-1$).
The polynomial that appears in the expression~(\ref{n-1m})
belongs to $\mathcal{P}_{n,m}$ since \\
$\left(\prod_{j=0,j\neq {n-m-1}}^{n-m}\dfrac{z-\eta_j}{\eta_{n-m-1}-\eta_j}
\prod_{i=0}^{m-1}\dfrac{w-\theta_i}{\theta_m-\theta_i}\right)$ has
total degree $n$ and\\
$\deg_w\left(\prod_{j=0,j\neq {n-m-1}}^{n-m}\dfrac{z-\eta_j}{\eta_{n-m-1}-\eta_j}
\prod_{i=0}^{m-1}\dfrac{w-\theta_i}{\theta_m-\theta_i}\right)= m$.
This proves~(\ref{prexplicitlagrangepol1}).

Now we prove~(\ref{prexplicitlagrangepol2}).
First, we have that
$\prod_{j=0,j\neq {n-m-1}}^{n-m}\dfrac{z-\eta_j}{\eta_{n-m-1}-\eta_j}
\prod_{i=0}^{m-1}\dfrac{w-\theta_i}{\theta_m-\theta_i}$
is divisible by
$\prod_{j=0,j\neq {n-m-1}}^{n-m-2}\dfrac{z-\eta_j}{\eta_{n-m-1}-\eta_j}
\prod_{i=0}^{m-1}\dfrac{w-\theta_i}{\theta_m-\theta_i}$, then
it cancels all the points $\left(\eta_k,\theta_l\right)$
with $0\leq k\leq n-m-2$ or $0\leq l\leq m-1$.

Next, if we fix $z=\eta_{n-m-1}$, we get
$1\times\prod_{i=0}^{m-1}\dfrac{w-\theta_i}{\theta_m-\theta_i}$,
that gives $1$ if $w=\theta_m$. Otherwise, $w=\theta_l$ with $l\geq m+1$,
then
$n-m-1+l\geq n-m-1+m+1=n$. Since
$\left(\eta_{n-m-1},\theta_l\right)$ must belong to
$\Omega_N$, this yields $l\leq m$ and it is impossible.

Now if we fix $w=\theta_m$, we get
$\prod_{j=0,j\neq n-m-1}^{n-m}\dfrac{z-\eta_j}{\eta_{n-m-1}-\eta_j}\times1$,
that gives $1$ if $z=\eta_{n-m-1}$. Otherwise, $z=\eta_k$ with $k\geq n-m$.
Since $k+m\leq n$, it follows that $k\leq n-m$, i.e. $k=n-m$
and the product vanishes.

The remaining case is the one for which
$(z,w)=\left(\eta_k,\theta_l\right)$ with
$k\geq n-m-1+1=n-m$ and $l\geq m+1$, but this gives
$k+l\geq n+1$ and it is impossible.

\medskip

\underline{Proof of~(\ref{nn-1m+1}):}
we deal with the case $p+q=n$ (resp., $p+q=n-1$ and $q\geq m+1$).
In the first case, the polynomial
$\prod_{j=0}^{p-1}\dfrac{z-\eta_j}{\eta_p-\eta_j}
\prod_{i=0}^{q-1}\dfrac{w-\theta_i}{\theta_q-\theta_i}$ has total degree $p+q=n$
and 
$\deg_w\left(\prod_{j=0}^{p-1}\dfrac{z-\eta_j}{\eta_p-\eta_j}
\prod_{i=0}^{q-1}\dfrac{w-\theta_i}{\theta_q-\theta_i}\right)=q\leq m$
(since $\left(\eta_p,\theta_q\right)$ must belong to $\Omega_N$).
In the second one, it has total degree $p+q=n-1$. It follows that
$\prod_{j=0}^{p-1}\dfrac{z-\eta_j}{\eta_p-\eta_j}
\prod_{i=0}^{q-1}\dfrac{w-\theta_i}{\theta_q-\theta_i}$
belongs to $\mathcal{P}_{n-1}\subset\mathcal{P}_{n,m}$, and this proves~(\ref{prexplicitlagrangepol1}).

On the other hand, we see that, in any case, the polynomial
$\prod_{j=0}^{p-1}\dfrac{z-\eta_j}{\eta_p-\eta_j}
\prod_{i=0}^{q-1}\dfrac{w-\theta_i}{\theta_q-\theta_i}$
cancels all the points $\left(\eta_k,\theta_l\right)$
with $0\leq k\leq p-1$ or $0\leq l\leq q-1$. We can then assume in the following
that $k\geq p$ and $l\geq q$ in order to prove~(\ref{prexplicitlagrangepol2}).

In the first case $p+q=n$, the conditions $k\geq p$ and $l\geq q$ yield
$k+l\geq p+q=n$, then $k+l=n$ (since
$\left(\eta_k,\theta_l\right)$ must belong to $\Omega_N$), i.e. $k=p$ and $l=q$, for which
the polynomial gives $1$. This completes the proof of the required equality~(\ref{nn-1m+1})
in that case.

In the second case $p+q=n-1$ and $q\geq m+1$, the conditions $k\geq p$ and $l\geq q$ yield
$k+l\geq p+q=n-1$, i.e.
$k+l=n-1$ or $k+l=n$ (since $k+l\leq n$).
If $k+l=n-1$, i.e. $k+l=p+q$, necessarily $k=p$ and $l=q$, for wich
the polynomial gives $1$. Otherwise, 
$k+l=n$ then necessarily $l\leq m\leq q-1$ (by assumption), and this is impossible
since $l\geq q$. The proof of~(\ref{nn-1m+1}) (and of the lemma) is complete.

\end{proof}

\begin{remark}

These formulas could also have been deduced by using the results from~\cite{sauerxu1995} and~\cite{calvi2005}
where the authors give algorithms to construct them in the general case of
{\em block unisolvent arrays} in $\C^d$. Here we independently computed
them for the special case of $\Omega_N$ and $\mathcal{P}_{n,m}$ in $\C^2$.

\end{remark}

A consequence of the previous lemma is the following result.

\begin{corollary}\label{unisolvent}

For all $N\geq1$, the set $\Omega_N$ is unisolvent for the space $\mathcal{P}_{n,m}$:
for every function $f$ that is defined on $\Omega_N$, there exists a unique
polynomial $P\in\mathcal{P}_{n,m}$ such that
$P(z,w)=f(z,w)$ for every $(z,w)\in\Omega_N$.
In addition, the polynomial $P$ is the multivariate Lagrange polynomial
\begin{eqnarray}\label{deflagpoldbis}
L^{(N)}[f]
& = &
\sum_{k=1}^N
f\left(H_k\right)\,l_{H_k}^{(N)}
\,,
\end{eqnarray}
where the $l_{H_k}^{(N)}$'s are the functions mentioned in the statement of
Proposition~\ref{explicitlagrangepol}.

\end{corollary}

In particular, it will follow that these functions $l_{H_k}^{(N)}$'s are indeed the 
required FLIPs (for $\mathcal{P}_{n,m}$ and $\Omega_N$) and this will complete
the proof of Proposition~\ref{explicitlagrangepol}.

\begin{proof}

As a recapitulation, 
Lemma~\ref{prexplicitlagrangepol} gives the existence of a family 
$\left\{l_{H_k}^{(N)}\right\}_{1\leq k \leq N}\subset\mathcal{P}_{n,m}$ that satisfies
$l_{H_k}^{(N)}\left(H_l\right)=\delta_{k,l}$ for all $k,\,l=1,\ldots,N$.
This family is necessarily linearly independent, and
since its cardinal is
$N=\dim\mathcal{P}_{n,m}$, it is also a basis.

Let be $P\in\mathcal{P}_{n,m}$. There are $\alpha_1,\ldots,\alpha_N\in\C$ such that
$P=\sum_{k=1}^N\alpha_kl_{H_k}^{(N)}$. If $P$ vanishes on $\Omega_N$, it follows that
for all $j=1,\ldots,N$,
\begin{eqnarray*}
0
\;=\;
P\left(H_j\right)
\;=\;
\sum_{k=1}^N\alpha_kl_{H_k}^{(N)}\left(H_j\right)
\;=\;
\sum_{k=1}^N\alpha_k\delta_{k,j}
\;=\;
\alpha_j
\,,
\end{eqnarray*}
thus $P\equiv0$. 

This proves that $\Omega_N$ is unisolvent: indeed, 
$f$ being given, the function $L^{(N)}[f]$ defined in~(\ref{deflagpoldbis}) satisfies the required properties.
If $Q$ is another polynomial that fulfills the same conditions,
$L^{(N)}[f]-Q$ belongs to $\mathcal{P}_{n,m}$ and vanishes on $\Omega_N$,
then $L^{(N)}[f]-Q\equiv0$.

\end{proof}

\subsection{An estimation of the Lebesgue constant for some intertwining sequences and applications}

Now we assume that the sequences $\left(\eta_j\right)$ and $\left(\theta_i\right)_{i\geq0}$
are both Leja sequences for the unit disk
(with $\left|\eta_0\right|=\left|\theta_0\right|=1$). We can then
give the proof of Proposition~\ref{unifestimintertwining}.

\begin{proof}

We begin with the proof of the first estimate. Since
$\left(\eta_j\right)_{j\geq0}$ and $\left(\theta_i\right)_{i\geq0}$
are Leja sequences for the unit disk (with
$\left|\eta_0\right|=\left|\theta_0\right|=1$), we get by an application of Theorem~\ref{unifestim} that
for all $k,l\geq0$ and $p,q$ with $0\leq p\leq k$ and $0\leq q\leq l$,
\begin{eqnarray*}
\sup_{z\in\overline{\D}}
\left|
\prod_{j=0,j\neq p}^k
\dfrac{z-\eta_j}{\eta_p-\eta_j}
\right|
\;\leq\;
\pi\exp(3\pi)
%
& \mbox{ and } &
\sup_{w\in\overline{\D}}
\left|
\prod_{i=0,i\neq p}^l
\dfrac{w-\theta_i}{\theta_q-\theta_i}
\right|
\;\leq\;
\pi\exp(3\pi)
\,.
\end{eqnarray*}

Let be $N\geq1$ and the associated numbers $n\geq0$ and $m=0,\ldots,n$.
For all $\left(\eta_p,\theta_q\right)\in\Omega_{N}$,
an application of Proposition~\ref{explicitlagrangepol} yields the following estimates:

\begin{itemize}

\item

if $p+q=n$, or $p+q=n-1$ and $q\geq m+1$,
one has by~(\ref{nn-1m+1}),
\begin{eqnarray*}
\sup_{(z,w)\in\overline{\D}^2}
\left|
l_{(\eta_p,\theta_q)}^{(N)}(z,w)
\right|
& \leq &
\pi\exp(3\pi)\times\pi\exp(3\pi)
\;=\;
\pi^2\exp(6\pi)
\,;
\end{eqnarray*}

\item

if $p+q=n-1$ and $q=m$, one has by~(\ref{n-1m}),
\begin{eqnarray*}
\sup_{(z,w)\in\overline{\D}^2}
\left|
l_{(\eta_p,\theta_q)}^{(N)}(z,w)
\right|
& \leq &
\pi^2\exp(6\pi)
\,;
\end{eqnarray*}

\item

if $p+q=n-1$ and $0\leq q\leq m-1$, one has by~(\ref{n-1m-1}),
\begin{eqnarray*}
\sup_{(z,w)\in\overline{\D}^2}
\left|
l_{(\eta_p,\theta_q)}^{(N)}(z,w)
\right|
& \leq &
3\pi^2\exp(6\pi)
\,;
\end{eqnarray*}

\item

if $0 \leq p+q\leq n-2$, $0\leq q\leq m-1$ and $0\leq p\leq n-m-1$, one has by~(\ref{n-2m-1n-m-1}),
\begin{eqnarray*}
\sup_{(z,w)\in\overline{\D}^2}
\left|
l_{(\eta_p,\theta_q)}^{(N)}(z,w)
\right|
& \leq &
[3+2(m-q-1+n-p-q-2-(m-q)+1)]\pi^2\exp(6\pi)
\\
& = &
\left[2(n-p-q)-1\right]\pi^2\exp(6\pi)
\,;
\end{eqnarray*}

\item

if $0 \leq p+q\leq n-2$, $0\leq q\leq m-1$ and $p\geq n-m$, one has by~(\ref{n-2m-1n-m}),
\begin{eqnarray*}
\sup_{(z,w)\in\overline{\D}^2}
\left|
l_{(\eta_p,\theta_q)}^{(N)}(z,w)
\right|
\,=\,
[3+2(n-p-q-1)]\pi^2\exp(6\pi)
\,=\,
[2(n-p-q)+1]\pi^2\exp(6\pi)
\,;
\end{eqnarray*}

\item

if $p+q\leq n-2$ and $q=m$, one has by~(\ref{n-2m}),
\begin{eqnarray*}
\sup_{(z,w)\in\overline{\D}^2}
\left|
l_{(\eta_p,\theta_q)}^{(N)}(z,w)
\right|
\,\leq\,
\left[3+2(n-m-p-2)\right]
\pi^2\exp(6\pi)
\,=\,
\left[2(n-p-q)-1\right]
\pi^2\exp(6\pi)
\,;
\end{eqnarray*}

\item

if $p+q\leq n-2$ and $q\geq m+1$, one has by~(\ref{n-2m+1}),
\begin{eqnarray*}
\sup_{(z,w)\in\overline{\D}^2}
\left|
l_{(\eta_p,\theta_q)}^{(N)}(z,w)
\right|
\,=\,
\left[3+2(n-p-q-2)\right]
\pi^2\exp(6\pi)
\,=\,
\left[2(n-p-q)-1\right]
\pi^2\exp(6\pi)
\,.
\end{eqnarray*}

\end{itemize}

In any case, the following estimate is valid for all $\left(\eta_p,\theta_q\right)\in\Omega_{N}$:
\begin{eqnarray*}
\sup_{(z,w)\in\overline{\D}^2}
\left|
l_{(\eta_p,\theta_q)}^{(N)}(z,w)
\right|
& \leq &
2\left(n-p-q+1\right)
\pi^2\exp(6\pi)
\,.
\end{eqnarray*}
This proves the first assertion of the corollary since by~(\ref{Nn}), $N\sim N_n\sim n^2/2$.

\medskip

The proof of the second part is similar since Proposition~\ref{explicitlagrangepol} is still valid
with the points $\eta_j$'s and  $\theta_i$'s replaced with the $\Phi_1\left(\eta_j\right)$'s
and $\Phi_2\left(\theta_i\right)$'s respectively (notice that the points
$\Phi_1\left(\eta_j\right)$'s and $\Phi_2\left(\theta_i\right)$'s are well-defined
since all the $\eta_j$'s and  $\theta_i$'s belong to the unit circle).
The only difference is an application of Theorem~\ref{unifestimcompact}
instead of Theorem~\ref{unifestim}, that yields
for all $k,l\geq0$ and $p,q$ with $0\leq p\leq k$ and $0\leq q\leq l$,
\begin{eqnarray}\label{pseudomejor1}
\sup_{z\in K_1}
\left|
\prod_{j=0,j\neq p}^k
\dfrac{z-\Phi_1\left(\eta_j\right)}{\Phi_1\left(\eta_p\right)-\Phi_1\left(\eta_j\right)}
\right|
& \leq &
M_1(k+2)^{2A_1/\ln(2)}
\end{eqnarray}
and
\begin{eqnarray}\label{pseudomejor2}
\sup_{w\in K_2}
\left|
\prod_{i=0,i\neq p}^l
\dfrac{w-\Phi_2\left(\theta_i\right)}{\Phi_2\left(\theta_q\right)-\Phi_2\left(\theta_i\right)}
\right|
& \leq &
M_2(l+2)^{2A_2/\ln(2)}
\,,
\end{eqnarray}
where $M_1$, $A_1$ and $M_2$, $A_2$ are positive constants depending only
on $K_1$ and $K_2$ respectively.
By repeating the same argument of the first part, we get for all
$\left(\eta_p,\theta_q\right)\in\Omega_{N}$:
\begin{eqnarray*}
\sup_{(z,w)\in K_1\times K_2}
\left|
l_{\left(\Phi_1\left(\eta_p\right),\Phi_2\left(\theta_q\right)\right)}^{(N)}(z,w)
\right|
\,\leq\,
2M_1M_2\left(n-p-q+1\right)
\max_{k+l=n}
(k+2)^{2A_1/\ln(2)}
(l+2)^{2A_2/\ln(2)}
\,.
\end{eqnarray*}
By setting $A=\max\left(A_1,A_2\right)$ and noticing that
\begin{eqnarray*}
\max_{k+l=n}(k+2)(l+2)
& = &
(n+2)^2
\max_{k+l=n}\dfrac{k+2}{n+2}\left(1-\dfrac{k}{n+2}\right)
\\
& \leq &
(n+2)^2
\sup_{x\in[0,1]}[x(1-x)]
\;=\;
\dfrac{1}{4}(n+2)^2
\,,
\end{eqnarray*}
we get
\begin{eqnarray*}
\sup_{(z,w)\in K_1\times K_2}
\left|
l_{\left(\Phi_1\left(\eta_p\right),\Phi_2\left(\theta_q\right)\right)}^{(N)}(z,w)
\right|
& \leq &
\dfrac{2M_1M_2}{2^{4A/\ln(2)}}
(n-p-q+1)(n+2)^{4A/\ln(2)}
\,.
\end{eqnarray*}
This proves the second part of the corollary since by~(\ref{Nn}), $N\sim N_n\sim n^2/2$.

\end{proof}

\begin{remark}

The constant $c_n$ that appears in Lemma~\ref{lemma3bialascalvi}
can be improved. We know by Theorem~6 from~\cite{bialascalvi} that
\begin{eqnarray*}
c_n
\;=\;
\exp\left(A\sum_{j=0}^s\epsilon_j\right)
\,,
& \mbox{ where } &
n
\;=\;
\sum_{j=0}^s\epsilon_j2^j
\,.
\end{eqnarray*}
At the end of the proof of Proposition~\ref{unifestimintertwining}, 
by using~(\ref{unifestimcompactaux}), we could improve~(\ref{pseudomejor1})
and~(\ref{pseudomejor2}) to get
\begin{eqnarray*}
\sup_{z\in K_1}
\left|
\prod_{j=0,j\neq p}^k
\dfrac{z-\Phi_1\left(\eta_j\right)}{\Phi_1\left(\eta_p\right)-\Phi_1\left(\eta_j\right)}
\right|
& \leq &
M_1c_k^2
\;=\;
M_1\exp
\left(2A_1\sum_{j=0}^{s_k}\epsilon_j^{(k)}\right)
\end{eqnarray*}
and
\begin{eqnarray*}
\sup_{w\in K_2}
\left|
\prod_{i=0,i\neq p}^l
\dfrac{w-\Phi_2\left(\theta_i\right)}{\Phi_2\left(\theta_q\right)-\Phi_2\left(\theta_i\right)}
\right|
& \leq &
M_2c_l^2
\;=\;
M_2\exp
\left(2A_2\sum_{j=0}^{s_l}\epsilon_j^{(l)}\right)
\,,
\end{eqnarray*}
respectively. We could get (with $A=\max\left(A_1,A_2\right)$)
\begin{eqnarray*}
\sup_{(z,w)\in K_1\times K_2}
\left|
l_{\left(\Phi_1\left(\eta_p\right),\Phi_2\left(\theta_q\right)\right)}^{(N)}(z,w)
\right|
\,\leq\,
2M_1M_2\left(n-p-q+1\right)
\max_{k+l=n}
\exp2A\left(\sum_{j=0}^{s_k}\epsilon_j^{(k)}+\sum_{j=0}^{s_l}\epsilon_j^{(l)}\right).
\end{eqnarray*}
 Nevertheless, we cannot hope any improvement better than
$\sum_{j=0}^{s_k}\epsilon_j^{(k)}+\sum_{j=0}^{s_l}\epsilon_j^{(l)}
\leq2\sum_{j=0}^s\epsilon_j$ (that yields
$(n+2)^{4A/\ln(2)}$ at the end of the proof), even though
$\sum_{j=0}^{s_k}\epsilon_j^{(k)}2^j+\sum_{j=0}^{s_l}\epsilon_j^{(l)}2^j
=k+l=n=\sum_{j=0}^s\epsilon_j2^j$.
For example, if $n=2^{r+1}-2=\sum_{j=1}^r2^j$ and $k=l=2^r-1=\sum_{j=0}^{r-1}2^j$, then
$\sum_{j=0}^{s_k}\epsilon_j^{(k)}+\sum_{j=0}^{s_l}\epsilon_j^{(l)}
=2\sum_{j=0}^{r-1}1=2r=2\sum_{j=1}^{r}1=2\sum_{j=0}^{s}\epsilon_j$.

\end{remark}

\medskip

Now we can give the proof of Theorem~\ref{lebcteintertwin}.

\begin{proof}

We begin with the proof of the part~(\ref{lebcteintertwinbidisc}).
By the first part of Proposition~\ref{unifestimintertwining}, we immediately get 
\begin{eqnarray*}
\Lambda_N\left(\overline{\D}^2,\left(H_j\right)_{k\geq1}\right)
& \leq &
\sum_{(\eta_p,\theta_q)\in\Omega_{N}}
\sup_{(z,w)\in\overline{\D}^2}
\left|
l_{(\eta_p,\theta_q)}^{(N)}(z,w)
\right|
\,\leq\,
\sum_{p+q\leq n}
2\left(n-p-q+1\right)
\pi^2\exp(6\pi)
\\
& \leq &
2n\pi^2\exp(6\pi)\dfrac{(n+1)(n+2)}{2}
\;=\;
\pi^2\exp(6\pi)\,n(n+1)(n+2)
\,.
\end{eqnarray*}
It follows that
\begin{eqnarray*}
\Lambda_N\left(\overline{\D}^2,\left(H_j\right)_{k\geq1}\right)
& = &
O\left(n^3\right)
\;=\;
O\left(N^{3/2}\right)
\,,
\end{eqnarray*}
since by~(\ref{Nn}), $N\sim N_n\sim n^2/2$. This proves~(\ref{lebcteintertwinbidisc}).

The proof of~(\ref{lebcteintertwincompact}) is similar. By applying the second part
of Proposition~\ref{unifestimintertwining}, we have that
\begin{eqnarray*}
\Lambda_N\left(K_1\times K_2,\left(\Phi\left(H_j\right)\right)_ {j\geq1}\right)
& \leq &
\sum_{(\eta_p,\theta_q)\in\Omega_{N}}
\sup_{(z,w)\in K_1\times K_2}
\left|
l_{\left(\Phi_1\left(\eta_p\right),\Phi_2\left(\theta_q\right)\right)}^{(N)}(z,w)
\right|
\\
& \leq &
M(n+2)^{4A/\ln(2)}
\sum_{p+q\leq n}
\left(n-p-q+1\right)
\\
& \leq &
\dfrac{M}{2}(n+2)^{4A/\ln(2)}n(n+1)(n+2)
\,.
\end{eqnarray*}
Once again, it follows by~(\ref{Nn}) that
\begin{eqnarray*}
\Lambda_N\left(K_1\times K_2,\left(\Phi\left(H_j\right)\right)_ {j\geq1}\right)
\;=\;
O\left(n^{4A/\ln(2)+3}\right)
\;=\;
O\left(N^{2A/\ln(2)+3/2}\right)
\,.
\end{eqnarray*}

\end {proof}

\begin{remark}

At the end of the above proof, we could get a sharper estimate by computing
the sum $\sum_{p+q\leq n}\left(n-p-q+1\right)$. But it seems useless since 
it will not change the exponent of $n$ (or $N$) that may not be optimal.

\end{remark}

Now we deal with another application of Theorem~\ref{lebcteintertwin}
that is Corollary~\ref{appljackson}. We first remind the notations from~\cite{renwang}:
the Lipschitz space $Lip_{\gamma}\left(\overline{\D}^d\right)$, $0<\gamma\leq1$,
consists of all holomorphic functions 
$f\in\mathcal{O}\left(\D^d\right)\bigcap C\left(\overline{\D}^d\right)$ satisfying
\begin{eqnarray*}
\left|f\left(e^{ih}\zeta\right)-f(\zeta)\right|
& \leq &
L\,|h|^{\gamma}
\,,
\end{eqnarray*}
for any $\zeta\in\mathbb{T}^d$ (the unit torus in $\C^d$) and $h\in\R$
($L>0$ being the Lipschitz constant);
then $m\in\N$ being given, the function 
$f\in\mathcal{O}\left(\D^d\right)\bigcap C\left(\overline{\D}^d\right)$ is said to belong to the Lipschitz space
$Lip_{\gamma}^m\left(\overline{\D}^d\right)$
if $\dfrac{\partial^{\alpha}f}{\partial z^{\alpha}}\in Lip_{\gamma}\left(\overline{\D}^d\right)$
for any $|\alpha|=\alpha_1+\cdots+\alpha_d\leq m$;
lastly, for all $n\geq0$, $\mathcal{P}_n^{(d)}$ is the space
of polynomials of total degree at most $n$ (in particular,
$\mathcal{P}_n^{(2)}=\mathcal{P}_n=\mathcal{P}_{n,n}$
from~(\ref{defPnm})).

Next, we can remind the following result
given as Theorem~7.2 from~\cite{renwang} and that is a generalization of
Jackson's theorem in the polydisc.

\begin{theorem}\label{thm72renwang}

If $f\in Lip_{\gamma}^m\left(\overline{\D}^d\right)$ then for all $n\in\N$,
\begin{eqnarray*}
\inf_{P\in\mathcal{P}^{(d)}_n}
\sup_{z\in\overline{\D}^d}
\left|
f(z)-P(z)
\right|
& \leq &
L\dfrac{C_{m,\gamma}}{n^{m+\gamma}}
\,.
\end{eqnarray*}

\end{theorem}

We can then give the proof of Corollary~\ref{appljackson}.

\begin{proof}

Let fix $m\geq3$, $0<\gamma\leq1$ and 
$f\in Lip_{\gamma}^m\left(\overline{\D}^2\right)$.
By Theorem~\ref{thm72renwang}, there is a positive constant
$M>0$ such that for all $n\geq1$, there exists $P_n\in\mathcal{P}_n$ that satisfies
\begin{eqnarray}\label{thm72renwang2}
\sup_{(z,w)\in\overline{\D}^2}
\left|
f(z,w)-P_n(z,w)
\right|
& \leq &
\dfrac{M}{n^{m+\gamma}}
\,.
\end{eqnarray}

On the other hand, let $\left(\eta_j\right)_{j\geq0}$ and $\left(\theta_i\right)_{i\geq0}$
be Leja sequences for the unit disk (with $\left|\eta_0\right|=\left|\theta_0\right|=1$),
and let us consider the intertwining sequence $\left(H_j\right)_{k\geq1}$ defined
as in~(\ref{defintertwin}). For all
$N\geq1$ with 
$N_n<N\leq N_{n+1}$ (see~(\ref{Nn})),
let us consider $\Omega_N$ defined as 
in~(\ref{defOmegaN}).
One has in particular that
\begin{eqnarray}\label{equivNn}
N
& \sim &
\dfrac{n^2}{2}
\,.
\end{eqnarray}

Now let us consider $L^{(N)}[f]$ the (multivariate) Lagrange polynomial of $f$
defined by~(\ref{deflagpold}). Its existence is guaranteed
by Proposition~\ref{explicitlagrangepol} and the fact that
$H_p\in\overline{\D}^2$ for all $p=1,\ldots,N$.
A classical calculation yields for all $(z,w)\in\overline{\D}^2$,
\begin{eqnarray}\nonumber
\left|
f(z,w)-L^{(N)}[f](z,w)
\right|
& \leq &
\left|
f(z,w)-P_n(z,w)
\right|
+
\left|
P_n(z,w)-L^{(N)}[f](z,w)
\right|
\\\label{appljacksonaux1}
& \leq &
\dfrac{M}{n^{m+\gamma}}
+
\left|
P_n(z,w)-L^{(N)}[f](z,w)
\right|
\,,
\end{eqnarray}
the second inequality being an application of~(\ref{thm72renwang2}).

On the other hand, we claim that
$L^{(N)}\left[P_n\right]\equiv P_n$. Indeed, 
$L^{(N)}\left[P_n\right]\in\mathcal{P}_{n+1,m}$
(by~(\ref{defnm}) and~(\ref{defPnm})) ,
as well as $P_n\in\mathcal{P}_n=\mathcal{P}_{n,n}\subset\mathcal{P}_{n+1,m}$.
Moreover, $L^{(N)}\left[P_n\right]$ and $P_n$ coincide on
$\Omega_N=\Omega_{n+1,m}$ that is unisolvent for $\mathcal{P}_{n+1,m}$ by Corollary~\ref{unisolvent}.
This proves the claim.

It follows that
\begin{eqnarray}\nonumber
\left|
P_n(z,w)-L^{(N)}[f](z,w)
\right|
& = &
\left|
L^{(N)}\left[P_n\right](z,w)-L^{(N)}[f](z,w)
\right|
\;=\;
\left|
L^{(N)}\left[P_n-f\right](z,w)
\right|
\\\nonumber
& \leq &
\Lambda_N\left(\overline{\D}^2,\left(H_j\right)_{k\geq1}\right)
\sup_{(z',w')\in\overline{\D}^2}
\left|
P_n\left(z',w'\right)-f\left(z',w'\right)
\right|
\\\label{appljacksonaux2}
& \leq &
\Lambda_N\left(\overline{\D}^2,\left(H_j\right)_{k\geq1}\right)
\times
\dfrac{M}{n^{m+\gamma}}
\,,
\end{eqnarray}
the first inequality being valid since the Lebesgue constant is also the operator norm
of the linear operator $L^{(N)}$ that is
the projection from $C\left(\overline{\D}^2\right)$ onto $\mathcal{P}_{n+1,m}$, 
and the second one being an application of~(\ref{thm72renwang2}).

We can deduce by~(\ref{appljacksonaux1}) and~(\ref{appljacksonaux2}) that
\begin{eqnarray}\nonumber
\left|
f(z,w)-L^{(N)}[f](z,w)
\right|
& \leq &
\left(
1+
\Lambda_N\left(\overline{\D}^2,\left(H_j\right)_{k\geq1}\right)
\right)
\dfrac{M}{n^{m+\gamma}}
\,.
\end{eqnarray}
Finally, an application of~(\ref{lebcteintertwinbidisc}) from Theorem~\ref{lebcteintertwin}
with~(\ref{equivNn}) leads to
\begin{eqnarray*}
\left|
f(z,w)-L^{(N)}[f](z,w)
\right|
\;=\;
\left(1+
O\left(N^{3/2}\right)
\right)
\times
O\left(\dfrac{1}{\left(\sqrt{N}\right)^{m+\gamma}}\right)
\;=\;
O\left(\dfrac{1}{N^{(m+\gamma-3)/2}}\right)
\,,
\end{eqnarray*}
and the corollary is proved.

\end{proof}

\medskip

\subsection{On the construction of explicit bidimensional Leja sequences}

Now we deal with the proof of Theorem~\ref{explicitlejasequ}. We remind
the intertwining sequence $\left(H_k\right)_{k\geq1}$ constructed from
any sequences $\left(\eta_j\right)_{j\geq0}$ and $\left(\theta_i\right)_{i\geq0}$
of pairwise distinct elements, and we prove the following result that gives a formula
for the Vandermonde determinant that can also be linked to Proposition~2.5 from~\cite{calvi2005}.

\begin{lemma}\label{explicitVDM}

We have for all $n\geq0$ and $(z,w)\in\C^2$:
\begin{eqnarray*}
VDM
\left(
H_1,H_2,\ldots,H_{N_n},(z,w)
\right)
& = &
\prod_{j=0}^n
\left(z-\eta_j\right)
\times
VDM
\left(
H_1,H_2,\ldots,H_{N_n}
\right)
\,.
\end{eqnarray*}
Similarly, $n\geq0$ being given, we have for all $N$ with
$N_{n-1}<N<N_n$ (and the associated $m$ with $0\leq m\leq n-1$), and all $(z,w)\in\C^2$:
\begin{eqnarray*}
VDM
\left(
H_1,H_2,\ldots,H_{N},(z,w)
\right)
\,=\,
\prod_{j=0}^{n-m-2}\left(z-\eta_j\right)
\,
\prod_{i=0}^{m}\left(w-\theta_i\right)
\times
VDM
\left(
H_1,H_2,\ldots,H_{N}
\right)
\end{eqnarray*}
(with the convention that $\prod_{\emptyset}=1$ in case $m=n-1$).

\end{lemma}

\begin{proof}

We have for all $N\geq1$ and all $(z,w)\in\C^2$ (the $e_p$'s being defined by~(\ref{deforderpol})),
\begin{eqnarray*}
VDM
\left(
H_1,H_2,\ldots,H_{N},(z,w)
\right)
\,=\,
\left|
\begin{array}{ccccc}
e_1\left(H_1\right) & e_1\left(H_2\right) & \cdots & e_1\left(H_N\right) & e_{1}(z,w)
\\
e_2\left(H_1\right) & e_2\left(H_2\right) & \cdots & e_2\left(H_N\right) & e_{2}(z,w)
\\
\vdots & \vdots & \ddots & \vdots & \vdots
\\
e_N\left(H_1\right) & e_N\left(H_2\right) & \cdots & e_N\left(H_N\right) & e_{N}(z,w)
\\
e_{N+1}\left(H_1\right) & e_{N+1}\left(H_2\right) & \cdots & e_{N+1}\left(H_N\right) & e_{{N+1}}(z,w)
\end{array}
\right|
.
\end{eqnarray*}

Let us consider the Lagrange polynomial of $e_{N+1}$,
\begin{eqnarray*}
L^{(N)}\left[e_{N+1}\right](z,w)
& = &
\sum_{p=1}^Ne_{N+1}\left(H_p\right)l_{H_p}^{(N)}(z,w)
\,,
\;\forall\,(z,w)\in\C^2
\,.
\end{eqnarray*}
Its existence is guaranteed by Proposition~\ref{explicitlagrangepol}.
Since $L^{(N)}\left[e_{N+1}\right]$ is spanned by the
$e_p$'s, $p=1,\ldots,N$, one also has
\begin{eqnarray}\label{explicitVDMaux1}
L^{(N)}\left[e_{N+1}\right](z,w)
& = &
\sum_{p=1}^N\alpha_pe_p(z,w)
\,,
\;\alpha_p\in\C,\,\;\forall\,p=1,\ldots,N.
\end{eqnarray}

The above determinant is not changed if we replace the last row $L_{N+1}$
with $L_{N+1}-\sum_{p=1}^N\alpha_pL_p$. It follows by~(\ref{explicitVDMaux1})
that for
all $j=1,\ldots,N$, $e_{N+1}\left(H_j\right)$ becomes
\begin{eqnarray*}
e_{N+1}\left(H_j\right)-\sum_{p=1}^N\alpha_pe_p\left(H_j\right)
\;=\;
e_{N+1}\left(H_j\right)-L^{(N)}\left[e_{N+1}\right]\left(H_j\right)
\;=\;
0
\,,
\end{eqnarray*}
by the fundamental property of $L^{(N)}\left[e_{N+1}\right]$.
On the other hand, the element $e_{N+1}(z,w)$ becomes
$e_{N+1}(z,w)-\sum_{p=1}^N\alpha_pe_p(z,w)=e_{N+1}(z,w)-L^{(N)}\left[e_{N+1}\right](z,w)$.
Thus
\begin{eqnarray*}
VDM
\left(
H_1,H_2,\ldots,H_{N},(z,w)
\right)
& = &
\;\;\;\;\;\;\;\;\;\;\;\;\;\;\;\;\;\;\;\;\;\;\;\;\;\;\;\;\;\;\;\;\;\;\;\;\;\;\;\;\;\;\;\;
\;\;\;\;\;\;\;\;\;\;\;\;\;\;\;\;\;\;\;\;\;\;\;\;\;\;\;\;
\end{eqnarray*}
\begin{eqnarray*}
\;\;\;\;\;\;\;\;
& = &
\left|
\begin{array}{ccccc}
e_1\left(H_1\right) & e_1\left(H_2\right) & \cdots & e_1\left(H_N\right) & e_{1}(z,w)
\\
e_2\left(H_1\right) & e_2\left(H_2\right) & \cdots & e_2\left(H_N\right) & e_{2}(z,w)
\\
\vdots & \vdots & \ddots & \vdots & \vdots
\\
e_N\left(H_1\right) & e_N\left(H_2\right) & \cdots & e_N\left(H_N\right) & e_{N}(z,w)
\\
0 & 0 & \cdots & 0 & e_{{N+1}}(z,w)-L^{(N)}\left[e_{N+1}\right](z,w)
\end{array}
\right|
\\
& = &
\left(
e_{{N+1}}(z,w)-L^{(N)}\left[e_{N+1}\right](z,w)
\right)
\times
\left|
\begin{array}{cccc}
e_1\left(H_1\right) & e_1\left(H_2\right) & \cdots & e_1\left(H_N\right)
\\
e_2\left(H_1\right) & e_2\left(H_2\right) & \cdots & e_2\left(H_N\right)
\\
\vdots & \vdots & \ddots & \vdots 
\\
e_N\left(H_1\right) & e_N\left(H_2\right) & \cdots & e_N\left(H_N\right) 
\end{array}
\right|
\\
& = &
\left(
e_{{N+1}}(z,w)-L^{(N)}\left[e_{N+1}\right](z,w)
\right)
\times
VDM
\left(
H_1,H_2,\ldots,H_{N}
\right)
\,.
\end{eqnarray*}

The proof of the lemma will be complete once we have proved that for
all $(z,w)\in\C^2$,
\begin{eqnarray}\nonumber
e_{{N_n+1}}(z,w)-L^{(N_n)}\left[e_{N_n+1}\right](z,w)
& = &
z^{n+1}-L^{(N_n)}\left[{X}^{n+1}\right](z,w)
\\\label{explicitVDMaux3}
& = &
\prod_{j=0}^{n}
\left(z-\eta_j\right)
\,,
\end{eqnarray}
and for all $m=0,\ldots,n-1$,
\begin{eqnarray}\nonumber
e_{N+1}(z,w)-L^{(N)}\left[e_{N+1}\right](z,w)
& = &
z^{n-m-1}w^{m+1}-L^{(N)}\left[{X}^{n-m-1}{Y}^{m+1}\right](z,w)
\\\label{explicitVDMaux4}
& = &
\prod_{j=0}^{n-m-2}
\left(z-\eta_j\right)
\;
\prod_{i=0}^{m}
\left(w-\theta_i\right)
\,.
\end{eqnarray}

We first deal with~(\ref{explicitVDMaux3}) and set
\begin{eqnarray*}
R_{N_n}(z,w)
& := &
z^{n+1}-\prod_{j=0}^{n}
\left(z-\eta_j\right)
\,.
\end{eqnarray*}
Since it is a difference of unitary polynomials of degree $n+1$, $R_{N_n}$
is spanned by the $z^j$'s for $j=0,\ldots,n$. Then
$R_{N_n}\in\mathcal{P}_{n,n}$, as well as $L^{(N_n)}\left[{X}^{n+1}\right]$.
On the other hand, for all
$\left(\eta_k,\theta_l\right)\in\Omega_{N_n}$, one has $0\leq k\leq k+l\leq n$, then
\begin{eqnarray*}
R_{N_n}\left(\eta_k,\theta_l\right)
& = &
\eta_k^{n+1}-\prod_{j=0}^n\left(\eta_k-\eta_j\right)
\;=\;
\eta_k^{n+1}
\;=\;
L^{(N_n)}\left[{X}^{n+1}\right]\left(\eta_k,\theta_l\right)
\,,
\end{eqnarray*}
i.e. $R_{N_n}$ and $L^{(N_n)}\left[{X}^{n+1}\right]$ coincide on the set $\Omega_{N_n}$.
Since $\Omega_{N_n}$ is unisolvent for the space $\mathcal{P}_{n,n}$ by
Corollary~\ref{unisolvent}, it follows that
$R_{N_n}\equiv L^{(N_n)}\left[{X}^{n+1}\right]$ and this 
proves~(\ref{explicitVDMaux3}).

Now we deal with~(\ref{explicitVDMaux4}). We similarly introduce
\begin{eqnarray*}
R_N(z,w)
\,=\,
z^{n-m-1}w^{m+1}-L^{(N)}\left[{X}^{n-m-1}{Y}^{m+1}\right](z,w)
-
\prod_{j=0}^{n-m-2}
\left(z-\eta_j\right)
\prod_{i=0}^{m}
\left(w-\theta_i\right)
\,.
\end{eqnarray*}
First, we claim that $R_N$ belongs to $\mathcal{P}_{n,m}$. Indeed,
$\prod_{j=0}^{n-m-2}
\left(z-\eta_j\right)
\prod_{i=0}^{m}
\left(w-\theta_i\right)
=\left(z^{n-m-1}+S_{n-m-2}\right)
\left(w^{m+1}+T_m\right)$,
where $S_{n-m-2}\in\C[z]$ with $\deg S_{n-m-2}\leq n-m-2$,
and $T_m\in\C[w]$ with $\deg T_m\leq m$. Then
\begin{eqnarray*}
z^{n-m-1}w^{m+1}
-
\prod_{j=0}^{n-m-2}
\left(z-\eta_j\right)
\prod_{i=0}^{m}
\left(w-\theta_i\right)
& = &
\;\;\;\;\;\;\;\;\;\;\;\;\;\;\;\;\;\;\;\;\;\;\;\;\;\;\;\;\;\;\;\;\;\;\;\;\;\;\;\;\;\;\;\;
\;\;\;\;\;\;\;\;\;\;\;\;
\end{eqnarray*}
\begin{eqnarray*}
& = &
z^{n-m-1}w^{m+1}
-
\left(z^{n-m-1}+S_{n-m-2}\right)
\left(w^{m+1}+T_m\right)
\\
& = &
-z^{n-m-1}T_m
-w^{m+1}S_{n-m-2}
-S_{n-m-2}T_m
\;\in\;\mathcal{P}_{n-1}\;\subset\;\mathcal{P}_{n,m}
\,,
\end{eqnarray*}
since each term has total degree at most $n-1$. On the other hand,
$L^{(N)}\left[{X}^{n-m-1}{Y}^{m+1}\right]\in\mathcal{P}_{n,m}$
then the claim is proved.

Next, let be $\left(\eta_k,\theta_l\right)\in\Omega_{N}$. If
$k\leq n-m-2$ (in case $m\leq n-2$) or $l\leq m$, we have
\begin{eqnarray*}
R_N\left(\eta_k,\theta_l\right)
& = &
\eta_k^{n-m-1}\theta_l^{m+1}
-
L^{(N)}\left[{X}^{n-m-1}{Y}^{m+1}\right]\left(\eta_k,\theta_l\right)
-
\prod_{j=0}^{n-m-2}
\left(\eta_k-\eta_j\right)
\prod_{i=0}^{m}
\left(\theta_l-\theta_i\right)
\\
& = &
\eta_k^{n-m-1}\theta_l^{m+1}
-
\eta_k^{n-m-1}\theta_l^{m+1}
-0
\;=\;
0
\,.
\end{eqnarray*}
Otherwise, $k\geq n-m-1$ and $l\geq m+1$, then
$k+l\geq n$. Necessarily, $k+l=n$ thus $l\leq m$
(since $\left(\eta_k,\theta_l\right)\in\Omega_{N}$).
It follows that $l\leq m<m+1\leq l$, and this is impossible.

Hence $R_N$ vanishes on the set $\Omega_{N}$ that is unisolvent
for the space $\mathcal{P}_{n,m}\ni R_N$. An application of Corollary~\ref{unisolvent}
yields $R_N\equiv0$ and this proves~(\ref{explicitVDMaux4}).

\end{proof}

As a first application, we prove Corollary~\ref{redemschifsic} given in the Introduction, that is the formula
for the two-variable Vandermonde determinant of Schiffer and Siciak.

\begin{proof}

The proof is by induction on $n\geq0$. If $n=0$, the assertion is obvious since
$VDM\left(\eta_0,\theta_0\right)=1=\prod_{\emptyset}$ .

Now if $n\geq0$, successive applications of the second assertion of Lemma~\ref{explicitVDM} (with
$m=n,n-1,\ldots,0$ and $(z,w)=\left(\eta_{n-m},\theta_{m+1}\right)$), gives that
\begin{eqnarray*}
VDM
\left(
\left(\eta_0,\theta_0\right),\ldots,\left(\eta_0,\theta_n\right),
\left(\eta_{n+1},\theta_0\right),\ldots,\left(\eta_0,\theta_{n+1}\right)
\right)
& = &
\;\;\;\;\;\;\;\;\;\;\;\;\;\;\;\;\;\;\;\;\;\;\;\;\;\;\;\;\;\;\;\;\;\;\;\;
\end{eqnarray*}
\begin{eqnarray*}
& = &
\prod_{i=0}^{n}\left(\theta_{n+1}-\theta_i\right)
\times
VDM
\left(
\left(\eta_0,\theta_0\right),\ldots,\left(\eta_0,\theta_n\right),
\left(\eta_{n+1},\theta_0\right),\ldots,\left(\eta_{1},\theta_n\right)
\right)
\\
& = &
\prod_{i=0}^{n}\left(\theta_{n+1}-\theta_i\right)
\left(\eta_1-\eta_0\right)
\prod_{i=0}^{n-1}\left(\theta_n-\theta_i\right)
VDM
\left(
\left(\eta_0,\theta_0\right),\ldots,\left(\eta_0,\theta_n\right),
\left(\eta_{n+1},\theta_0\right),\ldots,\left(\eta_{2},\theta_{n-1}\right)
\right)
\\
&  \vdots & 
\\
& = &
\prod_{m=0}^n
\left[
\prod_{j=0}^{n-m-1}\left(\eta_{n-m}-\eta_j\right)
\,
\prod_{i=0}^{m}\left(\theta_{m+1}-\theta_i\right)
\right]
\times
VDM
\left(
\left(\eta_0,\theta_0\right),\ldots,\left(\eta_0,\theta_n\right),\left(\eta_{n+1},\theta_0\right)
\right)
\\
& = &
\prod_{m=0}^n
\left[
\prod_{j=0}^{n-m-1}\left(\eta_{n-m}-\eta_j\right)
\prod_{i=0}^{m}\left(\theta_{m+1}-\theta_i\right)
\right]
\times
\prod_{j=0}^n
\left(\eta_{n+1}-\eta_j\right)
\times
VDM
\left(
\left(\eta_0,\theta_0\right),\ldots,\left(\eta_0,\theta_n\right)
\right),
\end{eqnarray*}
the last equality being an application of the first assertion of Lemma~\ref{explicitVDM}
(with $(z,w)=\left(\eta_{n+1},\theta_0\right)$).
It follows that
\begin{eqnarray*}
VDM
\left(
\left(\eta_0,\theta_0\right),\ldots,\left(\eta_0,\theta_n\right),
\left(\eta_{n+1},\theta_0\right),\ldots,\left(\eta_0,\theta_{n+1}\right)
\right)
& = &
\;\;\;\;\;\;\;\;\;\;\;\;\;\;\;\;\;\;\;\;\;\;\;\;\;\;\;\;\;\;\;\;
\end{eqnarray*}
\begin{eqnarray*}
& = &
\prod_{m=0}^n
\left[
\prod_{j=0}^{m}\left(\eta_{m+1}-\eta_j\right)
\,\times\,
\prod_{i=0}^{m}\left(\theta_{m+1}-\theta_i\right)
\right]
\times
VDM
\left(
\left(\eta_0,\theta_0\right),\ldots,\left(\eta_0,\theta_n\right)
\right)
\,,
\end{eqnarray*}
because
$\prod_{m=0}^n\prod_{j=0}^{n-m-1}\left(\eta_{n-m}-\eta_j\right)=
\prod_{m=0}^{n-1}\prod_{j=0}^{n-m-1}\left(\eta_{n-m}-\eta_j\right)=\\
\prod_{m=0}^{n-1}\prod_{j=0}^m\left(\eta_{m+1}-\eta_j\right)$.
The induction is achieved since
\begin{eqnarray*}
\prod_{m=0}^n
\left[
\prod_{j=0}^{m}\left(\eta_{m+1}-\eta_j\right)
\right]
\;=\;
\prod_{0\leq i<j\leq n+1}
\left(\eta_j-\eta_i\right)
\;=\;
VDM\left(\eta_0,\ldots,\eta_{n+1}\right)
\end{eqnarray*}
(similarly, 
$\prod_{m=0}^n
\left[\,\prod_{i=0}^{m}\left(\theta_{m+1}-\theta_i\right)\right]
=VDM\left(\theta_0,\ldots,\theta_{n+1}\right)$).

\end{proof}

Another application of the above lemma is the proof of Theorem~\ref{explicitlejasequ}.

\begin{proof}

The proof is by induction on $N\geq1$. If $N=1$, we have that
$\left|VDM(z,w)\right|=|1|=1$, then any point
$(z,w)\in\C^2$ reaches the maximum (in particular $\left(\eta_0,\theta_0\right))$.

Now let be $N\geq1$ and assume that the first $N$ points $H_j$'s of the intertwining of the Leja sequences
$\left(\eta_j\right)_{j\geq0}$ and $\left(\theta_i\right)_{i\geq0}$, are an
$N$-Leja section for the compact set $K_1\times K_2$. The goal is to prove that the $(N+1)$st
point $H_{N+1}$ of the intertwining sequence reaches
$\sup_{(z,w)\in K_1\times K_2}\left|VDM\left(H_1,\ldots,H_N,(z,w)\right)\right|$.

If $N=N_n$ for some $n\geq0$, then an application of the first assertion of Lemma~\ref{explicitVDM}
gives that
\begin{eqnarray*}
\sup_{(z,w)\in K_1\times K_2}
\left|
VDM
\left(
H_1,\ldots,H_{N_n},(z,w)
\right)
\right|
& = &
\;\;\;\;\;\;\;\;\;\;\;\;\;\;\;\;\;\;\;\;\;\;\;\;\;\;\;\;\;\;\;\;\;\;\;\;\;\;\;\;
\;\;\;\;\;\;\;\;\;\;\;\;\;\;\;\;\;\;\;\;\;\;\;\;\;\;\;\;\;\;\;\;\;\;\;\;\;\;\;\;
\end{eqnarray*}
\begin{eqnarray*}
& = &
\sup_{(z,w)\in K_1\times K_2}
\left|
VDM
\left(
\left(\eta_0,\theta_0\right),\ldots,\left(\eta_0,\theta_{n-1}\right),
\left(\eta_{n},\theta_0\right),\ldots,\left(\eta_0,\theta_n\right),(z,w)
\right)
\right|
\\
& = &
\sup_{(z,w)\in K_1\times K_2}
\left|
\prod_{j=0}^n
\left(z-\eta_j\right)
\times
VDM
\left(
\left(\eta_0,\theta_0\right),\ldots,\left(\eta_0,\theta_{n-1}\right),
\left(\eta_{n},\theta_0\right),\ldots,\left(\eta_0,\theta_n\right)
\right)
\right|
\\
& = &
\left(
\sup_{z\in K_1}
\prod_{j=0}^n
\left|z-\eta_j\right|
\right)
\times
\left|
VDM
\left(
\left(\eta_0,\theta_0\right),\ldots,\left(\eta_0,\theta_{n-1}\right),
\left(\eta_{n},\theta_0\right),\ldots,\left(\eta_0,\theta_n\right)
\right)
\right|
\,.
\end{eqnarray*}
The maximum is reached on any point $\left(\eta_{n+1},w\right)$ with $w\in K_2$
(since $\left(\eta_j\right)_{j\geq0}$ is a Leja sequence for $K_1$), then in particular
on the $(N+1)$st point $H_{N_n+1}=\left(\eta_{n+1},\theta_0\right)$ of the intertwining sequence.

Otherwise, $N_{n-1}<N< N_n$ with $n\geq1$. An application of the second assertion of
Lemma~\ref{explicitVDM} gives that
\begin{eqnarray*}
\sup_{(z,w)\in K_1\times K_2}
\left|
VDM
\left(
H_1,\ldots,H_{N},(z,w)
\right)
\right|
& = &
\;\;\;\;\;\;\;\;\;\;\;\;\;\;\;\;\;\;\;\;\;\;\;\;\;\;\;\;\;\;\;\;\;\;\;\;\;\;\;\;
\;\;\;\;\;\;\;\;\;\;\;\;\;\;\;\;\;\;\;\;\;\;\;\;\;\;\;\;\;\;\;\;\;\;\;\;\;\;\;\;
\end{eqnarray*}
\begin{eqnarray*}
& = &
\sup_{(z,w)\in K_1\times K_2}
\left|
VDM
\left(
\left(\eta_0,\theta_0\right),\ldots,\left(\eta_0,\theta_{n-1}\right),
\left(\eta_{n},\theta_0\right),\ldots,\left(\eta_{n-m},\theta_m\right),(z,w)
\right)
\right|
\\
& = &
\sup_{(z,w)\in K_1\times K_2}
\left[
\prod_{j=0}^{n-m-2}\left|z-\eta_j\right|
\prod_{i=0}^{m}\left|w-\theta_i\right|
\times
\left|
VDM
\left(
\left(\eta_0,\theta_0\right),\ldots,\left(\eta_{n-m},\theta_m\right)
\right)
\right|
\right]
\\
& = &
\left(
\sup_{z\in K_1}
\prod_{j=0}^{n-m-2}\left|z-\eta_j\right|
\right)
\times
\left(
\sup_{w\in K_2}
\prod_{i=0}^{m}\left|w-\theta_i\right|
\right)
\times
\left|
VDM
\left(
\left(\eta_0,\theta_0\right),\ldots,\left(\eta_{n-m},\theta_m\right)
\right)
\right|,
\end{eqnarray*}
whose maximum is reached on the point $\left(\eta_{n-m-1},\theta_{m+1}\right)$
(since $\left(\eta_j\right)_{j\geq0}$ and $\left(\theta_i\right)_{i\geq0}$
are both Leja sequences for $K_1$ and $K_2$ respectively), that is exactly the
$(N+1)$st point $H_{N+1}$ of the intertwining sequence. This completes the induction
and the proof of the theorem.

\end{proof}

\begin{remark}

We can notice in the above proof that the choice for the $(N+1)$st Leja point is not necessarily unique.
For $N=N_n$, we could choose any point
$\left(z_{N+1},w_{N+1}\right)$ with $w_{N+1}\in K_2$
and such that 
$\left(\eta_0,\ldots,\eta_n,z_{N+1}\right)$ is an $(n+2)$-Leja section.
For $N_{n-1}<N< N_n$, we could similarly choose any point $\left(z_{N+1},w_{N+1}\right)$
such that $\left(\eta_0,\ldots,\eta_{n-m-2},z_{N+1}\right)$
(resp., $\left(\theta_0,\ldots,\theta_m,w_{N+1}\right)$)
is an $(n-m)$-Leja section 
(resp., an $(m+2)$-Leja section).

But the choice of $\left(\eta_{n+1},\theta_0\right)$ (resp., $\left(\eta_{n-m-1},\theta_{m+1}\right)$)
for the $(N+1)$st Leja point for $N=N_n$ (resp., $N_{n-1}<N< N_n$), is essential
so that $\left(H_1,\ldots,H_N,H_{N+1}\right)$ are the first $(N+1)$ points of an
intertwining sequence (in order to apply Lemma~\ref{explicitVDM}
for the determination of the following Leja points). Otherwise, searching the $(N+2)$nd Leja point would become
a hard task (and {\em a fortiori} the next ones).

\end{remark}

\end{document}